\documentclass{article}%
\usepackage{amsmath}%
\usepackage{amsfonts}%
\usepackage{amssymb}%
\usepackage{graphicx}
\usepackage{hyperref}
\newtheorem{theorem}{Theorem}

\newtheorem{definition}[theorem]{Definition}

\newtheorem{lemma}[theorem]{Lemma}

\newtheorem{remark}[theorem]{Remark}

\newenvironment{proof}[1][Proof]{\noindent\textbf{#1.} }{\ \rule{0.5em}{0.5em}}

\title{Beyond the Melnikov method: a computer assisted approach}

\author{Maciej J. Capi\'nski \thanks{AGH University of Science and Technology, Faculty of Applied Mathematics, al. Mickiewicza 30, 30-059 Krak\'ow, Poland. Email: {\tt maciej.capinski@agh.edu.pl}.
}
\and
Piotr Zgliczy\'nski \thanks{Jagiellonian University, Faculty of Mathematics and Computer Science, ul. prof. Stanis\l awa \L ojasiewicza 6,
30-348 Krak\'ow, Poland. Email: {\tt umzglicz@cyf-kr.edu.pl}}
}

\begin{document}
\maketitle

\begin{abstract}

We present a Melnikov type approach for establishing transversal intersections of stable/unstable manifolds of perturbed normally hyperbolic invariant manifolds (NHIMs). The method is based on a new geometric proof of the normally hyperbolic invariant manifold theorem, which establishes the existence of a NHIM, together with its associated invariant manifolds and bounds on their first and second derivatives.
 We do not need to know the explicit formulas for the homoclinic orbits prior to the perturbation. We also do not need to compute any integrals along such homoclinics. All needed bounds are established using rigorous computer assisted numerics. Lastly, and most importantly, the method establishes intersections for an explicit range of parameters, and not only for perturbations that are `small enough', as is the case in the classical Melnikov approach.
\end{abstract}


\textbf{AMS classification numbers:} 37D10, 58F15 ,65G20


\section{Introduction}

The presence of the transversal intersection between stable and unstable
manifolds for fixed point or periodic orbit is one of the main technical
tools used to prove the chaotic behavior of the deterministic dynamical
system (see for example \cite{GH} and the literature given there). In the
context of the small perturbations of an integrable system the basic
analytical technique used to establish the transversality is the Melnikov
method \cite{M} introduced in 1963. V.I. Arnold generalized these ideas to
produce the first example of what is now called Arnold Diffusion \cite%
{Arnold-diff64}. In fact the (now widely-used) Melnikov function (see for
example \cite{Wi,HM}) is, up to a constant, exactly the integral that Poincar%
\'{e} derived from Hamilton-Jacobi theory to obtain his obstruction to
integrability in the restricted three body problem in \cite{Poincare1890}. 

Melnikov type methods are based on investigating integrals along homoclinic
orbits to normally hyperbolic invariant manifolds (NHIMs) \cite%
{Chow,DRR,DLS1, HM, M, Wi}. There are natural problems with such approach:
It is very rarely the case that one can establish analytic formulae for such
homoclinics. In most cases they are not known, and then computing integrals
along them is impossible. The second problem is that even if one has an
analytic formula for the homoclinic, the integral in question can be very
hard to compute. In most real life systems such integrals would not be
expressed through simple formulas.

We resolve these two problems the following way. Firstly, we investigate the
dependence of the manifolds on the parameter using geometric and computer
assisted tools. The slopes of the manifolds depending on the parameter
follow from cone condition type bounds in the state space extended by the
parameters. Second order derivatives also follow from geometric structures.
This way we obtain bounds on the stable and unstable manifolds of NHIMs,
together with their dependence up to second order on the perturbation
parameter. We then propagate these bounds using rigorous (interval based)
integration up to a section where they meet. Based on the bounds, and in
particular using the dependence on the manifolds on the perturbation
parameter, we establish transversal intersections for a given, explicit,
range of perturbations. The range is large enough so that for the larger
parameters from the range we can detect the transversal intersections
directly, and continue to higher perturbations using standard techniques.

Our contribution to the existing theory is twofold:

Firstly, in this paper we develop a method for establishing centre unstable
manifolds of NHIMs, in the context of ordinary differential equations. The
main benefit from our approach is that we do not need to assume that the
NHIM exists in order to apply our method. (Our method is constructive, not
perturbative.) We formulate assumptions, which guarantee the existence of a
center-unstable manifold within an investigated neighborhood. The
assumptions of our theorem depend only on the bounds on the first derivative
of the vector field. These guarantee that the center-unstable manifold
exists, and is a graph of a function within the investigated region. The
method gives explicit bounds on the slope of the manifold. Moreover, by
considering bounds on the second derivative of the vector field, we obtain
explicit estimates on the second order derivatives of the center-unstable
manifold. By changing the sign of the vector field, the method establishes
existence of center-stable manifolds. By intersecting the center-stable
manifold with the centre-unstable manifold we establish the existence of a
NHIM within the investigated region. Our method also establishes bounds on
the first and second order dependence of the manifolds on the parameters for
families of ODEs. Summing up: the method is explicit, establishes existence
of the manifolds over a specified, macroscopic domain, all assumptions can
be verified from simple estimates on the first and second derivative of the
vector field, and gives explicit estimates on the dependence of the
manifolds on parameters.

Our second contribution is developing a Melnikov-type theory for
establishing transversal intersections of stable/unstable manifolds of
NHIMs. The method is based on interval arithmetic integration of ODEs and
propagation of local bounds on the manifolds up to the point of their
intersection. The benefit from our approach is the following. We do not have
to know any analytic formulae for the homoclinics. They are established
using rigorous computer assisted numerics. Secondly, we do not need to
compute any integrals. All bounds on the manifolds are propagated by our
integrator in form of rigorous, interval arithmetic bounds for the jets.
This method allows us to establish intersections of the manifolds for
specific ranges of parameters. These ranges are large enough to later
continue the proof of the intersections of the manifolds using standard
continuation arguments.

We emphasize that, to the best of our knowledge, this is the first computer
assisted Melnikov type method, which works over an explicit parameter range.
Since our method does not rely on analytic computations along homoclinics,
we believe that our approach is very versatile and can be applied to
numerous problems that are not accessible to the standard methods.

The paper is organized as follows. We first address the problem how to
establish transversal intersections of manifolds for given ranges of
perturbation parameters, under the assumption that we have bounds on the
first and second derivatives of the stable/unstable manifolds of NHIMs. This
problem is introduced below in subsection \ref{sec:intro-setup}, and the
main idea behind our approach is explained in subsection \ref{sec:main-idea}%
. We then follow up with full details in section \ref{sec:Meln}, where the
formulation is made precise and the main results are proven. Secondly, we
address how to establish the needed bounds for the derivatives of
stable/unstable manifolds. In section \ref{sec:wcu-maps} we recall the
results from \cite{Geom}, where such bounds are established in the setting
of discrete dynamical systems. In section \ref{sec:wcu-maps} we also extend
the method to obtain explicit bounds on second derivatives of the manifolds.
In section \ref{sec:ODE-Wcu} we further extend the results from section \ref%
{sec:wcu-maps} to the setting of ODEs. We make sure that the needed
assumptions follow from the bounds on the vector field, so that we do not
have to integrate the ODEs. As the by-product we obtain also a
generalization of results from \cite{Geom} about establishing of NHIM for
ODEs. In section \ref{sec:num-example} we give an example of application of
our method.

An alternative to \cite{Geom} and its extension presented in this paper for
obtaining bounds on derivatives of stable/unstable manifolds of NHIMs, is
the parameterization method \cite{param-method}. This method is suitable for
application to computer assisted proofs. (For examples of such applications
see \cite{CF, CLJ, FH, LJR, JM}, amongst others.) We believe that our
approach to Melnikov method (from sections \ref{sec:intro-setup}, \ref%
{sec:main-idea}, and \ref{sec:Meln}) could also be successfully combined
with \cite{param-method}. We decide to use the geometric method \cite{Geom}
and its generalization to ODEs, since it does not require high order
expansions in order to establish existence of the manifolds, but follows
from direct estimates on first and second derivatives of the vector field.

In the two subsections that follow we specify the setup under which our
paper is written and outline the main idea.

\subsection{The setup \label{sec:intro-setup}}

In this section we formulate our main goals and set up the notation. The
problem is formulated in the simplest possible setting. We consider a
non-autonomous perturbation of an autonomous ODE on the plane. This enables
us to present the main features. Our method though can be applied in a much
more general setting.

We consider a vector field 
\begin{equation*}
f_{0}:\mathbb{R}^{2}\rightarrow \mathbb{R}^{2},
\end{equation*}%
and a function 
\begin{equation*}
g:\mathbb{R}^{2}\times \mathbb{R}\times \mathbb{R}\rightarrow \mathbb{R}^{2}.
\end{equation*}%
We assume that $g$ is $2\pi $ periodic in the last coordinate. We shall
consider the following family of time periodic ODEs 
\begin{equation}
\left( x,y\right) ^{\prime }=f_{\varepsilon }\left( x,y,t\right)
=f_{0}\left( x,y\right) +g\left( x,y,\varepsilon ,t\right) .
\label{eq:ode-perturbation}
\end{equation}%
We assume that for $\varepsilon =0$ holds $g\left( x,y,0,t\right) =0$. This
means that we treat $g$ as a perturbation, with $\varepsilon $ being the
perturbation parameter.

We shall assume that for $\varepsilon=0,$ (\ref{eq:ode-perturbation}) has a
hyperbolic fixed point $\left( x_{0},y_{0}\right) $ and that we have a
homoclinic orbit along the stable/unstable manifold of $\left(
x_{0},y_{0}\right) $.

Since the fixed point is hyperbolic, for $\varepsilon \neq 0$ it will be
perturbed to a $2\pi $ periodic hyperbolic orbit. We shall use a notation $%
\gamma _{\varepsilon }\left( t\right) $ for this orbit and assume that such
orbits exist for $\varepsilon \in E$, where $E\subset \mathbb{R}$ is a
closed interval around zero.

In order to investigate the intersections of the stable/unstable manifolds
of $\gamma _{\varepsilon }$ we consider a section $\Sigma $, which is
transversal to the homoclinic orbit (which exists for $\varepsilon =0$). For 
$\varepsilon \geq 0$, the stable manifold of $\gamma _{\varepsilon }$ for
the problem (\ref{eq:ode-perturbation}), with initial condition starting at
time $\tau $ will hit $\Sigma $ at a point, which we denote as $p^{s}\left(
\varepsilon ,\tau \right) $. Similarly, by $p^{u}\left( \varepsilon ,\tau
\right) $ we denote the point of intersection of the unstable manifold with $%
\Sigma $.

\begin{remark}
The construction of the points $p^{s}\left( \varepsilon ,\tau \right) $, $%
p^{u}\left( \varepsilon ,\tau \right) $ is made precise and carried out in
full detail in section \ref{sec:method}. Figure \ref{fig:WuWs} gives a
graphical illustration of the setup.
\end{remark}

We then define the (signed) distance $\delta $ between the two manifolds on $%
\Sigma $ as%
\begin{equation}
\delta \left( \varepsilon ,\tau \right) :=p^{u}\left( \varepsilon ,\tau
\right) -p^{s}\left( \varepsilon ,\tau \right) .  \label{eq:delta-setup-def}
\end{equation}%
The main question is to establish conditions on $\delta \left( \varepsilon
,\tau \right) $ that ensure that the stable/unstable manifolds of such
orbits intersect transversally, for all $\varepsilon \in E\setminus \{0\}$.

The above setting, in which we are perturbing a fixed point, is the simplest
one. In general we could be interested in intersections of stable/unstable
manifolds of perturbed NHIMs. The tools for establishing such manifolds and
their perturbations, together with all the ingredients needed to apply our
method are developed in \cite{Geom} and its generalization to ODEs form
section \ref{sec:ODE-Wcu}. There are no obstacles to generalizing to such
setting. We restrict ourselves though to the simplest case for the sake of
clarity of exposition and postpone detailed treatment of the general case
for NHIMs for later publication.

\subsection{The main idea in simplest terms \label{sec:main-idea}}

We consider a $C^{2}$ function $\delta :\mathbb{R\times S}^{1}\rightarrow 
\mathbb{R}$, which is defined in (\ref{eq:delta-setup-def}). Since for $%
\varepsilon =0$ the stable and unstable manifolds coincide forming a
homoclinic orbit, we know that%
\begin{equation*}
\delta \left( 0,\tau \right) =0\qquad \text{for all }\tau \in \mathbb{S}.
\end{equation*}%
For fixed $\varepsilon \in \mathbb{R}$ we will use the notation%
\begin{equation*}
\delta _{\varepsilon }\left( \tau \right) :=\delta \left( \varepsilon ,\tau
\right) .
\end{equation*}%
Let $E$ be a closed interval in $\mathbb{R}$, which contains zero. Our aim
is to give a simple set of assumptions that will lead to a conclusion that
for any $\varepsilon \in E\setminus \{0\}$ the function $\delta
_{\varepsilon }$ will have nontrivial zeroes. In other words, that there
exists a $\xi \left( \varepsilon \right) \in \mathbb{S}^{1}$ such that%
\begin{equation*}
\delta _{\varepsilon }\left( \xi \left( \varepsilon \right) \right)
=0,\qquad \frac{d\delta _{\varepsilon }}{d\tau }\left( \xi \left(
\varepsilon \right) \right) \neq 0.
\end{equation*}

For $A\subset \mathbb{R\times S}^{1}$ we shall write%
\begin{equation*}
\left[ \frac{\partial \delta }{\partial \varepsilon }\left( A\right) \right]
:=\left[ \inf_{\left( \varepsilon ,\tau \right) \in A}\frac{\partial \delta 
}{\partial \varepsilon }\left( \varepsilon ,\tau \right) ,\sup_{\left(
\varepsilon ,\tau \right) \in A}\frac{\partial \delta }{\partial \varepsilon 
}\left( \varepsilon ,\tau \right) \right] .
\end{equation*}%
Our idea is based on the fact that for any $\varepsilon \in E\setminus \{0\}$%
\begin{equation}
\frac{\delta _{\varepsilon }\left( \tau \right) }{\varepsilon }\in \left[ 
\frac{\partial \delta }{\partial \varepsilon }\left( E\times \left\{ \tau
\right\} \right) \right] .  \label{eq:mean-eps}
\end{equation}%
This means that if we can establish that for some $\tau _{1},\tau _{2}\in 
\mathbb{S}^{1}$%
\begin{equation}
\left[ \frac{\partial \delta }{\partial \varepsilon }\left( E\times \left\{
\tau _{1}\right\} \right) \right] <0<\left[ \frac{\partial \delta }{\partial
\varepsilon }\left( E\times \left\{ \tau _{2}\right\} \right) \right] ,
\label{eq:Bolzano-assmpt}
\end{equation}%
then for any $\varepsilon \in E\setminus \{0\}$, by (\ref{eq:mean-eps}) and (%
\ref{eq:Bolzano-assmpt}), 
\begin{equation*}
\frac{\delta _{\varepsilon }\left( \tau _{1}\right) }{\varepsilon }<0<\frac{%
\delta _{\varepsilon }\left( \tau _{2}\right) }{\varepsilon }.
\end{equation*}%
Hence, by the Bolzano theorem, there exists a $\xi \left( \varepsilon
\right) \in \left[ \tau _{1},\tau _{2}\right] $, such that%
\begin{equation*}
\delta _{\varepsilon }\left( \xi \left( \varepsilon \right) \right) =0.
\end{equation*}

If in addition to (\ref{eq:Bolzano-assmpt}) we have that%
\begin{equation}
\left[ \frac{\partial ^{2}\delta }{\partial \tau \partial \varepsilon }%
\left( E\times \left[ \tau _{1},\tau _{2}\right] \right) \right] >0,
\label{eq:der-assmpt}
\end{equation}%
then for any $\tau \in \left[ \tau _{1},\tau _{2}\right] $%
\begin{equation*}
\frac{d}{d\tau }\left( \frac{\delta _{\varepsilon }\left( \tau \right) }{%
\varepsilon }\right) \in \left[ \frac{\partial ^{2}\delta }{\partial \tau
\partial \varepsilon }\left( E\times \left[ \tau _{1},\tau _{2}\right]
\right) \right] >0.
\end{equation*}%
Thus, for each $\varepsilon \in E\setminus \{0\},$ the point $\xi \left(
\varepsilon \right) $ is unique and 
\begin{equation*}
\frac{d\delta _{\varepsilon }}{d\tau }\left( \xi \left( \varepsilon \right)
\right) \neq 0.
\end{equation*}

To sum up the above discussion, in order to show that for any $\varepsilon
\in E\setminus \{0\}$ we have nontrivial zeros of the function $\delta
_{\varepsilon }$, it is sufficient to verify (\ref{eq:Bolzano-assmpt}) and (%
\ref{eq:der-assmpt}). We emphasize that in this approach we have an explicit
range $E$ of $\varepsilon $ for which the nontrivial zeros exist.

Summing up, to compute the Melnikov distance $\delta $, our method combines
two ingredients, both computer assisted:

\begin{itemize}
\item the geometric method to establish explicit bounds for normally
hyperbolic invariant manifolds and their stable and unstable fibers,
together with their dependence on parameter.

\item the rigorous $C^2$-integration of our system away from the NHIM.
\end{itemize}

\noindent This method can be generalized to many dimensions.


\section{Preliminaries}

\subsection{Notations and conventions}

We will use the Euclidian norm unless stated otherwise. For two vectors $%
z_{1},z_{2}$ we denote their scalar product by $(z_{1}|z_{2})$. For a matrix 
$A$, by $A^{\top }$ we denote the transpose of $A$. By $I$ we will denote
the identity matrix, the dimension will be known from the context.

For a set $A$, we shall use $A^c$ to denote its complement.

For a function $f:\mathbb{R}^{n}\rightarrow \mathbb{R}^{m}$ for $%
z_{1},z_{2}\in \mathbb{R}^{n}$ we define an average of $f$ on the segment $%
[z_{1},z_{2}]$ by 
\begin{equation*}
\overline{f}(z_{1},z_{2})=\int_{0}^{1}f(z_{1}+s(z_{2}-z_{1}))ds.
\end{equation*}%
Observe that we have the following equality for $f\in C^{1}$: 
\begin{equation*}
f(z_{2})-f(z_{1})=\overline{Df}(z_{1},z_{2})(z_{2}-z_{1}).
\end{equation*}

\subsection{Logarithmic norms and related topics}

In this section we state some facts about logarithmic norms \cite{D,L,HNW,KZ}
and some analogous notions. These are later used in section \ref{sec:ODE-Wcu}%
. Since the results are of technical nature, we give their proofs in the
appendix.

\begin{definition}
For a square matrix $A\in \mathbb{R}^{n\times n}$ we define $m(A)$ by 
\begin{equation*}
m(A)=\min_{z\in \mathbb{R}^{n},\Vert z\Vert =1}\Vert Az\Vert ,
\end{equation*}%
the logarithmic norm of $A$ denoted by $l(A)$ by \cite{L,D,HNW,KZ} 
\begin{equation}
l(A)=\lim_{h\rightarrow 0^{+}}\frac{\Vert I+hA\Vert -\Vert I\Vert }{h}
\label{eq:def-log-norm}
\end{equation}%
and the logarithmic minimum of $A$ 
\begin{equation}
m_{l}(A)=\lim_{h\rightarrow 0^{+}}\frac{m(I+hA)-\Vert I\Vert }{h}.
\label{eq:def-ml}
\end{equation}
\end{definition}

It is easy to see that if $A$ is invertible, then 
\begin{equation*}
m(A)= \frac{1}{\|A^{-1}\|},
\end{equation*}
otherwise $m(A)=0$.

It is known that $l(A)$ is a convex function.

\begin{lemma}
\label{lem:def-ml-ok} The limit in the definition of $m_{l}(A)$ exists and 
\begin{equation}
m_{l}(A)=-l(-A).  \label{eq:ml-min-l}
\end{equation}%
Moreover, the convergence to this limit is locally uniform with respect to $%
A $ and $m_{l}(A)$ is a concave function.
\end{lemma}

\begin{proof}
See Appendix~\ref{app:ml}.
\end{proof}

Below theorem establishes a bound on distances of solutions of an ODE in
terms of the logarithmic norm. The proof of this result can be found in \cite%
{HNW}.

\begin{theorem}
\label{th:log-norm-ode}Consider an ODE 
\begin{equation}
x^{\prime }=f(t,x),  \label{eq:non-auto-ode}
\end{equation}%
where $x\in \mathbb{R}^{n}$ and $f:\mathbb{R}\times \mathbb{R}%
^{n}\rightarrow \mathbb{R}^{n}$ is $C^{1}$.

Let $x(t)$ and $y(t)$ for $t\in \lbrack t_{0},t_{0}+T]$ be two solutions of (%
\ref{eq:non-auto-ode}). Let $W\subset \mathbb{R}^{n}$ such that for each $%
t\in \lbrack t_{0},t_{0}+T]$ the segment connecting $x(t)$ and $y(t)$ is
contained in $W$. Let 
\begin{equation*}
L=\sup_{x\in W,t\in \lbrack t_{0},t_{0}+T]}l\left( \frac{\partial f}{%
\partial x}(t,x)\right) .
\end{equation*}%
Then for $t\in \lbrack 0,T]$ holds 
\begin{equation*}
\left\Vert x(t_{0}+t)-y(t_{0}+t)\right\Vert \leq \exp (Lt)\left\Vert
x(t_{0})-y(t_{0})\right\Vert .
\end{equation*}
\end{theorem}

Theorem \ref{th:log-norm-ode} gives an upper bound for the distance between
solutions of an ODE. We now show a similar result, which allows us to obtain
a lower bound.

\begin{theorem}
\label{th:log-norm-ode-lower-bound}Consider an ODE 
\begin{equation}
x^{\prime }=f(x),  \label{eq:lh-ode}
\end{equation}%
where $x\in \mathbb{R}^{n}$ and $f:\mathbb{R}^{n}\rightarrow \mathbb{R}^{n}$
is $C^{1}$.

Let $x(t)$ and $y(t)$ for $t\in \lbrack 0,T]$ be two solutions of (\ref%
{eq:lh-ode}). Let $W\subset \mathbb{R}^{n}$ be such that for each $t\in
\lbrack 0,T]$ the segment connecting $x(t)$ and $y(t)$ is contained in $W$.
Let 
\begin{equation*}
m_{l}(Df,W)=\inf_{x\in W}m_{l}(Df(x)).
\end{equation*}%
Then for $t>0$ holds 
\begin{equation*}
\Vert x(t)-y(t)\Vert \geq \exp (m_{l}(Df,W)t)\Vert x(0)-y(0)\Vert .
\end{equation*}
\end{theorem}

\begin{proof}
See Appendix~\ref{app:log-norm-ode-lower-bound}.
\end{proof}

In above results the choice of norms was arbitrary. We will apply these
results in the case when the norm is Euclidean. In such case we have the
following results.

\begin{lemma}
\label{lem:eucl-log-norm} For the Euclidian norm holds 
\begin{eqnarray}
l(A) &=&\max \{\lambda \in \text{spectrum of }(A+A^{\top })/2\}
\label{eq:eucl-log-norm} \\
m_{l}(A) &=&\min \{\lambda \in \text{spectrum of }(A+A^{\top })/2\}.
\label{eq:eucl-ml}
\end{eqnarray}
\end{lemma}

\begin{proof}
The formula (\ref{eq:eucl-log-norm}) is well known \cite{HNW}. From Lemma %
\ref{lem:def-ml-ok}, $m_{l}(A)=-l(-A)$, which gives (\ref{eq:eucl-ml}).
\end{proof}

\begin{lemma}
\label{lem:lognorm-conv}Consider the Euclidean norm $\left\Vert \cdot
\right\Vert $. Assume that $A\in W$, where $W\subset \mathbb{R}^{n\times n}$
is compact. Assume also that $h\in (0,h_{0}]$. Then 
\begin{equation*}
\Vert I+hA\Vert =1+hl(A)+r(h,A),
\end{equation*}%
where 
\begin{equation*}
\Vert r(h,A)\Vert \leq Mh^{2},
\end{equation*}%
for some constant $M=M(h_{0},W)$ (the constant $M$ depends on $h_{0}$ and $W$%
).
\end{lemma}

\begin{proof}
See Appendix~\ref{sec:app-lognorm}.
\end{proof}

\begin{lemma}
\label{lem:ml-conv} Assume that $A\in W$, where $W\subset \mathbb{R}%
^{n\times n}$ compact. Assume $h\in (0,h_{0}]$. Then 
\begin{equation*}
m(I+hA)=1+hm_{l}(A)+r(h,A)
\end{equation*}%
where 
\begin{equation*}
\Vert r(h,A)\Vert \leq Mh^{2}
\end{equation*}%
for some constant $M=M(h_{0},W)$ (the constant $M$ depends on $h_{0}$ and $W$%
).
\end{lemma}

\begin{proof}
See Appendix~\ref{sec:app-ml-norm}.
\end{proof}


\section{Melnikov type method\label{sec:Meln}}

In this section we introduce a Melnikov type method. The difference with the
standard approach is that we do not integrate along the homoclinic orbit.
Instead, we assume that we have bounds on the local parameterizations of the
stable/unstable manifold of the perturbed orbit. These are then propagated
to the section where we measure the distance. This formulation allows us to
verify our assumptions for a given range of perturbations. We do not need to
assume that the perturbation is small enough.

While presenting the method, we make a number of assumptions about the
stable and unstable manifolds. Namely that we have their parameterization,
and that we have bounds on their derivatives. Based on these assumptions we
formulate our results. We emphasize straightaway that we know how to obtain
such bounds. This is the subject of subsequent sections.

In section \ref{sec:method} we present the method; in particular Theorem \ref%
{th:main}, which contains the main result. In section \ref%
{sec:assumptions-verif} we discuss how to verify the assumptions, based on
the bounds on the derivatives of the parametrizations of the stable and
unstable manifolds. How such bounds can be obtained is presented in section %
\ref{sec:ODE-Wcu}.

\subsection{The method\label{sec:method}}

\label{subsec:method} In our treatment of the problem we shall consider the
following formulation of (\ref{eq:ode-perturbation}), in a state space that
is extended to include both the time and the parameter:%
\begin{align}
x^{\prime} & =\pi_{x}f_{\varepsilon}\left( x,y,s\right) ,  \notag \\
y^{\prime} & =\pi_{y}f_{\varepsilon}\left( x,y,s\right) ,  \label{eq:our-ODE}
\\
\varepsilon^{\prime} & =0,  \notag \\
s^{\prime} & =1.  \notag
\end{align}

In the extended phase space coordinates $q=(x,y,\varepsilon ,s),$ we shall
use the notation%
\begin{equation}
q^{\prime }=f(q),  \label{eq:our-ODE-simple}
\end{equation}%
for the ODE (\ref{eq:our-ODE}), where%
\begin{equation*}
f:\mathbb{R}^{3}\times \mathbb{S}^{1}\rightarrow \mathbb{R}^{4}.
\end{equation*}%
We shall write $\Phi _{t}(q)$ for the flow of (\ref{eq:our-ODE-simple}).

We will refer to the cyclic variable $s$ as $s$-time or just a time. There
will be also other `time' occasionally appearing in our discussion, this
will be the time along the solution of the system (\ref{eq:our-ODE-simple}),
we will refer this variable as $t$-time. Given two points on the trajectory
of (\ref{eq:our-ODE-simple}) the $t$-time between them will be the
difference between $s$-times of these two points.

The family of periodic orbits $\gamma_{\varepsilon}(s) $ forms a two
dimensional invariant manifold (with a boundary) for (\ref{eq:our-ODE-simple}%
):%
\begin{equation*}
\Lambda=\left\{ \left( \gamma_{\varepsilon}\left( s\right) ,\varepsilon
,s\right) :\varepsilon\in E,s\in\mathbb{S}^{1}\right\} .
\end{equation*}
(The boundary of $\Lambda$ is $\partial\Lambda=\left\{ \left( \gamma
_{\varepsilon}\left( s\right) ,\varepsilon,s\right) :\varepsilon\in\partial
E,s\in\mathbb{S}^{1}\right\} .$)

For any fixed $\varepsilon\in E$, we shall write 
\begin{equation*}
\Lambda_{\varepsilon}=\left\{ \left( \gamma_{\varepsilon}\left( s\right)
,e,s\right) :e=\varepsilon,s\in\mathbb{S}^{1}\right\} ,
\end{equation*}
to denote the invariant set containing the periodic orbit of (\ref%
{eq:ode-perturbation}), in the extended phase space.

Let $N$ be a set in $\mathbb{R}^{3}\times \mathbb{S}^{1}$ containing $%
\Lambda $. We shall use $W_{\mathrm{loc}}^{s}\left( \Lambda \right) $ and $%
W_{\mathrm{loc}}^{u}\left( \Lambda \right) $ to denote the local stable and
unstable manifolds in $N$, respectively i.e.%
\begin{eqnarray*}
W_{\mathrm{loc}}^{s}\left( \Lambda \right)  &=&\left\{ q:\Phi _{t}(q)\in N%
\text{ for all }t\geq 0\right\} , \\
W_{\mathrm{loc}}^{u}\left( \Lambda \right)  &=&\left\{ q:\Phi _{t}(q)\in N%
\text{ for all }t\leq 0\right\} .
\end{eqnarray*}%
(Since the set $N$ will be fixed, we do not include it in our notations for
the local manifolds.) We assume that in the neighborhood $N$ we can
parameterize $W_{\mathrm{loc}}^{u}\left( \Lambda \right) $ by a function%
\begin{equation*}
w^{u}:\left[ -r_{u},r_{u}\right] \times E\times \mathbb{S}^{1}\rightarrow 
\mathbb{R}^{3}\times \mathbb{S}^{1},
\end{equation*}%
where $r_{u}\in \mathbb{R}^{+}$. We also assume that $W_{\mathrm{loc}%
}^{s}\left( \Lambda \right) $ is parameterized by%
\begin{equation*}
w^{s}:\left[ -r_{s},r_{s}\right] \times E\times \mathbb{S}^{1}\rightarrow 
\mathbb{R}^{3}\times \mathbb{S}^{1},
\end{equation*}%
for $r_{s}\in \mathbb{R}^{+}$. We assume that our parameterizations satisfy%
\begin{equation}
\pi _{\varepsilon ,s}w^{\iota }\left( r,\varepsilon ,s\right) =\left(
\varepsilon ,s\right) ,\qquad \text{for }\iota \in \left\{ s,u\right\} .
\label{eq:wi-proj}
\end{equation}%
We shall use notations $W^{u}\left( \Lambda \right) ,$ $W^{s}(\Lambda )$ for
the unstable and stable manifolds of $\Lambda $, respectively.

The existence of the manifolds within the set $N$, together with the fact
that they are graphs of the functions $w^u$ and $w^s$, will follow from our
construction. Namely, in sections \ref{sec:wcu-maps} and \ref{sec:ODE-Wcu}
we present a detailed method which ensures, using constructive arguments,
that above assumptions are fulfilled within an explicitly given set $N$.

Let $\Sigma \subset R^{3}\times S^{1}$ be a $3$-dimensional section for (\ref%
{eq:our-ODE-simple}), such that for any $q\in w^{u}\left( (0,r_{u}]\times
E\times \mathbb{S}^{1}\right) $ the first intersection for time $t>0$ of the
trajectory $\Phi _{t}\left( q\right) $ with $\Sigma $ is transversal. We
also assume that for any $q\in w^{s}\left( (0,r_{s}]\times E\times \mathbb{S}%
^{1}\right) $ the first intersection for time $t<0$ of the trajectory $\Phi
_{t}\left( q\right) $ with $\Sigma $ is transversal. For simplicity, without
loss of generality, we shall assume that $\Sigma =\left\{ y=0\right\} $,
hence the coordinates on $\Sigma $ are $\left( x,\varepsilon ,s\right) $
(see Figure \ref{fig:WuWs}).

\begin{figure}[tbp]
\begin{center}
\includegraphics[height=3cm]{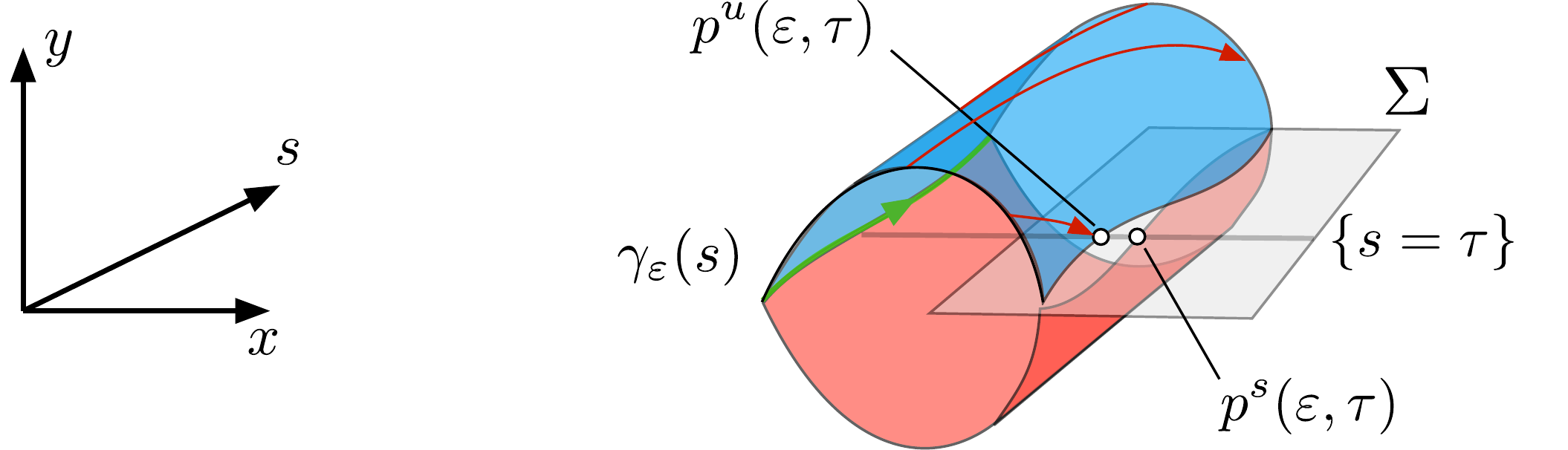}
\end{center}
\caption{The perturbed orbit $\protect\gamma _{\protect\varepsilon }(s)$ in
green (which coincides with $\Lambda _{\protect\varepsilon }$), its unstable
manifold in blue and stable manifold in red. The two points on $\Sigma
=\{y=0\}$ are $p^{u}\left( \protect\varepsilon ,\protect\tau \right)
:=P^{u}\left( w^{u}\left( r_{u},\protect\varepsilon ,\protect\kappa %
^{u}\left( \protect\varepsilon ,\protect\tau \right) \right) \right) $ and $%
p^{s}\left( \protect\varepsilon ,\protect\tau \right) :=P^{s}\left(
w^{s}\left( r_{s},\protect\varepsilon ,\protect\kappa ^{s}\left( \protect%
\varepsilon ,\protect\tau \right) \right) \right) $. In red, we have the
trajectory along the solution of the ODE, which leads to $p^{u}\left( 
\protect\varepsilon ,\protect\tau \right) $. The (signed) distance between $%
p^{u}\left( \protect\varepsilon ,\protect\tau \right) $ and $p^{s}\left( 
\protect\varepsilon ,\protect\tau \right) $ is the $\protect\delta \left( 
\protect\varepsilon ,\protect\tau \right) $.}
\label{fig:WuWs}
\end{figure}
\begin{figure}[tbp]
\caption{}
\begin{center}
\includegraphics[height=2.7cm]{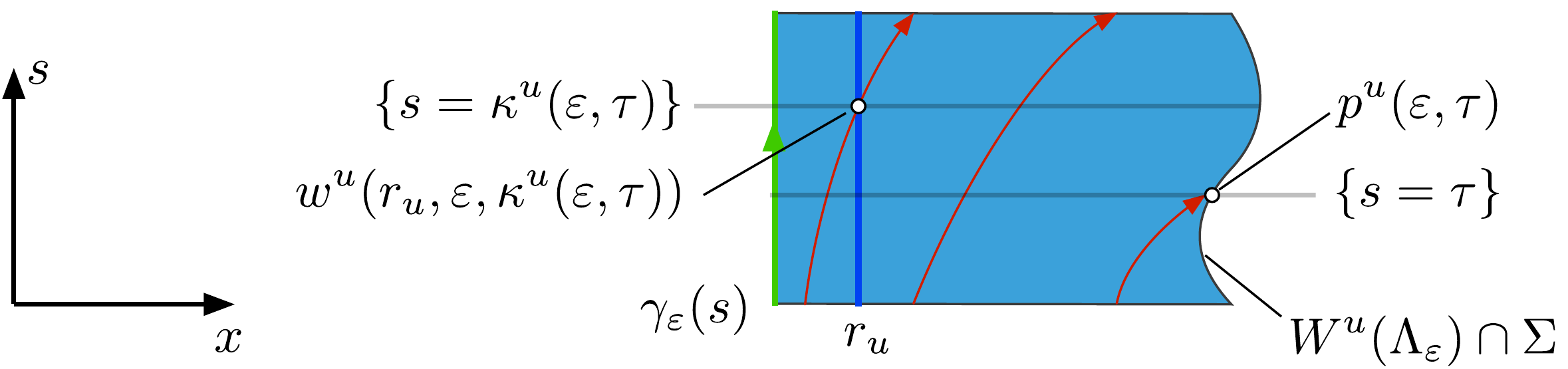}
\end{center}
\caption{The perturbed orbit $\protect\gamma _{\protect\varepsilon }(s)$
(which coincides with $\Lambda _{\protect\varepsilon }$) in green and its
unstable manifold in blue, projected onto the $x,s$ coordinates. The curled
right edge of the blue region is the intersection of the unstable manufold
with $\Sigma $. If we start from the point $w^{u}\left( r_{u},\protect%
\varepsilon ,\protect\kappa ^{u}\left( \protect\varepsilon ,\protect\tau %
\right) \right) $, which is at an $r_{u}$ distance from $\protect\gamma _{%
\protect\varepsilon }(s)$, and whose $s$-time is $\protect\kappa ^{u}\left( 
\protect\varepsilon ,\protect\tau \right) $, then we will reach $p^{u}\left( 
\protect\varepsilon ,\protect\tau \right) :=P^{u}\left( w^{u}\left( r_{u},%
\protect\varepsilon ,\protect\kappa ^{u}\left( \protect\varepsilon ,\protect%
\tau \right) \right) \right) $ along a trajectory of the ODE (depicted in
red). }
\label{fig:Wu}
\end{figure}

Let $\tau ^{u}(q)$ and $\tau ^{s}\left( q\right) $ stand for 
\begin{align}
\tau ^{u}\left( q\right) & =\pi _{s}q+\inf \left\{ t>0:\Phi _{t}\left(
q\right) \in \Sigma \right\} ,  \label{eq:tau-def} \\
\tau ^{s}\left( q\right) & =\pi _{s}q+\sup \left\{ t<0:\Phi _{t}\left(
q\right) \in \Sigma \right\} .  \notag
\end{align}%
Therefore $\tau ^{u}(q)$ is the $s$-time coordinate of the point from the
first intersection of $\Sigma $ with the forward trajectory of point $q$.
Then $\tau ^{u}(q)-\pi _{s}q$ is the $t$-time needed for $q$ to reach the
section $\Sigma $. For the $\tau ^{s}(q)$ we have analogous interpretation.

Let $P^{u}\ $and $P^{s}$ be maps defined as%
\begin{align*}
P^{u}\left( q\right) & =\Phi _{\tau ^{u}\left( q\right) -\pi _{s}q}\left(
q\right) , \\
P^{s}\left( q\right) & =\Phi _{\tau ^{s}\left( q\right) -\pi _{s}q}\left(
q\right) .
\end{align*}%
The domains of $P^{u}\ $and $P^{s}$ are subsets of $\mathbb{R}^{2}\times
E\times \mathbb{S}^{1}$, which contain $w^{u}\left( \left( 0,r_{u}\right]
\times E\times \mathbb{S}^{1}\right) $ and $w^{s}\left( \left( 0,r_{s}\right]
\times E\times \mathbb{S}^{1}\right) $, respectively. Observe that%
\begin{equation}
\pi _{s}P^{u}\left( q\right) =\tau ^{u}\left( q\right) \qquad \text{and}%
\qquad \pi _{s}P^{s}\left( q\right) =\tau ^{s}\left( q\right) .
\label{eq:Ps-tau-relation}
\end{equation}%
We shall assume that for any $\tau $ we can solve the following implicit
equations for functions $\kappa ^{u},\kappa ^{s}:E\times \mathbb{S}%
^{1}\rightarrow \mathbb{S}^{1}$: 
\begin{align}
\tau ^{u}\left( w^{u}\left( r_{u},\varepsilon ,\kappa ^{u}\left( \varepsilon
,\tau \right) \right) \right) & =\tau ,  \label{eq:kappa-def1} \\
\tau ^{s}\left( w^{s}\left( r_{s},\varepsilon ,\kappa ^{s}\left( \varepsilon
,\tau \right) \right) \right) & =\tau .  \label{eq:kappa-def2}
\end{align}

Function $\kappa ^{u}(\varepsilon ,\tau )$ gives the $s$-time of the point
on the unstable manifold with the unstable parameter $r_{u}$ that reaches
the section $\Sigma $ in the $s$-time equal to $\tau $ (see Figure \ref%
{fig:Wu}).

The questions related to the solvability of (\ref{eq:kappa-def1}),(\ref%
{eq:kappa-def2}) are discussed in Remark \ref{rem:kappa-sol}. We define the
distance function $\delta :E\times \mathbb{S}^{1}\rightarrow \mathbb{R}$: 
\begin{equation}
\delta \left( \varepsilon ,\tau \right) :=\pi _{x}p^{u}\left( \varepsilon
,\tau \right) -\pi _{x}p^{s}\left( \varepsilon ,\tau \right) ,
\label{eq:delta-def}
\end{equation}%
where%
\begin{equation}
p^{\iota }\left( \varepsilon ,\tau \right) :=P^{\iota }\left( w^{\iota
}\left( r_{\iota },\varepsilon ,\kappa ^{\iota }\left( \varepsilon ,\tau
\right) \right) \right) \qquad \qquad \text{for }\iota \in \left\{
s,u\right\} .  \label{eq:p-e-tau-def}
\end{equation}%
The $\delta $ will play the key role in our derivations. It will turn out
that $\delta (\varepsilon ,\tau )$ measures the (signed) distance between
the intersections of $W^{u}(\Lambda _{\varepsilon })\cap \{s=\tau \}$ and $%
W^{s}(\Lambda _{\varepsilon })\cap \{s=\tau \}$ on $\Sigma $.

We now formulate our main result.

\begin{theorem}
\label{th:main}Assume that there exists $\tau _{1},\tau _{2}\in \mathbb{S}%
^{1}$ such that for any $\varepsilon \in E$%
\begin{equation}
\frac{\partial }{\partial \varepsilon }\delta \left( \varepsilon ,\tau
_{1}\right) <0,\qquad \frac{\partial }{\partial \varepsilon }\delta \left(
\varepsilon ,\tau _{2}\right) >0.  \label{eq:topol-transv-cond}
\end{equation}%
Then for any $\varepsilon \in E\setminus \{0\}$ there exists $\tau ^{\ast
}\left( \varepsilon \right) \in \left( \tau _{1},\tau _{2}\right) $ such
that $W^{u}\left( \Lambda _{\varepsilon }\right) $ and $W^{s}\left( \Lambda
_{\varepsilon }\right) $ intersect at a point $q\left( \varepsilon \right)
\in \Sigma $, for which $\pi _{\left( \varepsilon ,\tau \right) }q\left(
\varepsilon \right) =\left( \varepsilon ,\tau \right) .$

Moreover, if in addition 
\begin{equation}
\frac{\partial ^{2}}{\partial \tau \partial \varepsilon }\delta \left(
\varepsilon ,\tau \right) >0,\qquad \text{for any }\varepsilon \in E\text{
and }\tau \in \left( \tau _{1},\tau _{2}\right) ,  \label{eq:transv-cond}
\end{equation}%
then $q(\varepsilon )$ is uniquely defined and for any fixed $\varepsilon
\in E\setminus \{0\}$, the manifolds $W^{u}\left( \Lambda _{\varepsilon
}\right) $ and $W^{s}\left( \Lambda _{\varepsilon }\right) $ intersect
transversally at $q\left( \varepsilon \right) $; the transversality is
considered in the $x,y,s$ coordinates.
\end{theorem}

\begin{proof}
Let us fix $\varepsilon$ and $\tau,$ and define two points 
\begin{equation}
q^{\iota}=q^{\iota}\left( \varepsilon,\tau\right) =P^{\iota}\left(
w^{\iota}\left( r_{\iota},\varepsilon,\kappa^{\iota}\left( \varepsilon
,\tau\right) \right) \right) \qquad\text{for }\iota\in\left\{ u,s\right\} .
\label{eq:qsu-def}
\end{equation}
By definition $q^{u},q^{s}\in\Sigma$. Moreover, by definition of $%
\kappa^{\iota}\left( \varepsilon,\tau\right) $ (see (\ref{eq:kappa-def1}-- %
\ref{eq:kappa-def2})) 
\begin{equation*}
\pi_{\left( \varepsilon,s\right) }q^{u}=\left( \varepsilon,\tau\right)
=\pi_{\left( \varepsilon,s\right) }q^{s}.
\end{equation*}
Moreover, by the definition of $\delta$, we also have 
\begin{equation*}
\pi_{x}\left( q^{u}-q^{s}\right) =\delta\left( \varepsilon,\tau\right) .
\end{equation*}
We therefore see that to establish that $q^{u}=q^{s}$ it is sufficient to
check that $\delta\left( \varepsilon,\tau\right) =0$.

If $\varepsilon=0$, then, since for the unperturbed problem we have a
homoclinic orbit, for any $\tau\in\mathbb{S}^{1}$%
\begin{equation*}
\delta( \varepsilon=0,\tau) =0.
\end{equation*}

We have 
\begin{align}
\delta (\varepsilon ,\tau )& =\delta \left( 0,\tau \right) +\int_{0}^{1}%
\frac{d}{dx}\delta \left( x\varepsilon ,\tau \right) dx  \notag \\
& =\varepsilon \int_{0}^{1}\frac{\partial }{\partial \varepsilon }\delta
\left( x\varepsilon ,\tau \right) dx.  \label{eq:delta-integral}
\end{align}%
From our assumptions it therefore follows that for any $\varepsilon \in
E\setminus \{0\}$%
\begin{equation*}
\delta \left( \varepsilon ,\tau _{1}\right) <0,\qquad\qquad \delta \left(
\varepsilon ,\tau _{2}\right) >0.
\end{equation*}%
By the Bolzano intermediate value theorem (applied to $\tau \rightarrow
\delta \left( \varepsilon ,\tau \right) $), for any $\varepsilon \in E$
there needs to be a $\tau ^{\ast }\left( \varepsilon \right) $ in $\left(
\tau _{1},\tau _{2}\right) $, such that%
\begin{equation*}
\delta \left( \varepsilon ,\tau ^{\ast }\left( \varepsilon \right) \right)
=\pi _{x}\left( q^{u}\left( \varepsilon ,\tau ^{\ast }\left( \varepsilon
\right) \right) -q^{s}\left( \varepsilon ,\tau ^{\ast }\left( \varepsilon
\right) \right) \right) =0,
\end{equation*}%
hence the manifolds intersect at $q^{u}\left( \varepsilon ,\tau ^{\ast
}\left( \varepsilon \right) \right) =q^{s}\left( \varepsilon ,\tau ^{\ast
}\left( \varepsilon \right) \right) .$

We now prove the transversality. As a consequence of the transversality we
obtain the uniqueness of $q(e)$. Let us fix $e\in E\setminus \{0\}$. Observe
that since 
\begin{equation*}
\frac{\partial }{\partial \tau }\delta \left( e,\tau \right) =e\int_{0}^{1}%
\frac{\partial ^{2}}{\partial \tau \partial \varepsilon }\delta \left(
xe,\tau \right) dx\neq 0,\quad \tau \in (\tau _{1},\tau _{2})
\end{equation*}%
the intersection parameter $\tau ^{\ast }(\varepsilon )$ is uniquely
defined. Let $q(e)$ denotes the intersection point 
\begin{equation*}
q(e)=q^{u}(e,\tau ^{\ast }(e))=q^{s}(e,\tau ^{\ast }(e)).
\end{equation*}

We consider the transversality in the coordinates $y,x,s$.

Let $w=f\left( q\left( e\right) \right) $. Since $q\left( e\right) \in
W^{s}\left( \Lambda\right) \cap W^{u}\left( \Lambda\right) $, we have%
\begin{equation*}
w\in T_{q\left( e\right) }W^{s}\left( \Lambda\right) \qquad\text{and\qquad }%
w\in T_{q\left( e\right) }W^{u}\left( \Lambda\right) .
\end{equation*}
Since the intersections of $W^{s}\left( \Lambda\right) $ and $W^{u}\left(
\Lambda\right) $ with $\Sigma$ are transversal, and since $\Sigma=\left\{
y=0\right\} ,$ it follows that 
\begin{equation}
w=\left( 
\begin{array}{c}
\pi_{x}w \\ 
\pi_{y}w \\ 
\pi_{\varepsilon}w \\ 
\pi_{s}w%
\end{array}
\right) \qquad\text{with }\pi_y w\neq0, \ \pi_s w=1.  \label{eq:tangent-w}
\end{equation}

We now consider additional two vectors $v^{\iota}\in T_{q\left( e\right)
}W^{\iota}\left( \Lambda\right) ,$ for $\iota\in\left\{ s,u\right\} $
defined as 
\begin{equation*}
v^{\iota}=\frac{\partial}{\partial\tau}q^{\iota}\left( e,\tau\right) .
\end{equation*}
Since by construction $\pi_{\tau}q^{\iota}\left( e,\tau\right) =\tau,$ $%
\pi_{\varepsilon}q^{\iota}\left( e,\tau\right) =e$ and $\pi_{y}q^{\iota
}\left( e,\tau\right) =0,$ we have%
\begin{equation}
v^{\iota}=\left( 
\begin{array}{c}
\pi_{x}v^{\iota} \\ 
\pi_{y}v^{\iota} \\ 
\pi_{\varepsilon}v^{\iota} \\ 
\pi_{\tau}v^{\iota}%
\end{array}
\right) =\left( 
\begin{array}{c}
\pi_{x}\frac{\partial}{\partial\tau}q^{\iota}\left( e,\tau\right) \\ 
0 \\ 
0 \\ 
1%
\end{array}
\right) .  \label{eq:tangent-vs}
\end{equation}
To show transversality in the $x,y,s$ coordinates, it is sufficient to show
that%
\begin{equation*}
\mathrm{span}\left( w,v^{s},v^{u}\right) =\mathbb{R}^{2}\times\left\{
0\right\} \times\mathbb{R}.
\end{equation*}
Looking at (\ref{eq:tangent-w}--\ref{eq:tangent-vs}) we see that this will
be the case if 
\begin{equation*}
\frac{\partial}{\partial\tau}\pi_{x}q^{u}\left( e,\tau\right) -\frac {%
\partial}{\partial\tau}\pi_{x}q^{s}\left( e,\tau\right) \neq0.
\end{equation*}
In other words, by (\ref{eq:delta-def}) and (\ref{eq:qsu-def}), we need to
show that 
\begin{equation*}
\frac{\partial}{\partial\tau}\delta\left( e,\tau\right) \neq0.
\end{equation*}
From (\ref{eq:delta-integral}) it follows that 
\begin{equation*}
\frac{\partial}{\partial\tau}\delta\left( e,\tau\right) =e\int_{0}^{1}\frac{%
\partial^{2}}{\partial\tau\partial\varepsilon}\delta\left( xe,\tau\right) dx.
\end{equation*}
From our assumptions we have that $\frac{\partial^{2}}{\partial\tau
\partial\varepsilon}\delta\left( \varepsilon,\tau\right) >0,$ hence from
above equation follows that $\frac{\partial}{\partial\tau}\delta\left(
e,\tau\right) \neq0$. This concludes the proof of the transversality.
\end{proof}

\begin{remark}
Theorem \ref{th:main} follows along the standard lines of Melnikov-type
arguments. The novelty is that we formulate our assumptions so that we
obtain the intersection for all $\varepsilon \in E\setminus \{0\}$, and not
only for \textquotedblleft sufficiently small" $\varepsilon $. The main
difficulty does not lie in the proof of this theorem, which is
straightforward, but in the ability to verify its assumptions. The
subsequent sections will be devoted to showing how (\ref%
{eq:topol-transv-cond}) and (\ref{eq:transv-cond}) can be validated using
(rigorous) computer assisted computations.
\end{remark}

\subsection{Verification of assumptions\label{sec:assumptions-verif}}

To apply Theorem \ref{th:main} we need to be able to obtain bounds for $%
\frac{\partial }{\partial \varepsilon }\delta \left( \varepsilon ,\tau
\right) $ and $\frac{\partial ^{2}}{\partial \tau \partial \varepsilon }%
\delta \left( \varepsilon ,\tau \right) $. Our objective will be to obtain
such bounds using rigorous, interval-arithmetic-based, computer assisted
computations. In this section we will show that the key are the bounds for $%
Dw^{\iota }$ and $D^{2}w^{\iota },\ $where $\iota \in \left\{ u,s\right\} $,
and that other estimates follow with relative ease.

Throughout the section we use the notation $\iota$ to stand for an index
from the set $\left\{ u,s\right\} $.

In our implementation we use the CAPD\footnote{%
Computer Assisted Proofs in Dynamics: http://capd.ii.uj.edu.pl/} package.
This package allows for the computation of derivatives (of a prescribed
order) of Poincar\'{e} maps of flows induced by ODEs. We therefore start
from a comfortable assumption that for a given set $U\subset \mathbb{R}%
^{2}\times E\times \mathbb{S}^{1}$ the bounds on $P^{\iota }\left( U\right) $%
, $DP^{\iota }\left( U\right) ,$ and $D^{2}P^{\iota }\left( U\right) $ are
automatically computed by the CAPD package \cite{WZ},\cite{Z}.

To simplify the notation, we consider%
\begin{equation}
g^{\iota }\left( \varepsilon ,\tau ,s\right) =\pi _{s}P^{\iota }\left(
w^{\iota }\left( r_{\iota },\varepsilon ,s\right) \right) -\tau \qquad \text{%
for }\iota \in \left\{ u,s\right\} ,  \label{eq:g1}
\end{equation}%
Since $\kappa ^{\iota }\left( \varepsilon ,\tau \right) $ is a solution of%
\begin{equation*}
g^{\iota }\left( \varepsilon ,\tau ,\kappa ^{\iota }\left( \varepsilon ,\tau
\right) \right) =0,
\end{equation*}
the $g^{\iota }$ will be used to find $\frac{\partial \kappa ^{\iota }}{%
\partial \varepsilon }$, $\frac{\partial \kappa ^{\iota }}{\partial \tau }$, 
$\frac{\partial ^{2}\kappa ^{\iota }}{\partial \tau \partial \varepsilon }$
using implicit differentiation.

\begin{remark}
\label{rem:kappa-sol}Let us assume that $\varepsilon =0$. We are then in the
setting of an autonomous ODE. Then $\tau ^{\iota }\left( w^{\iota }\left(
r_{\iota },\varepsilon =0,s\right) \right) =s+\omega ^{\iota }$ for some
fixed $\omega ^{\iota }\in \mathbb{R}$, and therefore $\kappa ^{\iota
}\left( \varepsilon =0,\tau \right) =\tau -\omega ^{\iota }$ is well
defined. Also, by (\ref{eq:Ps-tau-relation}), 
\begin{equation*}
\frac{\partial }{\partial s}g^{\iota }\left( \varepsilon =0,\tau ,s\right)
=1,\qquad \text{for any }\tau \in \mathbb{S}^{1},
\end{equation*}%
which means that we can apply the implicit function theorem for $g^{\iota
}=0 $ to obtain existence of $\kappa ^{\iota }\left( \varepsilon ,\tau
\right) ,$ for sufficiently small $\varepsilon \geq 0$.
\end{remark}

We can now differentiate $g^{\iota }$ to obtain (below we omit the
dependence of $g$ and $\kappa $ on $\iota $ to simplify notations) 
\begin{align}
\frac{d}{d\varepsilon }g(\varepsilon ,\tau ,\kappa (\varepsilon ,\tau ))& =%
\frac{\partial g}{\partial \varepsilon }+\frac{\partial g}{\partial s}\frac{%
\partial \kappa }{\partial \varepsilon },  \label{eq:partials-of-g-1} \\
\frac{d}{d\tau }g(\varepsilon ,\tau ,\kappa (\varepsilon ,\tau ))& =\frac{%
\partial g}{\partial \tau }+\frac{\partial g}{\partial s}\frac{\partial
\kappa }{\partial \tau },  \label{eq:partials-of-g} \\
\frac{d}{d\varepsilon }\frac{d}{d\tau }g(\varepsilon ,\tau ,\kappa
(\varepsilon ,\tau ))& =\frac{\partial ^{2}g}{\partial \varepsilon \partial
\tau }+\frac{\partial ^{2}g}{\partial \varepsilon \partial s}\frac{\partial
\kappa }{\partial \tau }+\left( \frac{\partial ^{2}g}{\partial s\partial
\tau }+\frac{\partial ^{2}g}{\partial ^{2}s}\frac{\partial \kappa }{\partial
\tau }\right) \frac{\partial \kappa }{\partial \varepsilon }+\frac{\partial g%
}{\partial s}\frac{\partial ^{2}\kappa }{\partial \varepsilon \partial \tau }%
.  \label{eq:partials-of-g-2}
\end{align}

To compute $\frac{\partial \kappa ^{\iota }}{\partial \varepsilon }$, $\frac{%
\partial \kappa ^{\iota }}{\partial \tau }$, $\frac{\partial ^{2}\kappa
^{\iota }}{\partial \tau \partial \varepsilon }$ we consider 
\begin{align}
\frac{d}{d\varepsilon }g^{\iota }(\varepsilon ,\tau ,\kappa ^{\iota
}(\varepsilon ,\tau ))& =0,  \label{eq:impl-for-kappa-e} \\
\frac{d}{d\tau }g^{\iota }(\varepsilon ,\tau ,\kappa ^{\iota }(\varepsilon
,\tau ))& =0,  \label{eq:impl-for-kappa-tau} \\
\frac{d^{2}}{d\tau d\varepsilon }g^{\iota }(\varepsilon ,\tau ,\kappa
^{\iota }(\varepsilon ,\tau ))& =0.  \label{eq:impl-for-kappa-e-tau}
\end{align}%
From (\ref{eq:partials-of-g-1}) together with (\ref{eq:impl-for-kappa-e}),
and from (\ref{eq:partials-of-g}) together with (\ref{eq:impl-for-kappa-tau}%
), we obtain 
\begin{equation}
\frac{\partial \kappa ^{\iota }}{\partial \varepsilon }\left( \varepsilon
,\tau \right) =-\frac{\frac{\partial g^{\iota }}{\partial \varepsilon }%
(\varepsilon ,\tau ,\kappa ^{\iota }(\varepsilon ,\tau ))}{\frac{\partial
g^{\iota }}{\partial s}(\varepsilon ,\tau ,\kappa ^{\iota }(\varepsilon
,\tau ))},.  \label{eq:kappa-der11}
\end{equation}%
and%
\begin{equation}
\frac{\partial \kappa }{\partial \tau }\left( \varepsilon ,\tau \right) =%
\frac{-\frac{\partial g^{\iota }}{\partial \tau }(\varepsilon ,\tau ,\kappa
^{\iota }(\varepsilon ,\tau ))}{\frac{\partial g^{\iota }}{\partial s}%
(\varepsilon ,\tau ,\kappa ^{\iota }(\varepsilon ,\tau ))}=\frac{-1}{\frac{%
\partial g^{\iota }}{\partial s}(\varepsilon ,\tau ,\kappa ^{\iota
}(\varepsilon ,\tau ))}  \label{eq:kappa-der12}
\end{equation}

We note that from (\ref{eq:g1}) follows that 
\begin{equation*}
\frac{\partial ^{2}g^{\iota }}{\partial \varepsilon \partial \tau }=\frac{%
\partial ^{2}g^{\iota }}{\partial s\partial \tau }=0.
\end{equation*}%
This means that from (\ref{eq:partials-of-g-2}),(\ref%
{eq:impl-for-kappa-e-tau}) we obtain 
\begin{equation}
\frac{\partial ^{2}\kappa ^{\iota }}{\partial \varepsilon \partial \tau }=%
\frac{-1}{\frac{\partial g^{\iota }}{\partial s}}\left( \frac{\partial
^{2}g^{\iota }}{\partial \varepsilon \partial s}\frac{\partial \kappa
^{\iota }}{\partial \tau }+\frac{\partial ^{2}g^{\iota }}{\partial ^{2}s}%
\frac{\partial \kappa ^{\iota }}{\partial \tau }\frac{\partial \kappa
^{\iota }}{\partial \varepsilon }\right) .  \label{eq:kappa-der2}
\end{equation}

To compute $\frac{\partial }{\partial \varepsilon }\delta \left( \varepsilon
,\tau \right) $, $\frac{\partial ^{2}}{\partial \tau \partial \varepsilon }%
\delta \left( \varepsilon ,\tau \right) $ we define%
\begin{equation}
h^{\iota }\left( \varepsilon ,s\right) =P^{\iota }\left( w^{\iota }\left(
r_{\iota },\varepsilon ,s\right) \right) ,  \label{eq:h-def}
\end{equation}%
compute%
\begin{eqnarray}
\frac{d}{d\varepsilon }h^{\iota }\left( \varepsilon ,\kappa ^{\iota }\left(
\varepsilon ,\tau \right) \right) &=&\frac{\partial h^{\iota }}{\partial
\varepsilon }+\frac{\partial h^{\iota }}{\partial s}\frac{\partial \kappa
^{\iota }}{\partial \varepsilon },  \notag \\
\frac{d^{2}}{d\tau d\varepsilon }h^{\iota }\left( \varepsilon ,\kappa
^{\iota }\left( \varepsilon ,\tau \right) \right) &=&\frac{\partial
^{2}h^{\iota }}{\partial s\partial \varepsilon }\frac{\partial \kappa
^{\iota }}{\partial \tau }+\frac{\partial ^{2}h^{\iota }}{\partial s^{2}}%
\frac{\partial \kappa ^{\iota }}{\partial \tau }\frac{\partial \kappa
^{\iota }}{\partial \varepsilon }+\frac{\partial h^{\iota }}{\partial s}%
\frac{\partial ^{2}\kappa ^{\iota }}{\partial \tau \partial \varepsilon },
\label{eq:d2-h-expansion}
\end{eqnarray}%
and obtain $\frac{\partial }{\partial \varepsilon }\delta \left( \varepsilon
,\tau \right) $, $\frac{\partial ^{2}}{\partial \tau \partial \varepsilon }%
\delta \left( \varepsilon ,\tau \right) $ from the fact that%
\begin{equation}
\delta \left( \varepsilon ,\tau \right) =\pi _{x}h^{u}\left( \varepsilon
,\kappa ^{u}\left( \varepsilon ,\tau \right) \right) -\pi _{x}h^{s}\left(
\varepsilon ,\kappa ^{s}\left( \varepsilon ,\tau \right) \right) .
\label{eq:delta-by-h}
\end{equation}

We finish this section by discussing how to solve (\ref{eq:kappa-def1}--\ref%
{eq:kappa-def2}) for $\kappa ^{u}$ and $\kappa ^{s}$. One possibility is to
use the interval Newton method. We present how this can be done in Appendix %
\ref{app:implicit-f-sol}. In our case, since the dimension of the equations
in question is one, we use the following lemma in our computer assisted part
of the proof:

\begin{lemma}
\label{lem:de-kappa-ver copy(1)}Let $\iota \in \left\{ u,s\right\} $ be
fixed and let $A=\left[ a_{1},a_{2}\right] $. Assume that for any $%
\varepsilon \in E$, function $s\rightarrow \pi _{s}h^{\iota }\left(
\varepsilon ,s\right) $ is strictly increasing on $A$. Consider a fixed $%
\tau \in \mathbb{S}^{1}$. If%
\begin{equation}
\pi _{s}h^{\iota }\left( \varepsilon ,a_{1}\right) <\tau <\pi _{s}h^{\iota
}\left( \varepsilon ,a_{2}\right) ,  \label{eq:expl-Lem1-2}
\end{equation}%
then for every $\varepsilon \in E,$ $\kappa ^{\iota }\left( \varepsilon
,\tau \right) \in A$.
\end{lemma}

\begin{proof}
The result follows directly from the Bolzano's intermediate value theorem.
\end{proof}

All computations discussed in this section can be performed in interval
arithmetic, provided that we have estimates for $\frac{\partial w^{\iota }}{%
\partial x_{j}}$, $\frac{\partial ^{2}w^{\iota }}{\partial x_{i}\partial
x_{j}}$. How to obtain such estimates will be discussed in section \ref%
{sec:ODE-Wcu}.

\begin{remark}
The method for obtaining bounds on $\frac{\partial w^{\iota }}{\partial x_{j}%
}$, $\frac{\partial ^{2}w^{\iota }}{\partial x_{i}\partial x_{j}}$, which is
the subject of sections \ref{sec:wcu-maps} and \ref{sec:ODE-Wcu}, is based
on the geometric method for normally hyperbolic invariant manifolds from 
\cite{conecond,Geom}. There are alternative methods to perform such
computation. For instance, \cite{param-method} discusses how such bounds can
be obtained using the parameterization method. This method can be
implemented to perform interval based validated numerical bounds. A reader
who is a specialist in this field can choose to use the parameterization
method to validate assumptions of Theorem \ref{th:main}. If such choice is
made, the specialist can in fact stop reading this paper at this point and
most likely successfully apply our method.
\end{remark}


\section{Center-unstable manifolds for maps\label{sec:wcu-maps}}

In this section we recall the results from \cite{Geom}, which give
conditions for establishing the existence and smoothness of normally
hyperbolic invariant manifolds, together with their associated center-stable
and cnter-unstable manifolds. Here we focus on the cnter-unstable manifolds,
since this is sufficient for our needs. (The center-stable manifold of an
ODE is the center unstable manifold for time reversed ODE, thus knowing how
to handle one of the two is enough.) The results from \cite{Geom} are
recalled in sections \ref{sec:wcu-maps-setup} and \ref{sec:wcu-maps-results}.

In sections \ref{sec:maps-param} and \ref{sec:wcu-maps-d2-bounds} we extend
the results from to \cite{Geom}. Section \ref{sec:maps-param} discusses the
dependence of the manifolds on parameters. In section \ref%
{sec:wcu-maps-d2-bounds} we show how to obtain explicit estimates for the
second derivatives of the manifolds with respect to parameters.

All results in this section are formulated in the setting of maps. In
section \ref{sec:ODE-Wcu} we reformulate them for ODEs.

\subsection{Definitions and setup\label{sec:wcu-maps-setup}}

We assume that $\Lambda$ is a $c$-dimensional torus and use the notation 
\begin{equation*}
\varphi:\mathbb{R}^{c}\rightarrow\Lambda=\left( \mathbb{R}/\mathbb{Z}\right)
^{c},
\end{equation*}
for its covering. This gives us the set of charts being the restriction of $%
\varphi$ to balls $B$ in $\mathbb{R}^{c}$, which are small enough so that $%
\varphi:B\rightarrow\Lambda$ is a homeomorphism on its image. We introduce a
notation $R_{\Lambda}>0$ for a radius such that $\varphi_{|B(\lambda
,R_{\Lambda})}$ is a homeomorphism onto its image. We can for instance take $%
R_{\lambda}=\frac{1}{2}.$

Let $R<\frac{1}{2}R_{\Lambda}$ and denote by $D$ the set 
\begin{equation*}
D=\Lambda\times\overline{B}_{u}(R)\times\overline{B}_{s}(R),
\end{equation*}
where $\overline{B}_{n}(R)$ stands for a closed ball of radius $R$, centered
at zero, in $\mathbb{R}^{n}$. We consider a $C^{k+1}$ map, for $k\geq1$, 
\begin{equation*}
F:D\rightarrow\Lambda\times\mathbb{R}^{u}\times\mathbb{R}^{s}.
\end{equation*}
Here we assume that the map is considered in local coordinates that are
(roughly) well aligned with the dynamics. Throughout the section we use the
notation $z=(\lambda,x,y)$ to denote points in $D$. This means that notation 
$\lambda$ will stand for points on $\Lambda$, notation $x$ for points in $%
\mathbb{R}^{u}$, and $y$ for points in $\mathbb{R}^{s}$. The coordinate $x$
will be the unstable direction and $y$ will be the stable. We will write $F$
as $(F_{\lambda},F_{x},F_{y})$ , where $F_{\lambda},F_{x},F_{y}$ stand for
projections onto $\Lambda$, $\mathbb{R}^{u}$ and $\mathbb{R}^{s}$,
respectively. On $\mathbb{R}^{c}\times\mathbb{R}^{u}\times\mathbb{R}^{s}$ we
will use the Euclidian norm.

The set of points which are in the same good chart with point $q\in D$ will
be denoted by 
\begin{equation}
P(q)=\{z\in D\ |\ \Vert\pi_{\lambda}z-\pi_{\lambda}q\Vert\leq R_{\Lambda
}/2\}.  \label{eq:map-P-const}
\end{equation}

Let $L\in\left( \frac{2R}{R_{\Lambda}},1\right) $, and let us define the
following constants: 
\begin{align*}
\mu_{s,1} & =\sup_{z\in D}\left\{ \left\Vert \frac{\partial F_{y}}{\partial y%
}\left( z\right) \right\Vert +\frac{1}{L}\left\Vert \frac{\partial F_{y}}{%
\partial(\lambda,x)}(z)\right\Vert \right\} , \\
\mu_{s,2} & =\sup_{z\in D}\left\{ \left\Vert \frac{\partial F_{y}}{\partial y%
}\left( z\right) \right\Vert +L\left\Vert \frac{\partial F_{\left(
\lambda,x\right) }}{\partial y}(z)\right\Vert \right\} ,
\end{align*}%
\begin{align*}
\xi_{u,1} & =\inf_{z\in D}\left\{ m\left( \frac{\partial F_{x}}{\partial x}%
(z)\right) -\frac{1}{L}\left\Vert \frac{\partial F_{x}}{\partial\left(
\lambda,y\right) }(z)\right\Vert \right\} , \\
\xi_{u,1,P} & =\inf_{z\in D}m\left( \frac{\partial F_{x}}{\partial x}%
(P(z))\right) -\frac{1}{L}\sup_{z\in D}\left\Vert \frac{\partial F_{x}}{%
\partial\left( \lambda,y\right) }(z)\right\Vert ,
\end{align*}%
\begin{align}
\mu_{cs,1} & =\sup_{z\in D}\left\{ \left\Vert \frac{\partial F_{\left(
\lambda,y\right) }}{\partial\left( \lambda,y\right) }(z)\right\Vert
+L\left\Vert \frac{\partial F_{\left( \lambda,y\right) }}{\partial x}%
(z)\right\Vert \right\} ,  \notag \\
\mu_{cs,2} & =\sup_{z\in D}\left\{ \left\Vert \frac{\partial F_{\left(
\lambda,y\right) }}{\partial\left( \lambda,y\right) }(z)\right\Vert +\frac{1%
}{L}\left\Vert \frac{\partial F_{x}}{\partial\left( \lambda,y\right) }%
(z)\right\Vert \right\} ,  \notag
\end{align}%
\begin{align*}
\xi_{cu,1} & =\inf_{z\in D}\left\{ m\left( \frac{\partial F_{(\lambda,x)}}{%
\partial(\lambda,x)}(z)\right) -L\left\Vert \frac{\partial F_{(\lambda,x)}}{%
\partial y}(z)\right\Vert \right\} , \\
\xi_{cu,1,P} & =\inf_{z\in D}m\left( \frac{\partial F_{(\lambda,x)}}{%
\partial(\lambda,x)}(P(z))\right) -L\sup_{z\in D}\left\Vert \frac{\partial
F_{(\lambda,x)}}{\partial y}(z)\right\Vert .
\end{align*}

Intuitively, the constants $\mu $ measure the contraction rates in $D$, and $%
\xi $ measure expansion. The index $cs$ stands for the `center-stable'
direction, $cu$ for `center-unstable', $s$ for `stable' and $u$ for
`unstable'. Thus, for instance, $\mu _{s,1}$ and $\mu _{s,2}$ measure
contraction in the stable direction. The number $1$ or $2$ as second index
is used according to the following rule: $1$, when both partial derivatives
are of the same component of $F$, while $2$ is used the differentiation is
done with respect to the same block of variables of various components of $f$%
. The occasional additional index $P$ indicates that the constants are `more
stringent' and defined over sets $P(z)$ defined in (\ref{eq:map-P-const}).
The constant $L$ will turn out to be the Lipschitz bound for the slope of
center-unstable manifold.

\begin{definition}
\label{def:rate-conditions}We say that $F$ satisfies rate conditions of
order $k\geq1$ if $\xi_{u,1},$ $\xi_{u,1,P},$ $\xi_{cu,1},$ $\xi_{cu,1,P},$
are strictly positive, and for all $k\geq j\geq1$ holds%
\begin{equation}
\mu_{s,1}<1<\xi_{u,1,P},  \label{eq:rate-cond-1}
\end{equation}%
\begin{align}
\frac{\mu_{cs,1}}{\xi_{u,1,P}} & <1,\qquad\frac{\mu_{s,1}}{\xi_{cu,1,P}}<1,
\label{eq:rate-cond-2} \\
\frac{\mu_{cs,2}}{\xi_{u,1}} & <1,\qquad\frac{\mu_{s,2}}{(\xi_{cu,1})^{j+1}}%
<1.  \label{eq:rate-cond-3}
\end{align}
\end{definition}

Intuitively, $F$ satisfies rate conditions if the contraction on the stable
coordinate is stronger than the contraction on center-unstable coordinate,
and the expansion on the unstable coordinate is stronger than expansion on
the center-stable coordinate.

We introduce the following notation: 
\begin{align*}
J_{s}(z,M) & =\left\{ \left( \lambda,x,y\right) :\left\Vert \left(
\lambda,x\right) -\pi_{\lambda,x}z\right\Vert \leq M\left\Vert y-\pi
_{y}z\right\Vert \right\} , \\
J_{u}\left( z,M\right) & =\left\{ \left( \lambda,x,y\right) :\left\Vert
\left( \lambda,y\right) -\pi_{\lambda,y}z\right\Vert \leq M\left\Vert
x-\pi_{x}z\right\Vert \right\} .
\end{align*}
We shall refer to $J_{s}(z,M)$ as a stable cone of slope $M$ at $z$, and to $%
J_{u}(z,M)$ as an unstable cone of slope $M$ at $z$. The cones are depicted
in Figures \ref{fig:st-cones} and \ref{fig:unst-cones}.

\begin{figure}[ptb]
\begin{center}
\includegraphics[height=4cm]{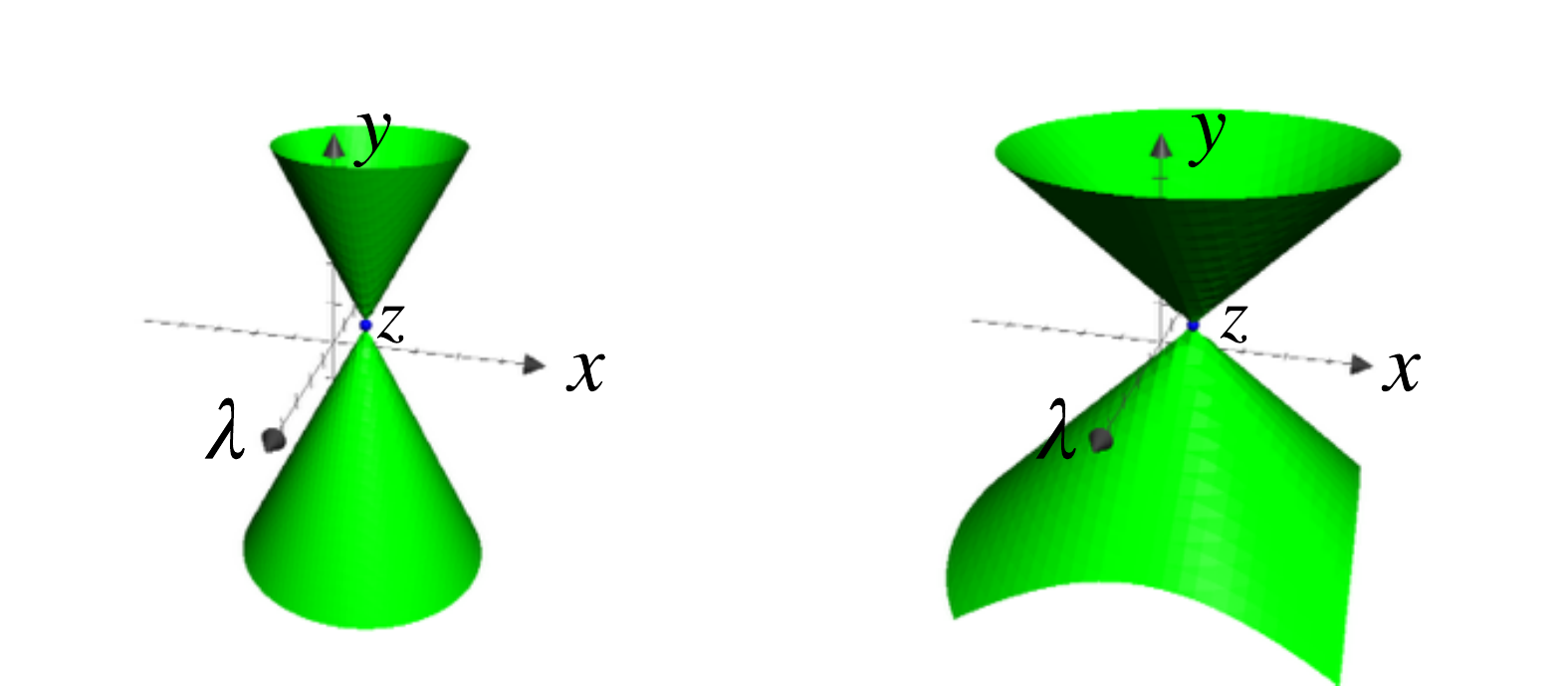}
\end{center}
\caption{The stable cone $J_{s}(z,M)$ for $M=\frac{1}{2}$ on the left, and $%
M=1$ on the right.}
\label{fig:st-cones}
\end{figure}

\begin{definition}
We say that a sequence $\left\{ z_{i}\right\} _{i=-\infty}^{0}$ is a (full)
backward trajectory of a point $z$ if $z_{0}=z,$ and $F\left( z_{i-1}\right)
=z_{i}$ for all $i\leq0.$
\end{definition}

\begin{definition}
\label{def:Wcu-maps}We define the center-unstable set in $D$ as%
\begin{equation*}
W^{cu}=\{z:\text{there is a full backward trajectory of }z\text{ in }D\}.
\end{equation*}
\end{definition}

\begin{definition}
\label{def:Wu-fiber} Assume that $z\in W^{cu}$. We define the unstable fiber
of $z$ as%
\begin{align*}
W_{z}^{u} & =\{p\in D:\exists\text{ backward trajectory }\left\{
p_{i}\right\} _{i=-\infty}^{0}\text{ of }p\text{ in }D, \\
& \quad\text{for any such backward trajectory} \\
& \quad\text{and any backward trajectory }\left\{ z_{i}\right\} _{i=-\infty
}^{0}\text{ of }z\text{ in }D \\
& \quad\text{holds }\left. p_{i}\in J_{u}\left( z_{i},1/L\right) \cap
D\right. \}.
\end{align*}
\end{definition}

\begin{figure}[ptb]
\begin{center}
\includegraphics[height=4cm]{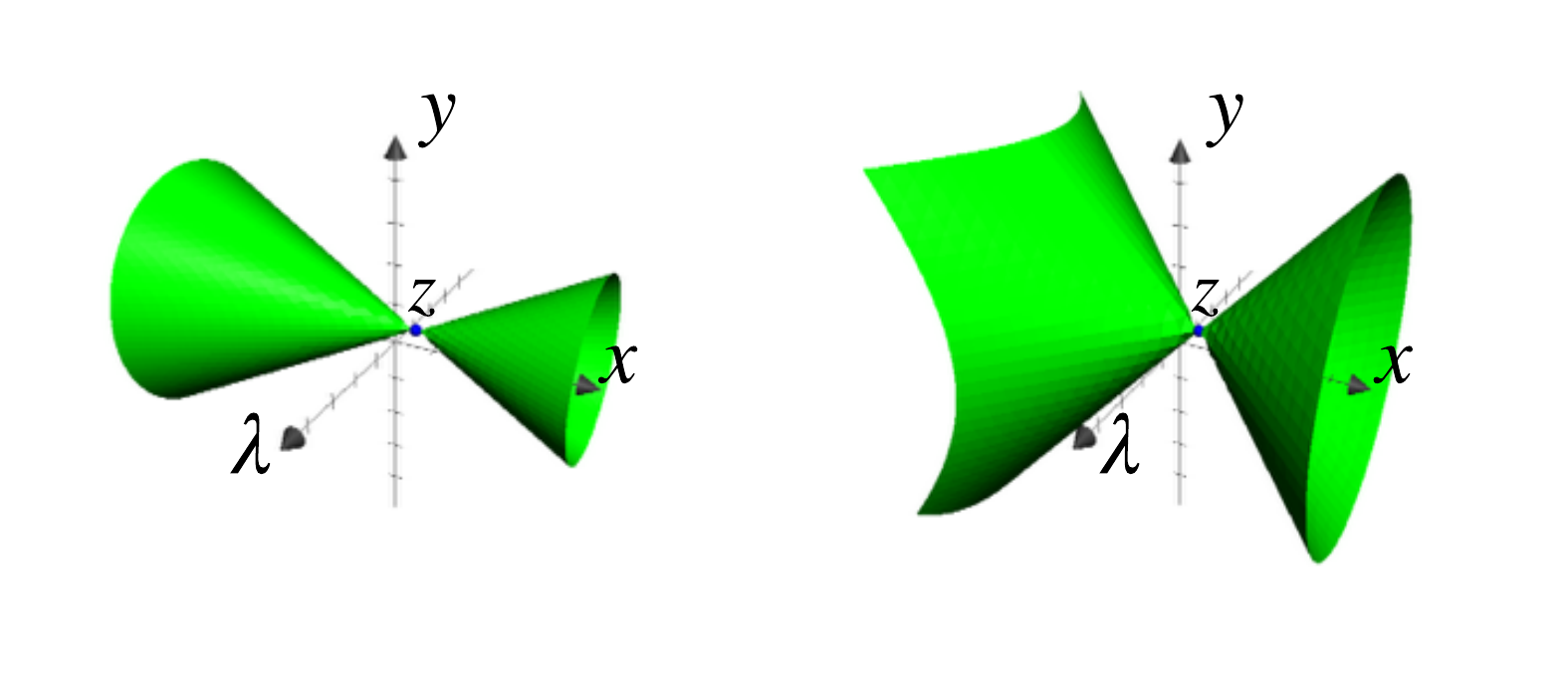}
\end{center}
\caption{The stable cone $J_{u}(z,M)$ for $M=\frac{1}{2}$ on the left, and $%
M=1$ on the right.}
\label{fig:unst-cones}
\end{figure}

The definition $W_{z}^{u}$ is related to cones, which is a nonstandard
approach, the standard one is through convergence rates. In Theorem \ref%
{th:main} we shall see that our definition implies the convergence rate as
in the standard theory \cite{F1,F2}.

\begin{definition}
\label{def:back-cc}We say that $F$ satisfies backward cone conditions if the
following condition is fulfilled:

If $z_{1},z_{2},F(z_{1}),F(z_{2})\in D$ and $F(z_{1})\in J_{s}\left(
F(z_{2}),1/L\right) $ then%
\begin{equation*}
z_{1}\in J_{s}\left( z_{2},1/L\right) .
\end{equation*}
\end{definition}

Intuitively, a function satisfies backward cone conditions, if images of two
points are vertically aligned, then the points themselves are also
vertically aligned. This is a technical condition that is associated with
the fact that we do not assume invertibility of our map. In the setting of
ODEs, the time shift along the trajectory map is invertible, and for small
times it is close to identity. It will turn out that backward cone
conditions are easily satisfied in the context of ODEs.

For $\lambda\in\Lambda$ we define the following sets: 
\begin{align*}
D_{\lambda} & =\overline{B}_{c}\left( \lambda,R_{\Lambda}\right) \times%
\overline{B}_{u}(R)\times\overline{B}_{s}(R), \\
D_{\lambda}^{+} & =\overline{B}_{c}\left( \lambda,R_{\Lambda}\right) \times%
\overline{B}_{u}(R)\times\partial B_{s}(R), \\
D_{\lambda}^{-} & =\overline{B}_{c}\left( \lambda,R_{\Lambda}\right)
\times\partial\overline{B}_{u}(R)\times B_{s}(R).
\end{align*}

\begin{definition}
\label{def:covering}We say that $F$ satisfies covering conditions if for any 
$z\in D$ there exists a $\lambda^{\ast}\in\Lambda$, such that the following
conditions hold:

For $U=J_{u}(z,1/L)\cap D$, there exists a homotopy $h$ 
\begin{equation*}
h:\left[ 0,1\right] \times U\rightarrow B_{c}\left( \lambda^{\ast
},R_{\Lambda}\right) \times\mathbb{R}^{u}\times\mathbb{R}^{s},
\end{equation*}
and a linear map $A:\mathbb{R}^{u}\rightarrow\mathbb{R}^{u}$ which satisfy:

\begin{enumerate}
\item \label{pt:covering-1}$h_{0}=F|_{U},$

\item \label{pt:covering-2}for any $\alpha\in\left[ 0,1\right] $, 
\begin{align}
h_{\alpha}\left( U\cap D_{\pi_{\theta}z}^{-}\right) \cap D_{\lambda^{\ast}}
& =\emptyset,  \label{eq:homotopy-exit} \\
h_{\alpha}\left( U\right) \cap D_{\lambda^{\ast}}^{+} & =\emptyset ,
\label{eq:homotopy-enter}
\end{align}

\item \label{pt:covering-3}$h_{1}\left( \lambda,x,y\right) =\left(
\lambda^{\ast},Ax,0\right) $,

\item \label{pt:covering-4}$A\left( \partial B_{u}(R)\right) \subset \mathbb{%
R}^{u}\setminus\overline{B}_{u}(R).$
\end{enumerate}
\end{definition}

In the above definition a reasonable choice for $\lambda^{\ast}$ will be $%
\lambda^{\ast}=\pi_{\lambda}F(z)$. In fact any point sufficiently close to $%
\pi_{\lambda}F(z)$ will be also good.

Intuitively, a function satisfies covering conditions if the coordinates are
topologically correctly aligned with the dynamics. The $D^{+}_{\lambda}$
plays the role of the topological exit set, and $D^{-}_{\lambda}$ of
topological entry.

\subsection{Establishing center unstable manifolds for maps\label%
{sec:wcu-maps-results}}

In this section we present a theorem which can be used to establish
existence of center unstable manifolds for maps.

\begin{theorem}
\cite[Theorem 16 + Remark 62]{Geom} \label{th:cu-maps}Let $k\geq1$ and $%
F:D\rightarrow\Lambda\times\mathbb{R}^{u}\times\mathbb{R}^{s}$ be a $C^{k+1}$
map. If $F$ satisfies covering conditions, rate conditions of order $k$ and
backward cone conditions, then $W^{cu}$ is a $C^{k}$ manifold, which are
graphs of a $C^{k}$ function 
\begin{equation*}
w^{cu}:\Lambda\times\overline{B}_{u}(R)\rightarrow\overline{B}_{s}(R),
\end{equation*}
meaning that 
\begin{equation*}
W^{cu}=\left\{ \left( \lambda,x,w^{cu}(\lambda,y)\right) :\lambda\in
\Lambda,x\in\overline{B}_{u}(R)\right\} .
\end{equation*}
Moreover, $w^{cu}$ is Lipschitz with constant $L$.

The manifold $W^{cu}$ is foliated by invariant fibers $W_{z}^{u}$, which are
graphs of $C^{k}$ functions%
\begin{equation*}
w_{z}^{u}:\overline{B}_{u}(R)\rightarrow\Lambda\times\overline{B}_{s}(R),
\end{equation*}
meaning that%
\begin{equation*}
W_{z}^{u}=\left\{ \left( \pi_{\lambda}w_{z}^{u}\left( x\right) ,x,\pi
_{y}w_{z}^{u}\left( x\right) \right) :x\in\overline{B}_{u}(R)\right\} .
\end{equation*}
The functions $w_{z}^{u}$ are Lipschitz with constants $1/L$. Moreover, for $%
C=2R\left( 1+1/L\right) ,$ 
\begin{align}
W_{z}^{u} & =\{p\in W^{cu}:\text{ }F^{-n}(p)\in D\text{ for all }n\in\mathbb{%
N}\text{,}  \label{eq:Wuz-convergence-maps} \\
& \left. \left\Vert F^{-n}\left( p\right) -F^{-n}\left( z\right) \right\Vert
\leq C\xi_{u,1,P}^{-n}\text{ for all }n\in\mathbb{N}\right\} .  \notag
\end{align}
\end{theorem}

Observe that bound on $L\in\left( \frac{2R}{R_{\Lambda}},1\right) $ gives us
lower bounds for the Lipschitz constants for functions $w^{cu}$, $w^{u}$,
which is clearly an overestimate for the case when $\mathbb{T}\times
\{0\}\times\{0\}$ is an invariant manifold. This lower bound is a
consequence of choices we have made when formulating Theorem \ref{th:cu-maps}%
, as we did not want to introduce different constants for each type of
cones, plus several inequalities between them. However, below theorem gives
conditions which allow to obtain better Lipschitz constants.

\begin{theorem}
\label{th:wuz}\cite[Theorem 18]{Geom} Let $M\in(0,1/L)$ and 
\begin{align*}
\xi & =\inf_{z\in D}m\left( \left[ \frac{\partial F_{x}}{\partial x}(P(z))%
\right] \right) -M\sup_{z\in D}\left\Vert \frac{\partial F_{x}}{%
\partial\left( \lambda,y\right) }(z)\right\Vert , \\
\mu & =\sup_{z\in D}\left\{ \left\Vert \frac{\partial F_{\left(
\lambda,y\right) }}{\partial\left( \lambda,y\right) }(z)\right\Vert +\frac{1%
}{M}\left\Vert \frac{\partial F_{\left( \lambda,y\right) }}{\partial x}%
(z)\right\Vert \right\} .
\end{align*}
If assumptions of Theorem \ref{th:cu-maps} hold true and also $\frac{\xi}{%
\mu }>1$, then the function $w_{z}^{u}$ from Theorem \ref{th:cu-maps} is
Lipschitz with constant $M.$
\end{theorem}

\begin{theorem}
\label{th:wcu}\cite[Theorem 19]{Geom} Let $M\in(0,L)$ and 
\begin{align*}
\xi & =\inf_{z\in D}m\left[ \frac{\partial F_{(\lambda,x)}}{\partial
(\lambda,x)}(P(z))\right] -M\sup_{z\in D}\left\Vert \frac{\partial
F_{(\lambda,x)}}{\partial y}(z)\right\Vert , \\
\mu & =\sup_{z\in D}\left\{ \left\Vert \frac{\partial F_{y}}{\partial y}%
\left( z\right) \right\Vert +\frac{1}{M}\left\Vert \frac{\partial F_{y}}{%
\partial(\lambda,x)}(z)\right\Vert \right\} .
\end{align*}
If assumptions of Theorem \ref{th:cu-maps} hold true and also $\frac{\xi}{%
\mu }>1$, then the function $w^{cu}$ from Theorem \ref{th:cu-maps} is
Lipschitz with constant $M$.
\end{theorem}

In our situation the map $F$ will be a time shift along a trajectory of an
ODE, which is invertible. We can apply Theorem \ref{th:cu-maps} to $F^{-1}$,
(reversing the roles of coordinates $x,y$) and thus obtain the bounds for
the center-stable manifold. The intersection of the center-stable manifold
with the center-unstable manifold is the normally hyperbolic invariant
manifold.

\subsection{Dependence of manifolds on parameters\label{sec:maps-param}}

We consider a family of maps $F_{\varepsilon }:D\rightarrow \Lambda \times 
\mathbb{R}^{u}\times \mathbb{R}^{s}$ with $\varepsilon \in E$. For
simplicity, we assume that $E=\mathbb{S}^{1}$. We can apply Theorem \ref%
{th:cu-maps} to each of the maps separately and obtain a family of functions 
$w_{\varepsilon }^{cu}$ and $w_{z,\varepsilon }^{u}$ for $\varepsilon \in E$%
. We can also extend the map to include the parameter as follows. We first
define $\tilde{\Lambda}=\mathbb{S}^{1}\times \Lambda $ and $\tilde{D}=\tilde{%
\Lambda}\times \overline{B}_{u}(R)\times \overline{B}_{s}(R)$ and consider%
\begin{equation*}
F:\tilde{D}\rightarrow \tilde{\Lambda}\times \mathbb{R}^{u}\times \mathbb{R}%
^{s},
\end{equation*}%
defined as%
\begin{equation*}
F\left( \varepsilon ,\lambda ,x,y\right) =\left( \varepsilon ,F_{\varepsilon
}\left( \lambda ,x,y\right) \right) .
\end{equation*}%
We can then apply Theorem \ref{th:cu-maps} to $F$. This will establish
existence of a center unstable manifold parameterized by%
\begin{equation*}
w^{cu}:\tilde{\Lambda}\times \overline{B}_{u}(R)\rightarrow \overline{B}%
_{s}(R).
\end{equation*}%
Theorem \ref{th:cu-maps} establishes that $w^{cu}$ is Lipschitz with
constant $L$. This means that for any $\left( \lambda ,x\right) \in \Lambda
\times \overline{B}_{u}(R)$ and any $\varepsilon _{1},\varepsilon _{2}\in E$
we have%
\begin{equation*}
\left\Vert w_{\varepsilon _{1}}^{cu}\left( \lambda ,x\right) -w_{\varepsilon
_{2}}^{cu}\left( \lambda ,x\right) \right\Vert =\left\Vert w^{cu}\left(
\varepsilon _{1},\lambda ,x\right) -w^{cu}\left( \varepsilon _{2},\lambda
,x\right) \right\Vert \leq L\left\Vert \varepsilon _{1}-\varepsilon
_{2}\right\Vert .
\end{equation*}%
If assumptions of Theorem \ref{th:cu-maps} are applied with $k>1$, then we
know that $w^{cu}$ is $C^{1}$, and above inequality gives us the following
dependence with respect to the parameter 
\begin{equation*}
\left\Vert \frac{\partial }{\partial \varepsilon }w_{\varepsilon
}^{cu}\left( \lambda ,x\right) \right\Vert \leq L.
\end{equation*}

Extending the $\Lambda$ to include the parameter can also be used to
establish bounds on the second or mixed derivative of $w_{\varepsilon}^{cu}$
with respect to the parameter, by using the method given in section \ref%
{sec:wcu-maps-d2-bounds} below.

\subsection{Bounds on second derivatives\label{sec:wcu-maps-d2-bounds}}

In this section we shall show how we can obtain explicit bounds on the
second derivatives of the parameterization of the center unstable manifold
established in Theorem \ref{th:cu-maps}.

For the sake of simplicity, we shall use two coordinates $\mathrm{x}$ and $%
\mathrm{y}$. We shall study the bounds on the second derivative of a
function $\mathrm{y}=w\left( \mathrm{x}\right) $ under appropriate rate
conditions. In applications, we can have:

\begin{itemize}
\item $\mathrm{x}=x,$ $\mathrm{y}=\left( \lambda,y\right) $ and $w(\mathrm{x}%
)=w_{z}^{u}\left( \mathrm{x}\right) $;

\item $\mathrm{x}=\left( \lambda,x\right) ,$ $\mathrm{y}=y$ and $w(\mathrm{x}%
)=w^{cu}\left( \mathrm{x}\right) .$
\end{itemize}

Similarly, in the case of a family of maps, which depend on parameters (as
discussed in Section \ref{sec:maps-param}), we can have:

\begin{itemize}
\item $\mathrm{x}=x,$ $\mathrm{y}=\left( \varepsilon,\lambda,y\right) $ and $%
w(\mathrm{x})=w_{z}^{u}\left( \mathrm{x}\right) $;

\item $\mathrm{x}=\left( \varepsilon,\lambda,x\right) ,$ $\mathrm{y}=y$ and $%
w(\mathrm{x})=w^{cu}\left( \mathrm{x}\right) .$
\end{itemize}

We shall assume that $(\mathrm{x},\mathrm{y})\in\mathbb{R}^{u+s}$ and
consider $F:D\rightarrow\mathbb{R}^{u+s}$ which is $C^{3}$ differentiable,
where $D\subset\mathbb{R}^{u+s}$ is the domain of $F$.

We assume that $F$ is such that if $v:\mathbb{R}^{u}\rightarrow \mathbb{R}%
^{s}$ is Lipschitz with constant $\mathcal{L}>0$, then the graph transform $%
\mathcal{G}\left( v\right) $ is well defined i.e.%
\begin{equation}
\mathcal{G}\left( v\right) =F_{y}\circ \left( \mathrm{id},v\right) \circ
\left( F_{x}\circ \left( \mathrm{id},v\right) \right) ^{-1}.
\label{eq:G-limit}
\end{equation}%
Assume also that for $v_{0}\left( \mathrm{x}\right) =0$ 
\begin{equation}
w=\lim_{n\rightarrow \infty }\mathcal{G}^{n}\left( v_{0}\right) .
\label{eq:w-G-limit}
\end{equation}%
Such is the setting in the construction of $w=w^{cu}$ and $w=w_{z}^{u}$ in 
\cite{Geom}. In such case, the property (\ref{eq:w-G-limit}) follows from
assumptions of Theorem \ref{th:cu-maps}; see \cite[Lemma 46]{Geom} and \cite[%
Lemma 57]{Geom}. In the case of $w=w^{cu}$ we take $\mathcal{L}=L$ (where $L$
is the constant from Theorem \ref{th:cu-maps}) and for $w=w_{z}^{u}$ we take 
$\mathcal{L}=1/L$. The following result will allow us to obtain estimates on
the second derivative of $w$.

\begin{theorem}
\label{th:unstable-cones}Let $\mathcal{L}>0$ and define 
\begin{align*}
\xi & =\inf_{z\in D}\left( m\left( \frac{\partial F_{\mathrm{x}}}{\partial 
\mathrm{x}}(z)\right) -\mathcal{L}\left\Vert \frac{\partial F_{\mathrm{x}}}{%
\partial \mathrm{y}}(z)\right\Vert \right) , \\
\mu _{1}& =\sup_{z\in D}\left( \left\Vert \frac{\partial F_{\mathrm{y}}}{%
\partial \mathrm{y}}(z)\right\Vert +\frac{1}{\mathcal{L}}\left\Vert \frac{%
\partial F_{\mathrm{y}}}{\partial \mathrm{x}}(z)\right\Vert \right) , \\
\mu _{2}& =\sup_{z\in D}\left( \left\Vert \frac{\partial F_{\mathrm{y}}}{%
\partial \mathrm{y}}(z)\right\Vert +\mathcal{L}\left\Vert \frac{\partial F_{%
\mathrm{x}}}{\partial \mathrm{y}}(z)\right\Vert \right) .
\end{align*}%
Assume that \footnote{%
We have the following link with the rate conditions from Definition \ref%
{def:rate-conditions}: When $\mathrm{x}=\left( \lambda ,x\right) $ and $%
\mathrm{y}=y$, then we take $\mathcal{L}=L$ and see that $\xi =\xi
_{cu,1}\geq \xi _{cu,1,P}$, $\mu _{1}=\mu _{s,1}$ and $\mu _{2}=\mu _{s,2}$.
Hence (\ref{eq:ratecond}) follows from the rate conditions:
\par
\begin{equation*}
\frac{\mu _{1}}{\xi }=\frac{\mu _{s,1}}{\xi _{cu,1}}\leq \frac{\mu _{s,1}}{%
\xi _{cu,1,P}}<1,\qquad \frac{\mu _{2}}{\xi ^{2}}=\frac{\mu _{s,2}}{\xi
_{cu,1}^{2}}<1.
\end{equation*}%
\par
\noindent Similarly, for $\mathrm{x}=x$ and $\mathrm{y}=\left( \lambda
,y\right) $, we consider $\mathcal{L}=1/L$. Then $\xi =\xi _{u,1}\geq \xi
_{u,1,P},$ $\mu _{1}=\mu _{cs,1}$, $\mu _{2}=\mu _{cs,2},$ and (\ref%
{eq:ratecond}) also follows from the rate conditions in a similar way.}%
\begin{equation}
\xi >0,\qquad \frac{\mu _{1}}{\xi }<1,\qquad \frac{\mu _{2}}{\xi ^{2}}<1.
\label{eq:ratecond}
\end{equation}%
Let 
\begin{align}
C_{x}& =\frac{1}{2}\max_{p\in D,\Vert h\Vert =1}\Vert D^{2}F_{\mathrm{x}%
}(p)(h,h)\Vert ,  \label{eq:c_x} \\
C_{y}& =\frac{1}{2}\max_{p\in D,\Vert h\Vert =1}\Vert D^{2}F_{\mathrm{y}%
}(p)(h,h)\Vert .  \label{eq:c_y}
\end{align}%
and 
\begin{align}
C_{y,1}& =\sup_{p\in D}\frac{1}{2}\left\Vert \frac{\partial ^{2}F_{\mathrm{y}%
}}{\partial \mathrm{x}^{2}}(p)\right\Vert ,  \notag \\
C_{y,2}& =\sup_{p\in D}\left\Vert \frac{\partial ^{2}F_{\mathrm{y}}}{%
\partial \mathrm{x}\partial \mathrm{y}}(p)\right\Vert ,  \label{eq:c_yi} \\
C_{y,3}& =\sup_{p\in D}\frac{1}{2}\left\Vert \frac{\partial ^{2}F_{\mathrm{y}%
}}{\partial \mathrm{y}^{2}}(p)\right\Vert .  \notag
\end{align}%
Then for any $\mathrm{x}$ and $h$ holds (where it makes sense) where $w$ is
defined by (\ref{eq:w-G-limit}) 
\begin{equation}
w(\mathrm{x}+h)=w(\mathrm{x})+Dw(\mathrm{x})h+\Delta y(\mathrm{x},h),\quad
\Vert \Delta y(\mathrm{x},h)\Vert \leq M\Vert h\Vert ^{2}
\label{eq:c2-bound-jet-claim}
\end{equation}%
where 
\begin{equation}
M>\frac{(\mathcal{L}C_{x}+C_{y})(1+\mathcal{L}^{2})}{\xi ^{2}-\mu _{2}}.
\label{eq:M}
\end{equation}%
One can obtain an alternative (giving tighter estimates; see Remark \ref%
{rem:flat-manif}) expression for $M$ 
\begin{equation}
M>\frac{\mathcal{L}C_{x}(1+\mathcal{L}^{2})+C_{y,1}+C_{y,2}\mathcal{L}%
+C_{y,3}\mathcal{L}^{2}}{\xi ^{2}-\mu _{2}}  \label{eq:M-improved}
\end{equation}%
Hence for any $h\in \mathbb{R}^{u}$ holds 
\begin{equation*}
\left\Vert \frac{1}{2}D^{2}w\left( \mathrm{x}\right) (h,h)\right\Vert \leq
M\Vert h\Vert ^{2}.
\end{equation*}
\end{theorem}

\begin{proof}
For a $s\times u$ matrix $A$, $M\in\mathbb{R},$ and a point $z\in \mathbb{R}%
^{u+s}$ we define a set (see Figure \ref{fig:c2-cone}) 
\begin{figure}[ptb]
\begin{center}
\includegraphics[height=4cm]{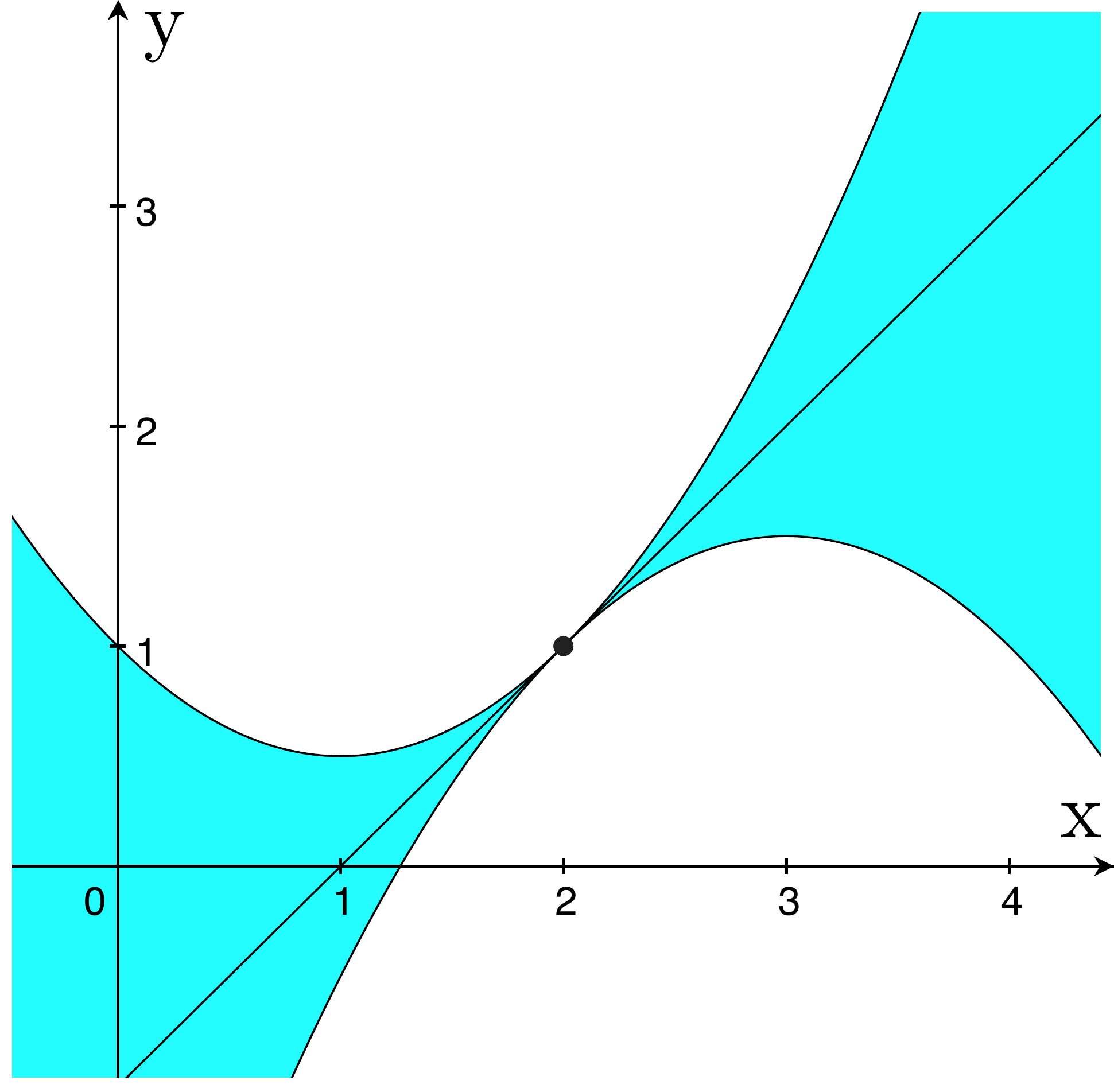}
\end{center}
\caption{The set $J_{u}(z,A,M)$ for $z=\left( 2,1\right) $, $A=1$ and $M=%
\frac{1}{2}$, in blue.}
\label{fig:c2-cone}
\end{figure}
\begin{equation}
J_{u}(z,A,M)=\{z+(\mathrm{x},A\mathrm{x}+\mathrm{y}):\quad\Vert\mathrm{y}%
\Vert\leq M\Vert\mathrm{x}\Vert^{2}\}.  \label{eq:J2-def}
\end{equation}

We shall look for the smallest $M$, such that for all $z\in D$ and all $%
\Vert A_{0}\Vert\leq\mathcal{L}$ there exists a $A_{1}$ such that $\Vert
A_{1}\Vert\leq\mathcal{L}$ and%
\begin{equation}
F(J_{u}(z,A_{0},M))\cap B(0,\delta))\subset J_{u}(F(z),A_{1},M),
\label{eq:jet-inclusion-c2-maps}
\end{equation}
for sufficiently small $\delta>0$ (which might depend on $z$ and $A_{0}$).

From now on we assume that $\Vert A_{0}\Vert\leq\mathcal{L}$.

Let us set 
\begin{align*}
D_{\mathrm{x}} & =\frac{\partial F_{\mathrm{x}}}{\partial\mathrm{x}}(z)+%
\frac{\partial F_{\mathrm{x}}}{\partial\mathrm{y}}(z)A_{0}, \\
D_{\mathrm{y}} & =\frac{\partial F_{\mathrm{y}}}{\partial\mathrm{x}}(z)+%
\frac{\partial F_{\mathrm{y}}}{\partial\mathrm{y}}(z)A_{0}.
\end{align*}

Observe that by the definition of $\xi$ 
\begin{equation}
m(D_{\mathrm{x}})\geq m\left( \frac{\partial F_{\mathrm{x}}}{\partial 
\mathrm{x}}(z)\right) -\left\Vert \frac{\partial F_{\mathrm{x}}}{\partial%
\mathrm{y}}(z)\right\Vert \left\Vert A_{0}\right\Vert \geq \xi>0.
\label{eq:mDx}
\end{equation}

We take 
\begin{equation*}
A_{1}=D_{\mathrm{y}}D_{\mathrm{x}}^{-1}.
\end{equation*}
From (\ref{eq:ratecond}) follows that if $\Vert A_{0}\Vert\leq\mathcal{L}$,
then 
\begin{equation}
\left\Vert A_{1}\right\Vert =\left\Vert D_{\mathrm{y}}D_{\mathrm{x}%
}^{-1}\right\Vert \leq\frac{\left\Vert \frac{\partial F_{\mathrm{y}}}{%
\partial\mathrm{x}}(z)\right\Vert +\left\Vert \frac{\partial F_{\mathrm{y}}}{%
\partial\mathrm{y}}(z)\right\Vert \left\Vert A_{0}\right\Vert }{m\left( 
\frac{\partial F_{\mathrm{x}}}{\partial\mathrm{x}}(z)\right) -\left\Vert 
\frac{\partial F_{\mathrm{x}}}{\partial\mathrm{y}}(z)\right\Vert \left\Vert
A_{0}\right\Vert }\leq\frac{\mathcal{L}\mu_{1}}{\xi}<\mathcal{L}.
\label{eq:A1-lip-L}
\end{equation}

Let $h=\left( \mathrm{x},A_{0}\mathrm{x}+\mathrm{y}\right) $. For $z+h\in
J_{u}(z,A,M)$, by (\ref{eq:J2-def}), we have 
\begin{equation*}
\Vert A_{0}\mathrm{x}+\mathrm{y}\Vert\leq\Vert A_{0}\Vert\Vert\mathrm{x}%
\Vert+M\Vert\mathrm{x}\Vert^{2}\leq\Vert\mathrm{x}\Vert\left( \mathcal{L}%
+M\Vert\mathrm{x}\Vert\right) .
\end{equation*}
Let $(\mathrm{x}_{1},\mathrm{\tilde{y}})=F(z+h)-F(z)$ and let $\mathrm{y}%
_{1}=\mathrm{\tilde{y}}-A_{1}\mathrm{x}_{1}$. Note that%
\begin{equation*}
F(z+h)=F(z)+\left( \mathrm{x}_{1},A_{1}\mathrm{x}_{1}+\mathrm{y}_{1}\right) .
\end{equation*}
Our goal will be to find a bound on $\frac{\Vert\mathrm{y}_{1}\Vert}{\Vert%
\mathrm{x}_{1}\Vert^{2}}$, and to show that $\frac{\Vert\mathrm{y}_{1}\Vert}{%
\Vert\mathrm{x}_{1}\Vert^{2}}\leq M.$ First we need to establish a number of
estimates.

We have 
\begin{align}
\mathrm{x}_{1} & =F_{\mathrm{x}}(z+h)-F_{\mathrm{x}}(z)  \notag \\
& =\frac{\partial F_{\mathrm{x}}}{\partial\mathrm{x}}(z)\mathrm{x}+\frac{%
\partial F_{\mathrm{x}}}{\partial\mathrm{y}}(z)\left( A_{0}\mathrm{x}+%
\mathrm{y}\right) +R_{\mathrm{x},2}\left( z,h\right)  \notag \\
& =D_{\mathrm{x}}\mathrm{x}+\frac{\partial F_{\mathrm{x}}}{\partial\mathrm{y}%
}(z)\mathrm{y}+R_{\mathrm{x},2}\left( z,h\right) .  \label{eq:x-bound-c2}
\end{align}
From (\ref{eq:c_x}) we know that 
\begin{equation}
\left\Vert R_{\mathrm{x},2}\left( z,h\right) \right\Vert \leq
C_{x}\left\Vert h\right\Vert ^{2}\leq C_{x}\left( \left\Vert \mathrm{x}%
\right\Vert ^{2}+\left\Vert A_{0}\mathrm{x}+\mathrm{y}\right\Vert
^{2}\right) \leq R_{x}\Vert\mathrm{x}\Vert^{2},  \label{eq:Rx2-bound-tmp}
\end{equation}
where 
\begin{equation*}
R_{x}\leq C_{x}(1+(\mathcal{L}+M\Vert\mathrm{x}\Vert)^{2}).
\end{equation*}
Thus%
\begin{equation*}
\left\Vert \mathrm{x}_{1}-\left( D_{\mathrm{x}}\mathrm{x}+\frac{\partial F_{%
\mathrm{x}}}{\partial\mathrm{y}}(z)\mathrm{y}\right) \right\Vert \leq
R_{x}\Vert\mathrm{x}\Vert^{2}.
\end{equation*}

Using mirror computations, for $\mathrm{\tilde{y}}$ from (\ref{eq:c_y}) we
obtain 
\begin{equation*}
\left\Vert \mathrm{\tilde{y}}-\left( D_{\mathrm{y}}\mathrm{x}+\frac{\partial
F_{\mathrm{y}}}{\partial\mathrm{y}}(z)\mathrm{y}\right) \right\Vert \leq
R_{y}\Vert\mathrm{x}\Vert^{2},
\end{equation*}
with 
\begin{equation}
R_{y}\leq C_{y}\left( 1+(\mathcal{L}+M\Vert\mathrm{x}\Vert)^{2}\right) .
\label{eq:Ry1}
\end{equation}

We have another possible variants for $R_{y}$ based on (\ref{eq:c_yi}). We
can compute 
\begin{equation}
\mathrm{\tilde{y}}=D_{\mathrm{y}}\mathrm{x}+\frac{\partial F_{\mathrm{y}}}{%
\partial\mathrm{y}}(z)\mathrm{y}+R_{\mathrm{y},2}\left( z,h\right) ,
\label{eq:y-bound-c2}
\end{equation}
with the estimate%
\begin{align*}
\left\Vert R_{\mathrm{y},2}\left( z,h\right) \right\Vert & \leq\sup_{p\in D}%
\frac{1}{2}\left\Vert \frac{\partial^{2}F_{\mathrm{y}}}{\partial \mathrm{x}%
^{2}}(p)\right\Vert \left\Vert \mathrm{x}\right\Vert ^{2}+\sup_{p\in
D}\left\Vert \frac{\partial^{2}F_{\mathrm{y}}}{\partial\mathrm{x}\partial%
\mathrm{y}}(p)\right\Vert \left\Vert \mathrm{x}\right\Vert \left\Vert 
\mathrm{y}\right\Vert \\
& +\sup_{p\in D}\frac{1}{2}\left\Vert \frac{\partial^{2}F_{\mathrm{y}}}{%
\partial\mathrm{y}^{2}}(p)\right\Vert \left\Vert \mathrm{y}\right\Vert ^{2}
\\
& \leq C_{y,1}\left\Vert \mathrm{x}\right\Vert ^{2}+C_{y,2}\left\Vert 
\mathrm{x}\right\Vert \left\Vert A_{0}\mathrm{x}+\mathrm{y}\right\Vert
+C_{y,3}\left\Vert A_{0}\mathrm{x}+\mathrm{y}\right\Vert \\
& \leq\left\Vert \mathrm{x}\right\Vert ^{2}R_{y}^{(2)},
\end{align*}
for%
\begin{equation}
R_{y}^{(2)}\leq C_{y,1}+C_{y,2}(\mathcal{L}+M\Vert\mathrm{x}\Vert )+C_{y,3}(%
\mathcal{L}+M\Vert\mathrm{x}\Vert)^{2}.  \label{eq:Ry2}
\end{equation}

To compute the bound for $\frac{\Vert\mathrm{y}_{1}\Vert}{\Vert\mathrm{x}%
_{1}\Vert^{2}}$ we must ensure that $\left\Vert \mathrm{x}_{1}\right\Vert
\neq0$. From (\ref{eq:x-bound-c2}) and (\ref{eq:Rx2-bound-tmp})%
\begin{equation}
\Vert\mathrm{x}_{1}\Vert\geq\Vert\mathrm{x}\Vert\left( m(D_{x})-\left\Vert 
\frac{\partial F_{\mathrm{x}}}{\partial\mathrm{y}}\right\Vert M\Vert \mathrm{%
x}\Vert-R_{x}\Vert\mathrm{x}\Vert\right) .  \label{eq:x1-nonzero}
\end{equation}
Since by (\ref{eq:mDx}) $m(D_{x})>0,$ we thus see that for sufficiently
small $\left\Vert \mathrm{x}\right\Vert $ (how small may depend on $M$) we
shall have $\Vert\mathrm{x}_{1}\Vert>0.$

From (\ref{eq:x-bound-c2}) we have%
\begin{equation*}
\mathrm{x}=D_{\mathrm{x}}^{-1}\mathrm{x}_{1}-D_{\mathrm{x}}^{-1}\frac{%
\partial F_{\mathrm{x}}}{\partial\mathrm{y}}(z)\mathrm{y}-D_{\mathrm{x}%
}^{-1}R_{\mathrm{x},2}\left( z,h\right)
\end{equation*}
hence by (\ref{eq:y-bound-c2})%
\begin{align*}
\mathrm{\tilde{y}} & =D_{\mathrm{y}}\left( D_{\mathrm{x}}^{-1}\mathrm{x}%
_{1}-D_{\mathrm{x}}^{-1}\frac{\partial F_{\mathrm{x}}}{\partial\mathrm{y}}(z)%
\mathrm{y}-D_{\mathrm{x}}^{-1}R_{\mathrm{x},2}\left( z,h\right) \right) \\
& +\frac{\partial F_{\mathrm{y}}}{\partial\mathrm{y}}(z)\mathrm{y}+R_{%
\mathrm{y},2}\left( z,h\right) \\
& =A_{1}x_{1}+\left( -A_{1}\frac{\partial F_{\mathrm{x}}}{\partial\mathrm{y}}%
(z)+\frac{\partial F_{\mathrm{y}}}{\partial\mathrm{y}}(z)\right) \mathrm{y}
\\
& -A_{1}R_{\mathrm{x},2}\left( z,h\right) +R_{\mathrm{y},2}\left( z,h\right)
\end{align*}
which by (\ref{eq:A1-lip-L}) and (\ref{eq:Rx2-bound-tmp}) gives%
\begin{align}
\left\Vert \mathrm{y}_{1}\right\Vert & =\left\Vert \mathrm{\tilde{y}}-A_{1}%
\mathrm{x}_{1}\right\Vert  \label{eq:y1-norm-c2} \\
& \leq\left\Vert \left( -A_{1}\frac{\partial F_{\mathrm{x}}}{\partial 
\mathrm{y}}(z)+\frac{\partial F_{\mathrm{y}}}{\partial\mathrm{y}}(z)\right) 
\mathrm{y}\right\Vert  \notag \\
& +\Vert A_{1}\Vert R_{x}\Vert\mathrm{x}\Vert^{2}+R_{y}\Vert\mathrm{x}%
\Vert^{2}  \notag \\
& \leq\mu_{2}\left\Vert \mathrm{y}\right\Vert +\mathcal{L}R_{x}\Vert \mathrm{%
x}\Vert^{2}+R_{y}\Vert\mathrm{x}\Vert^{2}  \notag \\
& \leq\left\Vert \mathrm{x}\right\Vert ^{2}\left( \mu_{2}M+\mathcal{L}%
R_{x}+R_{y}\right) .  \notag
\end{align}
Since by (\ref{eq:mDx}) $m(D_{\mathrm{x}})>\xi$ by combining (\ref%
{eq:x1-nonzero}) and (\ref{eq:y1-norm-c2}) we obtain 
\begin{align*}
\frac{\Vert y_{1}\Vert}{\Vert x_{1}\Vert^{2}} & \leq\frac{\left\Vert \mathrm{%
x}\right\Vert ^{2}\left( \mu_{2}M+\mathcal{L}R_{x}+R_{y}\right) }{\Vert%
\mathrm{x}\Vert^{2}\left( \xi-\left\Vert \frac{\partial F_{\mathrm{x}}}{%
\partial\mathrm{y}}\right\Vert M\Vert\mathrm{x}\Vert-R_{x}\Vert \mathrm{x}%
\Vert\right) ^{2}} \\
& =M\frac{\frac{\mu_{2}}{\xi^{2}}+\frac{1}{\xi^{2}M}\left( \mathcal{L}%
R_{x}+R_{y}\right) }{\left( 1-\frac{\Vert\mathrm{x}\Vert}{\xi}\left(
\left\Vert \frac{\partial F_{\mathrm{x}}}{\partial\mathrm{y}}\right\Vert
M+R_{x}\right) \right) ^{2}}.
\end{align*}
We want this ratio to be less than $M$ for sufficiently small $\Vert x\Vert$%
. Therefore we can set $\Vert x\Vert=0$, so we obtain the following
condition 
\begin{equation*}
\frac{\mu_{2}}{\xi^{2}}+\frac{1}{\xi^{2}M}\left( \mathcal{L}%
R_{x}+R_{y}\right) <1.
\end{equation*}
This condition follows from (\ref{eq:M}) for $R_{y}$ given by (\ref{eq:Ry1}%
). For $R_{y}=R_{y}^{(2)}$, where $R_{y}^{(2)}$ was defined in (\ref{eq:Ry2}%
), above condition follows from (\ref{eq:M-improved}).

By our assumption (\ref{eq:w-G-limit}), we know that $w=\lim_{n\rightarrow
+\infty}\mathcal{G}^{n}\left( v_{0}\right) $. Taking $A_{0}=0$ we see that
for any $z\in\mathrm{graph}\left( v_{0}\right) $ and for sufficiently small $%
\delta$%
\begin{equation*}
\mathrm{graph}\left( v_{0}\right) \cap B\left( z,\delta\right) \subset
J_{u}\left( z,A_{0},M\right) .
\end{equation*}
By (\ref{eq:jet-inclusion-c2-maps}),%
\begin{equation*}
\mathrm{graph}\left( \mathcal{G}\left( v_{0}\right) \right) \cap B\left(
z,\delta\right) \subset J_{u}\left( F(z),A_{1},M\right) .
\end{equation*}
Applying this argument inductively, for any $z\in\mathrm{graph}\left(
w\right) ,$ taking $A=Dw\left( z\right) $, 
\begin{equation}
\mathrm{graph}\left( w\right) \cap B\left( z,\delta\right) \subset
J_{u}\left( z,A,M\right) .  \label{eq:graph-jet-c2}
\end{equation}
From (\ref{eq:graph-jet-c2}) follows (\ref{eq:c2-bound-jet-claim}), which
concludes our proof.
\end{proof}

\begin{remark}
\label{rem:flat-manif}Observe that in the case of totally flat invariant
manifold $($\textrm{$x$}$,0)$ we have $F_{\mathrm{y}}(\mathrm{x},\mathrm{y}%
)=g(\mathrm{x},\mathrm{y})\mathrm{y}$ and $\mathcal{L}$ could be taken as
small as we want.

In such case we obtain from (\ref{eq:M}) the bound $M\approx\frac {C_{%
\mathrm{y}}}{\xi^{2}-\mu}$, which might be quite large as $C_{\mathrm{y}}$
depends on $\frac{\partial^{2}F_{\mathrm{y}}}{\partial\mathrm{x}\partial%
\mathrm{y}}$ and $\frac{\partial^{2}F_{\mathrm{y}}}{\partial \mathrm{y}^{2}}%
, $ which might be nonzero even on our flat manifold.

When using (\ref{eq:M-improved}) we obtain $M\approx\frac{C_{\mathrm{y},1}}{%
\xi^{2}-\mu}$, where $C_{\mathrm{y},1}$ is to be expected to be very small,
because $D_{\mathrm{x}}^{2}F_{\mathrm{y}}(\mathrm{x},\mathrm{y})=(D_{\mathrm{%
x}}^{2}g(\mathrm{x},\mathrm{y}))\mathrm{y}$, hence it vanishes on the
invariant manifold.
\end{remark}

Using Theorem \ref{th:unstable-cones} we can obtain estimates on the partial
derivatives of $w$ using the following lemma.

\begin{lemma}
Assume that $\Vert D^{2}w\left( \mathrm{x}\right) (h,h)\Vert\leq2M\Vert
h\Vert^{2}$. Then in orthogonal coordinates $(\mathrm{x}_{1},\dots ,\mathrm{x%
}_{n})$ holds 
\begin{equation*}
\left\Vert \frac{\partial^{2}w}{\partial\mathrm{x}_{i}\partial\mathrm{x}_{j}}%
\right\Vert \leq2M,\quad i,j=1,\dots,n.
\end{equation*}
\end{lemma}

\begin{proof}
Let us denote by $W$ the symmetric map $D^{2}w$. Let $e_{1},\dots,e_{n}$ be
a basis corresponding our coordinates. Then 
\begin{equation*}
W(e_{i},e_{j})=\frac{\partial^{2}F}{\partial\mathrm{x}_{i}\partial \mathrm{x}%
_{j}}.
\end{equation*}
Our task is to recover the map $W$ knowing only the behavior on the
diagonal. This is accomplished using the following identity 
\begin{equation*}
W(p+q,p+q)-W(p-q,p-q)=4W(p,q).
\end{equation*}
Let us set $p=e_{i}+e_{j}$ and $q=e_{i}-e_{j}$. Observe that $\Vert p\Vert
^{2}=\Vert q\Vert^{2}=2$. We have 
\begin{align*}
4\left\Vert W(e_{i},e_{j})\right\Vert & \leq\left\Vert
W(e_{i}+e_{j},e_{i}+e_{j})\right\Vert +\left\Vert
W(e_{i}-e_{j},e_{i}-e_{j})\right\Vert \\
& \leq2\cdot2M\sqrt{2}^{2},
\end{align*}
which concludes our proof.
\end{proof}


\section{Center-unstable manifolds for ODEs\label{sec:ODE-Wcu}}

In this section we show how to establish the existence of center unstable
manifolds for ODEs. The results will follow from the ones established for
maps in section \ref{sec:wcu-maps}. To obtain our results, we will consider
the time shift map along the solution of the ODE. Our objective though will
be to reformulate the conditions to obtain our results based on assumptions
on the vector field, rather than to integrate the ODE.

\subsection{Definitions and setup}

We consider an ODE%
\begin{equation}
q^{\prime }=f(q)  \label{eq:ode-Wcu}
\end{equation}%
where%
\begin{equation*}
f:\Lambda \times \mathbb{R}^{u}\times \mathbb{R}^{s}\rightarrow \mathbb{R}%
^{c}\times \mathbb{R}^{u}\times \mathbb{R}^{s}.
\end{equation*}%
We denote by $\Phi \left( t,q\right) $ the flow induced by (\ref{eq:ode-Wcu}%
).

We shall consider a set%
\begin{equation*}
D=\Lambda \times \overline{B}_{u}\left( R\right) \times \overline{B}%
_{s}\left( R\right) .
\end{equation*}

\begin{definition}
\label{def:Wcu-ode}We define the center-unstable set of (\ref{eq:ode-Wcu})
in $D$ as%
\begin{equation*}
W_{\mathrm{loc},D}^{cu}=\{z:\Phi \left( t,z\right) \in D\text{ for all }%
t<0\}.
\end{equation*}
\end{definition}

Since the set $D$ will remain fixed throughout the discussion, from now on
we will simplify notation by writing $W^{cu}$ instead of $W_{\mathrm{loc}%
,D}^{cu}.$

As in section \ref{sec:wcu-maps-setup}, we consider a constant $L\in \left( 
\frac{2R}{R_{\Lambda }},1\right) ,$ and define the stable fiber analogously
to Definition \ref{def:Wu-fiber}.

\begin{definition}
\label{def:ode-fiber}Assume that $z\in W^{cu}$. We define the unstable fiber
of $z$ as%
\begin{equation*}
W_{z}^{u}=\{p\in D:\Phi \left( t,p\right) \in J_{u}\left( \Phi \left(
t,z\right) ,1/L\right) \cap D,\text{ for all }t<0\}.
\end{equation*}
\end{definition}

Let us introduce the following constants (compare with constants from
section \ref{sec:wcu-maps-setup} for maps)%
\begin{align*}
\overrightarrow{\mu _{s,1}}& =\sup_{z\in D}\left\{ l\left( \frac{\partial
f_{y}}{\partial y}(z)\right) +\frac{1}{L}\left\Vert \frac{\partial f_{y}}{%
\partial (\lambda ,x)}(z)\right\Vert \right\} , \\
\overrightarrow{\mu _{s,2}}& =\sup_{z\in D}\left\{ l\left( \frac{\partial
f_{y}}{\partial y}\left( z\right) \right) +L\left\Vert \frac{\partial
f_{\left( \lambda ,x\right) }}{\partial y}(z)\right\Vert \right\} ,
\end{align*}%
\begin{align*}
\overrightarrow{\xi _{u,1}}& =\inf_{z\in D}\left\{ m_{l}\left( \frac{%
\partial f_{x}}{\partial x}(z)\right) -\frac{1}{L}\left\Vert \frac{\partial
f_{x}}{\partial \left( \lambda ,y\right) }(z)\right\Vert \right\} , \\
\overrightarrow{\xi _{u,1,P}}& =\inf_{z\in D}m_{l}\left( \frac{\partial f_{x}%
}{\partial x}(P(z))\right) -\frac{1}{L}\sup_{z\in D}\left\Vert \frac{%
\partial f_{x}}{\partial \left( \lambda ,y\right) }(z)\right\Vert ,
\end{align*}%
\begin{align}
\overrightarrow{\mu _{cs,1}}& =\sup_{z\in D}\left\{ l\left( \frac{\partial
f_{\left( \lambda ,y\right) }}{\partial \left( \lambda ,y\right) }(z)\right)
+L\left\Vert \frac{\partial f_{\left( \lambda ,y\right) }}{\partial x}%
(z)\right\Vert \right\} ,  \notag \\
\overrightarrow{\mu _{cs,2}}& =\sup_{z\in D}\left\{ l\left( \frac{\partial
f_{\left( \lambda ,y\right) }}{\partial \left( \lambda ,y\right) }(z)\right)
+\frac{1}{L}\left\Vert \frac{\partial f_{x}}{\partial \left( \lambda
,y\right) }(z)\right\Vert \right\} ,  \notag
\end{align}%
\begin{align*}
\overrightarrow{\xi _{cu,1}}& =\inf_{z\in D}\left\{ m_{l}\left( \frac{%
\partial f_{(\lambda ,x)}}{\partial (\lambda ,x)}(z)\right) -L\left\Vert 
\frac{\partial f_{(\lambda ,x)}}{\partial y}(z)\right\Vert \right\} , \\
\overrightarrow{\xi _{cu,1,P}}& =\inf_{z\in D}m_{l}\left( \frac{\partial
f_{(\lambda ,x)}}{\partial (\lambda ,x)}(P(z))\right) -L\sup_{z\in
D}\left\Vert \frac{\partial f_{(\lambda ,x)}}{\partial y}(z)\right\Vert .
\end{align*}%
The arrow is used to emphasize that the constants are computed for the
vector field.

Analogously to the case of maps (Definition \ref{def:rate-conditions}) we
define the rate conditions for ODEs as follows.

\begin{definition}
\label{def:rate-cond-ode}We say that the vector field $f$ satisfies rate
conditions of order $k\geq 1$ if for all $k\geq j\geq 1$ holds 
\begin{equation}
\overrightarrow{\mu _{s,1}}<0<\overrightarrow{\xi _{u,1,P}},
\label{eq:rate-cond-1-ode}
\end{equation}%
\begin{gather}
\overrightarrow{\mu _{cs,1}}<\overrightarrow{\xi _{u,1,P}},\qquad \qquad 
\overrightarrow{\mu _{s,1}}<\overrightarrow{\xi _{cu,1,P}},
\label{eq:rate-cond-2-ode} \\
\overrightarrow{\mu _{s,2}}<\left( j+1\right) \overrightarrow{\xi _{cu,1}}%
,\qquad \overrightarrow{\mu _{cs,2}}<\overrightarrow{\xi _{u,1}}.
\label{eq:rate-cond-3-ode}
\end{gather}
\end{definition}

We now define the notion of an isolating block.

\begin{definition}
\label{def:isolating-segment}We say that $D=\Lambda \times \overline{B}%
_{u}\left( R\right) \times \overline{B}_{s}\left( R\right) $ is an isolating
block for (\ref{eq:ode-Wcu}) if

\begin{enumerate}
\item For any $q\in \Lambda \times \partial \overline{B}_{u}\left( R\right)
\times \overline{B}_{s}\left( R\right) $, 
\begin{equation*}
\left( \pi _{x}f(q)|\pi _{x}q\right) >0.
\end{equation*}

\item For any $q\in \Lambda \times \overline{B}_{u}\left( R\right) \times
\partial \overline{B}_{s}\left( R\right) $,%
\begin{equation*}
\left( \pi _{y}f(q)|\pi _{y}q\right) <0.
\end{equation*}
\end{enumerate}
\end{definition}

Isolating blocks are important constructs in the Conley index theory \cite%
{MM}. Intuitively, in Definition \ref{def:isolating-segment} the set $%
\Lambda \times \partial \overline{B}_{u}\left( R\right) \times \overline{B}%
_{s}\left( R\right) $ plays the role of the exit set, and $\Lambda \times 
\overline{B}_{u}\left( R\right) \times \partial \overline{B}_{s}\left(
R\right) $ of the entry set. Isolating blocks will play the same role as the
covering condition for maps (Definition \ref{def:covering}).

\begin{theorem}
\label{th:main-ode}Let $k\geq 1$. Assume that $f$ is $C^{k+1}$ and satisfies
rate conditions of order $k$. Assume also that $D=\Lambda \times \overline{B}%
_{u}\left( R\right) \times \overline{B}_{s}\left( R\right) $ is an isolating
segment for $f$. Then the center-unstable set $W^{cu}$ in $D$ is a $C^{k}$
manifold, which satisfies the properties listed in Theorem \ref{th:cu-maps}.

The manifold $W^{cu}$ is foliated by invariant fibers $W_{z}^{u}$, which are
graphs of $C^{k}$ functions (as in Theorem \ref{th:cu-maps}). Moreover, for $%
C=2R\left( 1+1/L\right) ,$ \textbf{\ }%
\begin{eqnarray}
W_{z}^{u} &=&\{p\in W^{cu}:\text{ }\Phi \left( -t,p\right) \in D\text{ for
all }t>0\text{,}  \label{eq:Wuz-covergence-ode} \\
&&\left. \left\Vert \Phi \left( -t,p\right) -\Phi \left( -t,z\right)
\right\Vert \leq Ce^{-t\overrightarrow{\xi _{u,1,P}}}\text{ for all }%
t>0\right\} .  \notag
\end{eqnarray}
\end{theorem}

\begin{proof}
The proof is given in section \ref{sec:wcu-ode-proof}.
\end{proof}

The proof of Theorem \ref{th:main-ode} will follow from Theorem \ref%
{th:cu-maps}, applied to a time shift along the trajectory. In section \ref%
{subsec:ode-verif-cc-rc} we will show how rate conditions (for maps; as in
Definition \ref{def:rate-conditions}) follow from Definition \ref%
{def:rate-cond-ode} for the time shift map along the trajectory. In section %
\ref{def:rate-cond-ode} we will show how the covering condition (Definition %
\ref{def:covering}) follows from Definition \ref{def:isolating-segment}.
This will lead to the proof of Theorem \ref{th:main-ode} in section \ref%
{sec:wcu-ode-proof}.

\subsection{Verification of rate conditions\label{subsec:ode-verif-cc-rc}}

We consider an ODE%
\begin{equation}
z^{\prime }=f\left( z\right) ,  \label{eq:ode-rates}
\end{equation}%
where $z=\left( x,y\right) $ and $f=\left( f_{x},f_{y}\right) $. Consider a
shift by $h>0$ along the solution of (\ref{eq:ode-rates}), which we will
denote by $\Phi (h,z)=\left( \Phi _{x}(h,z),\Phi _{y}(h,z)\right) $. We will
show how to establish rate conditions for a map $F\left( z\right) =\Phi
(h,z) $, for sufficiently small (fixed) $h.$

The results obtained in this section will be applicable for the setting
where:

\begin{itemize}
\item $\mathrm{x}=x,$ $\mathrm{y}=\left( \lambda ,y\right) ,$

\item $\mathrm{x}=\left( \lambda ,x\right) ,$ $\mathrm{y}=y,$
\end{itemize}

Similarly, in the case of a family of maps (as discussed in Section \ref%
{sec:maps-param}), which depend on parameters, we can have:

\begin{itemize}
\item $\mathrm{x}=x,$ $\mathrm{y}=\left( \varepsilon ,\lambda ,y\right) $,

\item $\mathrm{x}=\left( \varepsilon ,\lambda ,x\right) ,$ $\mathrm{y}=y$.
\end{itemize}

We define%
\begin{eqnarray*}
\overrightarrow{\xi \left( M\right) } &=&\inf_{z\in D}\left\{ m_{l}\left( 
\frac{\partial f_{\mathrm{x}}}{\partial \mathrm{x}}(z)\right) -M\left\Vert 
\frac{\partial f_{\mathrm{x}}}{\partial \mathrm{y}}(z)\right\Vert \right\} ,
\\
\overrightarrow{\xi _{P}\left( M\right) } &=&\inf_{z\in D}m_{l}\left( \frac{%
\partial f_{\mathrm{x}}}{\partial \mathrm{x}}(P(z))\right) -M\sup_{z\in
D}\left\Vert \frac{\partial f_{\mathrm{x}}}{\partial \mathrm{y}}%
(z)\right\Vert , \\
\overrightarrow{\mu _{1}(M)} &=&\sup_{z\in D}\left\{ l\left( \frac{\partial
f_{\mathrm{y}}}{\partial \mathrm{y}}(z)\right) +M\left\Vert \frac{\partial
f_{\mathrm{y}}}{\partial \mathrm{x}}(z)\right\Vert \right\} , \\
\overrightarrow{\mu _{2}(M)} &=&\sup_{z\in D}\left\{ l\left( \frac{\partial
f_{\mathrm{y}}}{\partial \mathrm{y}}(h,z)\right) +M\left\Vert \frac{\partial
f_{\mathrm{x}}}{\partial \mathrm{y}}(h,z)\right\Vert \right\} .
\end{eqnarray*}%
We also consider the following quantities, which are defined for a given $%
h>0 $ 
\begin{eqnarray*}
\xi (h,M) &=&\inf_{z\in D}\left\{ m\left( \frac{\partial \Phi _{\mathrm{x}}}{%
\partial \mathrm{x}}(h,z)\right) -M\left\Vert \frac{\partial \Phi _{\mathrm{x%
}}}{\partial \mathrm{y}}(h,z)\right\Vert \right\} , \\
\xi _{P}\left( h,M\right) &=&\inf_{z\in D}m\left( \frac{\partial \Phi _{%
\mathrm{x}}}{\partial \mathrm{x}}(h,P(z))\right) -M\sup_{z\in D}\left\Vert 
\frac{\partial \Phi _{\mathrm{x}}}{\partial \mathrm{y}}(h,z)\right\Vert , \\
\mu _{1}(h,M) &=&\sup_{z\in D}\left\{ \left\Vert \frac{\partial \Phi _{%
\mathrm{y}}}{\partial \mathrm{y}}(h,z)\right\Vert +M\left\Vert \frac{%
\partial \Phi _{\mathrm{y}}}{\partial \mathrm{x}}(h,z)\right\Vert \right\} ,
\\
\mu _{2}(h,M) &=&\sup_{z\in D}\left\{ \left\Vert \frac{\partial \Phi _{%
\mathrm{y}}}{\partial \mathrm{y}}(h,z)\right\Vert +M\left\Vert \frac{%
\partial \Phi _{\mathrm{x}}}{\partial \mathrm{y}}(h,z)\right\Vert \right\} .
\end{eqnarray*}%
(in the application we will choose $M$ as $L$ or $1/L$, depending on which
of the rate conditions (\ref{eq:rate-cond-1}--\ref{eq:rate-cond-3}) we wish
to verify).

Below theorem can be used to establish the fact that rate conditions (see
Definition \ref{def:rate-conditions}) hold for the time shift map $\Phi
(h,\cdot )$.

\begin{theorem}
\label{th:ODE-rate-verif}Let $M,M_{1},M_{2}>0$. We have the following
conditions:

\begin{enumerate}
\item We have 
\begin{eqnarray}
\xi (h,M) &=&1+h\overrightarrow{\xi (M)}+O(h^{2})  \label{eq:ode-verif-xi} \\
\xi _{P}\left( h,M\right) &=&1+h\overrightarrow{\xi _{P}\left( M\right) }%
+O(h^{2}),  \label{eq:ode-verif-xiP} \\
\mu _{1}(h,M) &=&1+h\overrightarrow{\mu _{1}(M)}+O(h^{2}),
\label{eq:ode-verif-mu1} \\
\mu _{2}(h,M) &=&1+h\overrightarrow{\mu _{2}(M)}+O(h^{2}).
\label{eq:ode-verif-mu2}
\end{eqnarray}

\item \label{itm:ODE-rate-verif-1} If for $j\geq 0$ 
\begin{equation}
\overrightarrow{\mu _{2}(M_{1})}<\left( j+1\right) \overrightarrow{\xi
\left( M_{2}\right) },  \label{eq:ode-rate-cond-1-prf}
\end{equation}%
then for sufficiently small $h_{0}>0$, and for any $h\in \left(
0,h_{0}\right) $,%
\begin{equation*}
\frac{\mu _{2}(h,M_{1})}{\xi (h,M_{2})^{j+1}}<1.
\end{equation*}

\item \label{itm:ODE-rate-verif-2}If $\overrightarrow{\mu _{1}(M_{1})}<%
\overrightarrow{\xi _{P}\left( M_{2}\right) }$ then for sufficiently small $%
h_{0}>0$, and for any $h\in \left( 0,h_{0}\right) $,%
\begin{equation*}
\frac{\mu _{1}(h,M_{1})}{\xi _{P}(h,M_{2})}<1.
\end{equation*}

\item \label{itm:ODE-rate-verif-3}If $\overrightarrow{\xi _{P}\left(
M\right) }>0$ then for sufficiently small $h_{0}>0$, and for any $h\in
\left( 0,h_{0}\right) $,%
\begin{equation*}
\xi _{P}\left( h,M\right) >1.
\end{equation*}%
Also%
\begin{equation*}
\xi _{P}\left( h,M\right) =1+h\overrightarrow{\xi _{P}\left( M\right) }%
+O(h^{2}).
\end{equation*}

\item \label{itm:ODE-rate-verif-4}If $\overrightarrow{\mu _{1}(M)}<0$ then
for sufficiently small $h_{0}>0$, and for any $h\in \left( 0,h_{0}\right) $,%
\begin{equation*}
\mu _{1}(h,M)<1.
\end{equation*}

\item \label{itm:ODE-rate-verif-5}If $h_{0}>0$ is sufficiently small, then
for any $h\in \left( 0,h_{0}\right) $,%
\begin{equation*}
\xi _{P}\left( h,M\right) >0\qquad \text{and\qquad }\xi \left( h,M\right) >0.
\end{equation*}
\end{enumerate}
\end{theorem}

\begin{proof}
We have 
\begin{eqnarray}
\Phi (h,z) &=&z+hf(z)+O(h^{2})  \notag \\
\frac{\partial \Phi }{\partial z}(h,z) &=&I+hDf(z)+O(h^{2}),  \notag \\
D_{z}^{2}\Phi (h,z) &=&hD^{2}f(z)+O(h^{2}).  \label{eq:d2varphi}
\end{eqnarray}%
where the $O(h^{2})$ are uniform in $z$ for $z\in D$.

Using (\ref{eq:def-ml}) and Lemma \ref{lem:ml-conv} we obtain 
\begin{eqnarray*}
m\left( \frac{\partial \Phi _{\mathrm{x}}}{\partial \mathrm{x}}(h,z)\right)
&=&m\left( I+h\frac{\partial f_{\mathrm{x}}}{\partial \mathrm{x}}%
+O(h^{2})\right) \\
&=&1+hm_{l}\left( \frac{\partial f_{\mathrm{x}}}{\partial \mathrm{x}}%
(z)\right) +O(h^{2}).
\end{eqnarray*}%
From (\ref{eq:def-log-norm}) and Lemma \ref{lem:lognorm-conv} we obtain 
\begin{eqnarray*}
\left\Vert \frac{\partial \Phi _{\mathrm{y}}}{\partial \mathrm{y}}%
(h,z)\right\Vert &=&\left\Vert I+h\frac{\partial f_{\mathrm{y}}}{\partial 
\mathrm{y}}(z)+O(h^{2})\right\Vert \\
&=&1+hl\left( \frac{\partial f_{\mathrm{y}}}{\partial \mathrm{y}}(z)\right)
+O(h^{2}).
\end{eqnarray*}%
And finally 
\begin{eqnarray*}
\left\Vert \frac{\partial \Phi _{\mathrm{x}}}{\partial \mathrm{y}}%
(h,z)\right\Vert &=&\left\Vert h\frac{\partial f_{\mathrm{x}}}{\partial 
\mathrm{y}}(z)+O(h^{2})\right\Vert =h\left\Vert \frac{\partial f_{\mathrm{x}}%
}{\partial \mathrm{y}}(z)\right\Vert +O(h^{2}), \\
\left\Vert \frac{\partial \Phi _{\mathrm{y}}}{\partial \mathrm{x}}%
(h,z)\right\Vert &=&\left\Vert h\frac{\partial f_{\mathrm{y}}}{\partial 
\mathrm{x}}(z)+O(h^{2})\right\Vert =h\left\Vert \frac{\partial f_{\mathrm{y}}%
}{\partial \mathrm{x}}(z)\right\Vert +O(h^{2}).
\end{eqnarray*}%
By combining the above formulas we obtain (\ref{eq:ode-verif-xi}--\ref%
{eq:ode-verif-mu2}).

We now prove the claim \textit{\ref{itm:ODE-rate-verif-1}.} of our theorem.
Since 
\begin{equation*}
\xi (h,M_{2})^{j+1}=\left( 1+h\overrightarrow{\xi (M_{2})}+O(h^{2})\right)
^{j+1}=1+h\left( j+1\right) \overrightarrow{\xi \left( M_{2}\right) }%
+O(h^{2}),
\end{equation*}%
from (\ref{eq:ode-rate-cond-1-prf}),%
\begin{multline*}
\mu _{2}(h,M_{1})=1+h\overrightarrow{\mu _{2}(M_{1})}+O(h^{2})< \\
<1+h\left( j+1\right) \overrightarrow{\xi \left( M_{2}\right) }+O(h^{2})=\xi
(h,M_{2})^{j+1}+O(h^{2}).
\end{multline*}%
This establishes the claim.

Claim \textit{\ref{itm:ODE-rate-verif-2}.} follows from mirror arguments
(taking $j=0$).

The claims \textit{\ref{itm:ODE-rate-verif-3}.} and \textit{\ref%
{itm:ODE-rate-verif-4}.} follow from (\ref{eq:ode-verif-xiP}) and (\ref%
{eq:ode-verif-mu1}), respectively.

Claim \textit{\ref{itm:ODE-rate-verif-5}.} follows from (\ref%
{eq:ode-verif-xi}) and (\ref{eq:ode-verif-xiP}).
\end{proof}

\subsection{Verification of covering conditions\label{sec:covering-ode}}

Here we show that from conditions in the definition of an isolating block
follow covering conditions for a time shift map along the trajectory of an
ODE.

\begin{theorem}
\label{th:cover-ode}Assume that $D=\Lambda \times \overline{B}_{u}\left(
R\right) \times \overline{B}_{s}\left( R\right) $ is an isolating block for (%
\ref{eq:ode-Wcu}) and let 
\begin{equation*}
F_{t}(q)=\Phi \left( t,q\right) .
\end{equation*}%
If $t$ is sufficiently small, then $F_{t}$ satisfies covering conditions
(see Definition \ref{def:covering}).
\end{theorem}

\begin{proof}
We need to construct the homotopy from $h$ from Definition \ref{def:covering}%
.

Let $C:\left( \lambda ,x,y\right) \rightarrow \left( 0,x,-y\right) $ and for 
$\alpha \in \left[ 0,\frac{1}{2}\right] $ let 
\begin{equation*}
H_{\alpha }=\left( 1-2\alpha \right) f+2\alpha C.
\end{equation*}%
For any $q\in \Lambda \times \partial \overline{B}_{u}\left( R\right) \times 
\overline{B}_{s}\left( R\right) ,$ 
\begin{equation}
\left( \pi _{x}H_{\alpha }\left( q\right) |\pi _{x}q\right) =\left(
1-2\alpha \right) \left( \pi _{x}f\left( q\right) |\pi _{x}q\right) +2\alpha
\left( \pi _{x}q|\pi _{x}q\right) >0,  \label{eq:homotopy-exit-ode}
\end{equation}%
and for any $q\in \Lambda \times \overline{B}_{u}\left( R\right) \times
\partial \overline{B}_{s}\left( R\right) $%
\begin{equation}
\left( \pi _{y}H_{\alpha }\left( q\right) |\pi _{y}q\right) =\left(
1-2\alpha \right) \left( \pi _{y}f\left( q\right) |\pi _{y}q\right) -2\alpha
\left( \pi _{y}q|\pi _{y}q\right) <0.  \label{eq:homotopy-enter-ode}
\end{equation}

Let $\phi _{\alpha }(t,q)$ be the flow induced by $q^{\prime }=H_{\alpha
}(q).$ Note that 
\begin{equation*}
\phi _{\alpha =1/2}(t,\left( \lambda ,x,y\right) )=\left( \lambda
,e^{t}x,e^{-t}y\right) .
\end{equation*}%
We shall fix a time $t$ (where $t$ will be sufficiently small) and define%
\begin{multline*}
h_{\alpha }\left( \lambda ,x,y\right) = \\
=\left\{ 
\begin{array}{lll}
\phi _{\alpha }(t,q) & \qquad & \text{for }\alpha \in \lbrack 0,\frac{1}{2})
\\ 
\left( \left( 2-2\alpha \right) \lambda +\left( 2\alpha -1\right) \lambda
^{\ast },e^{t}x,\left( 2-2\alpha \right) e^{-t}y\right) &  & \text{for }%
\alpha \in \lbrack \frac{1}{2},1]%
\end{array}%
\right. .
\end{multline*}%
Let $z\in D$ be a fixed point, let $U=J_{u}(z,1/L)\cap D$ be the set from
Definition \ref{def:covering} and let $z^{\ast }=\pi _{\lambda }z$. Note
that for small $t$ and $\alpha \in \lbrack 0,\frac{1}{2}),$ the $h_{\alpha
}\left( q\right) $ is close to identity. This means that for sufficiently
small $t$, for any $\alpha \in \lbrack 0,1]$%
\begin{equation*}
\pi _{\lambda }U\subset B_{c}\left( \lambda ^{\ast },R_{\Lambda }\right) .
\end{equation*}%
This means that our homotopy is well defined on $U$, i.e. 
\begin{equation*}
h:\left[ 0,1\right] \times U\rightarrow B_{c}\left( \lambda ^{\ast
},R_{\Lambda }\right) \times \mathbb{R}^{u}\times \mathbb{R}^{s}.
\end{equation*}

We now verify conditions \ref{pt:covering-1}.--\ref{pt:covering-4}. of
Definition \ref{def:covering}. The point \ref{pt:covering-1}. follows from
our construction. The conditions (\ref{eq:homotopy-exit}) and (\ref%
{eq:homotopy-enter}) follow from (\ref{eq:homotopy-exit-ode}) and (\ref%
{eq:homotopy-enter-ode}), respectively, provided that $t$ is sufficiently
small. Conditions \ref{pt:covering-3}. and \ref{pt:covering-4}. follow from
our definition of $h_{\alpha }$.
\end{proof}

\subsection{Verification of backward cone conditions}

In this section we show that if our vector field satisfies rate conditions,
then the time shift map along the solution of the ODE will satisfy backward
cone conditions.

For our proof we will need the following lemma:

\begin{lemma}
\cite[Corollary 35]{Geom}\label{lem:Jsc-propagation} If a map $F$ satisfies
rate conditions (for maps; as in Definition \ref{def:rate-conditions}) of
order $k\geq 0$ then for any $z\in D$%
\begin{equation*}
F(\overline{J_{s}^{c}\left( z,1/L\right) }\cap (\overline{B}_{c}\left( \pi
_{\lambda }z,R_{\Lambda }\right) \times \overline{B}_{u}\left( R\right)
\times \overline{B}_{s}\left( R\right) )\subset J_{s}^{c}\left(
F(z),1/L\right) \cup \left\{ F(z)\right\} .
\end{equation*}
\end{lemma}

We can now formulate our theorem.

\begin{theorem}
\label{th:ode-back-cc}If $f$ satisfies rate conditions of order $k\geq 0$,
then for sufficiently small $h>0$, for any $t\in \left( 0,h\right) $, the
map $F_{t}(z):=\Phi \left( t,z\right) $ satisfies backward cone conditions.
\end{theorem}

\begin{proof}
The proof is based on Lemma \ref{lem:Jsc-propagation}, which establishes
forward invariance of complements of $J_s$ for maps satisfying rate
conditions. In the proof, these maps will be time shifts along the
trajectory of an ODE. We will also make use of the fact that such maps are
close to identity for small times.

Recall that we have chosen $L\in \left( \frac{2R}{R_{\Lambda }},1\right) .$
This implies that for $z_{2}\in D$ and $z_{1}\in J_{s}\left(
z_{2},1/L\right) \cap D$%
\begin{equation*}
\left\Vert \pi _{\lambda }\left( z_{1}-z_{2}\right) \right\Vert \leq
\left\Vert \pi _{\lambda ,x}\left( z_{1}-z_{2}\right) \right\Vert \leq \frac{%
1}{L}\left\Vert \pi _{y}\left( z_{1}-z_{2}\right) \right\Vert \leq \frac{1}{L%
}2R<R_{\Lambda }.
\end{equation*}%
In other words, for any $z\in D$%
\begin{equation*}
\pi _{\lambda }\left( J_{s}\left( z,1/L\right) \cap D\right) \subset
B_{c}\left( \pi _{\lambda }z,R_{\Lambda }\right) .
\end{equation*}%
Since for small $t$ the $F_{t}$ is close to identity, we can choose $h$
small enough so that for any $t\in \left( 0,h\right) $%
\begin{equation}
\pi _{\lambda }F_{-t}(J_{s}\left( z,1/L\right) \cap D)\subset B_{c}\left(
\pi _{\lambda }F_{-t}\left( z\right) ,R_{\Lambda }\right) .
\label{eq:back-jet-projection}
\end{equation}

Suppose that backward cone conditions do not hold. Then for any $h>0$, there
exists a $t\in \left( 0,h\right) $ and a pair of points $%
z_{1},z_{2},F_{t}(z_{1}),F_{t}(z_{2})\in D$ satisfying 
\begin{equation}
F_{t}(z_{1})\in J_{s}\left( F_{t}(z_{2}),1/L\right) 
\label{eq:tmp-back-cc-proof-1}
\end{equation}%
such that 
\begin{equation*}
z_{1}\in J_{s}^{c}\left( z_{2},1/L\right) .
\end{equation*}%
By (\ref{eq:tmp-back-cc-proof-1}) and (\ref{eq:back-jet-projection}) 
\begin{equation*}
\pi _{\lambda }z_{1}=\pi _{\lambda }F_{-t}\left( F_{t}(z_{1})\right) \in 
\overline{B}_{c}\left( \pi _{\lambda }F_{-t}\left( F_{t}(z_{2})\right)
,R_{\Lambda }\right) =\overline{B}_{c}\left( \pi _{\lambda }z_{2},R_{\Lambda
}\right) ,
\end{equation*}%
which since $z_{1}\in J_{s}^{c}\left( z_{2},1/L\right) $ means that 
\begin{equation*}
z_{1}\in J_{s}^{c}\left( z_{2},1/L\right) \cap \overline{B}_{c}\left( \pi
_{\lambda }z_{2},R_{\Lambda }\right) \times \overline{B}_{u}\left( R\right)
\times \overline{B}_{s}\left( R\right) .
\end{equation*}%
Since $f$ satisfies rate conditions, by Theorem \ref{th:ODE-rate-verif}, for
sufficiently small $h$ and any $t\in \left( 0,h\right) $ the map $F_{t}$
will satisfy rate conditions (for maps; as in Definition \ref%
{def:rate-conditions}). This, by Lemma \ref{lem:Jsc-propagation} contradicts
(\ref{eq:tmp-back-cc-proof-1}). This concludes our proof.
\end{proof}

\subsection{Proof of the existence of the center unstable manifold\label%
{sec:wcu-ode-proof}}

In this section we will prove the existence of the center unstable manifold
and unstable fibers, which was formulated in Theorem \ref{th:main-ode}.
First we need a technical lemma:

\begin{lemma}
\cite[Corollary 34]{Geom}\label{lem:Ju-cone-inclusion} If a map $F$
satisfies rate conditions of order $k\geq 0$ (for maps, as in Definition \ref%
{def:rate-conditions}), then for any $z\in D$%
\begin{equation*}
F\left( J_{u}\left( z,1/L\right) \cap D\right) \subset \mathrm{int}%
J_{u}\left( F\left( z\right) ,1/L\right) \cup \left\{ F\left( z\right)
\right\}
\end{equation*}
\end{lemma}

We are now ready to prove Theorem \ref{th:main-ode}.

\begin{proof}[Proof of Theorem \protect\ref{th:main-ode}]
Let $F_{t}(q)=\Phi \left( t,q\right) $. We shall write $W^{cu}\left(
F_{t}\right) $ and $W_{z}^{u}\left( F_{t}\right) $ for the center unstable
manifold and for the unstable fiber induced by the map $F_{t}$,
respectively. (These are in the sense of Definitions \ref{def:Wcu-maps}, \ref%
{def:Wu-fiber}.) We shall also write $W^{cu}\left( \Phi \right) $ and $%
W_{z}^{u}\left( \Phi \right) $ for the manifolds induced by the flow (in the
sense of Definitions \ref{def:Wcu-ode}, \ref{def:ode-fiber}).

By Theorems \ref{th:ODE-rate-verif}, \ref{th:cover-ode} and \ref%
{th:ode-back-cc}, for sufficiently small $t$ the function $F_{t}(q)=\Phi
\left( t,q\right) $ satisfies assumptions of Theorem \ref{th:cu-maps}. We
can therefore fix a small $h$ and apply Theorem \ref{th:cu-maps} for the map 
$F_{h}$ and obtain the center unstable manifold $W^{cu}\left( F_{h}\right) $
and the unstable fiber $W_{z}^{u}\left( F_{h}\right)$. It will turn out that
if we choose $h$ sufficiently small, then we can show that $W^{cu}\left(
F_{h}\right) =W^{cu}\left( \Phi \right)$ and $W_{z}^{u}\left( F_{h}\right)
=W_{z}^{u}\left( \Phi \right) $.

We first show that if $h$ is chosen to be small, then $W^{cu}\left(
F_{h}\right) \subset W^{cu}\left( \Phi \right) $. Consider $D^{+}:=\Lambda
\times \overline{B}_{u}\left( R\right) \times \partial \overline{B}%
_{s}\left( R\right) $. Since $D$ is an isolating block, and $D^{+}$ is
compact, there exists a $\delta >0$ such that 
\begin{equation}
\Phi \left( -s,z\right) \notin D\qquad \text{for all }z\in D^{+}\text{ and }%
s\in (0,\delta ].  \label{eq:flow-back-exit}
\end{equation}%
Let us choose $h<\delta $. We shall show that with such choice of $h$, for
any $z\in W^{cu}\left( F_{h}\right) $ we will have $\Phi \left( -t,z\right)
\in D$, for all $t>0$. Since $z\in W^{cu}\left( F_{h}\right) ,$ we know that 
\begin{equation}
F_{h}^{-n}(z)=\Phi \left( -nh,z\right) \in D.  \label{eq:tmp-Phi-in-D}
\end{equation}%
Suppose now that for some $t>0$, $\Phi \left( -t,z\right) \notin D$. By (\ref%
{eq:tmp-Phi-in-D}), $t\in \left( -\left( n+1\right) h,-nh\right) $ for some $%
n\in \mathbb{N}$. Since $D$ is an isolating block, the only possibility to
leave $D$ going backwards in time is by passing through $D^{+}$. Hence, for
some $\tau ^{\ast }\in \left( nh,t\right) $ there exists a $z^{\ast }=\Phi
\left( -\tau ^{\ast },z\right) \in D^{+}$. We see that 
\begin{equation*}
\Phi \left( -(n+1)h+\tau ^{\ast },z^{\ast }\right) =\Phi \left( -(n+1)h+\tau
^{\ast },\Phi \left( -\tau ^{\ast },z\right) \right) =\Phi \left(
-(n+1)h,z\right) \in D,
\end{equation*}%
but this contradicts (\ref{eq:flow-back-exit}) by taking $s=(n+1)h-\tau
^{\ast }$. We have thus shown that for $z\in W^{cu}\left( F_{h}\right) $, $%
\Phi \left( t,z\right) \in D$ for all $t<0$, hence $W^{cu}\left(
F_{h}\right) \subset W^{cu}\left( \Phi \right) $. The inclusion in the
opposite direction is evident.

We now show that $W_{z}^{u}(F_{h})\subset W_{z}^{u}(\Phi )$. Let us consider
a point $p\in W_{z}^{u}\left( F_{h}\right) $. We will show that $p\in
W_{z}^{u}(\Phi )$. Since $W_{z}^{u}\left( F_{h}\right) \subset W^{cu}\left(
F_{h}\right) =W^{cu}\left( \Phi \right) ,$ 
\begin{equation*}
\Phi \left( t,p\right) \in D\qquad \text{for all }t<0.
\end{equation*}%
We also know that since $p\in W_{z}^{u}\left( F_{h}\right) $, 
\begin{equation}
F_{h}^{-n}\left( p\right) =J_{u}\left( F_{h}^{-n}\left( z\right) ,1/L\right)
\cap D.  \label{eq:tmp-cones-Ju-inclusion}
\end{equation}%
By Theorem \ref{th:ODE-rate-verif}, for any $\tau \in \left( 0,h\right) $,
the map $F_{\tau }$ satisfies rate conditions, so, by Lemma \ref%
{lem:Ju-cone-inclusion} and (\ref{eq:tmp-cones-Ju-inclusion}),%
\begin{equation*}
F_{\tau }\left( F_{h}^{-n}\left( p\right) \right) \in J_{u}\left( F_{\tau
}\left( F_{h}^{-n}\left( z\right) \right) ,1/L\right) \cap D.
\end{equation*}%
Since $F_{\tau }\left( F_{h}^{-n}\left( \cdot \right) \right) =\Phi \left(
-nh+\tau ,\cdot \right) $ and $n\in \mathbb{N}$, $\tau \in \left( 0,h\right) 
$ are arbitrary, we obtain%
\begin{equation*}
\Phi \left( t,p\right) \in J_{u}\left( \Phi \left( t,p\right) ,1/L\right)
\cap D\qquad \text{for all }t<0.
\end{equation*}%
We have thus shown that $p\in W_{z}^{u}\left( \Phi \right) ,$ hence $%
W_{z}^{u}\left( F_{h}\right) \subset W_{z}^{u}\left( \Phi \right) .$ The
inclusion in the opposite direction is evident.

What remains is to show (\ref{eq:Wuz-covergence-ode}). Let us denote by $\xi
_{u,1,P}(h)$ the constant $\xi _{u,1,P}$ defined for the map $F_{h}.$ (See
beginning of section \ref{sec:wcu-maps-setup} for the definition of $\xi
_{u,1,P}.$) By (\ref{eq:ode-verif-xiP}) we know that%
\begin{equation*}
\xi _{u,1,P}\left( h\right) =1+h\overrightarrow{\xi _{u,1,P}}+O(h^{2}).
\end{equation*}%
We have shown above that for sufficiently small $h$, $W_{z}^{u}\left(
F_{h}\right) =W_{z}^{u}\left( \Phi \right) $. Therefore, by (\ref%
{eq:Wuz-convergence-maps}) from Theorem \ref{th:cu-maps},%
\begin{eqnarray*}
\left\Vert \Phi \left( -t,p\right) -\Phi \left( -t,z\right) \right\Vert
&=&\left\Vert F_{t/n}^{-n}\left( p\right) -F_{t/n}^{-n}\left( z\right)
\right\Vert \\
&\leq &C\left( 1+\frac{t}{n}\overrightarrow{\xi _{u,1,P}}+O\left( \frac{t}{n}%
\right) ^{2}\right) ^{-n}.
\end{eqnarray*}%
Passing to the limit with $n\rightarrow \infty $,%
\begin{equation*}
\left\Vert \Phi \left( -t,p\right) -\Phi \left( -t,z\right) \right\Vert \leq
Ce^{-t\overrightarrow{\xi _{u,1,P}}},
\end{equation*}%
which concludes the proof of (\ref{eq:Wuz-covergence-ode}).
\end{proof}

\subsection{Bounds on second derivatives\label{subsec:ode-bnds-sec-der}}

In this section, for the sake of simplicity, we shall again use two
coordinates $\mathrm{x}$ and $\mathrm{y}$. We shall study the bounds on the
second derivative of a function $\mathrm{y}=w\left( \mathrm{x}\right) $
under appropriate rate conditions. In applications, we can have:

\begin{itemize}
\item $\mathrm{x}=x,$ $\mathrm{y}=\left( \lambda ,y\right) $ and $w(\mathrm{x%
})=w_{z}^{u}\left( \mathrm{x}\right) $;

\item $\mathrm{x}=\left( \lambda ,x\right) ,$ $\mathrm{y}=y$ and $w(\mathrm{x%
})=w^{cu}\left( \mathrm{x}\right) .$
\end{itemize}

Similarly, in the case of a family of odes, which depend on parameters, we
can have:

\begin{itemize}
\item $\mathrm{x}=x,$ $\mathrm{y}=\left( \varepsilon ,\lambda ,y\right) $
and $w(\mathrm{x})=w_{z}^{u}\left( \mathrm{x}\right) $;

\item $\mathrm{x}=\left( \varepsilon ,\lambda ,x\right) ,$ $\mathrm{y}=y$
and $w(\mathrm{x})=w^{cu}\left( \mathrm{x}\right) .$
\end{itemize}

We shall assume that $(\mathrm{x},\mathrm{y})\in \mathbb{R}^{u+s}$ and
consider vector field $f:D\rightarrow \mathbb{R}^{u+s}$ which is $C^{3}$,
where $D\subset \mathbb{R}^{u+s}$ is the domain of $f$. We consider a map $%
F=\left( F_{\mathrm{x}},F_{\mathrm{y}}\right) =\Phi \left( h,\cdot \right) $%
, a time shift by $h$ along the trajectory of the flow.

We assume that $F$ is such that if $v:\mathbb{R}^{u}\rightarrow \mathbb{R}%
^{s}$ is Lipschitz with constant $\mathcal{L}$, then the graph transform $%
\mathcal{G}\left( v\right) $ is well defined i.e.%
\begin{equation*}
\mathcal{G}\left( v\right) =F_{\mathrm{y}}\circ \left( \mathrm{id},v\right)
\circ \left( F_{\mathrm{x}}\circ \left( \mathrm{id},v\right) \right) ^{-1}.
\end{equation*}

Assume also that for $v_{0}\left( \mathrm{x}\right) =0$%
\begin{equation}
w=\lim_{n\rightarrow \infty }\mathcal{G}^{n}\left( v_{0}\right) ,
\label{eq:G-ode-limit}
\end{equation}%
for all $h\in (0,h_{0}]$. (Such is the setting in the construction of $%
w^{cu} $ and $w_{z}^{u}$ in \cite{Geom}. For $w^{cu}$, $\mathcal{L}=L$ and
for $w_{z}^{u}$, $\mathcal{L}=1/L.$ These properties follow from assumptions
of Theorem \ref{th:main-ode}.) The following result will allow us to obtain
estimates on the second derivative of $w$.

\begin{theorem}
\label{th:ode-unstable-cones} Let 
\begin{eqnarray*}
\overrightarrow{\xi } &=&\inf_{z\in D}m_{l}\left( \frac{\partial f_{\mathrm{x%
}}}{\partial \mathrm{x}}(z)\right) -\mathcal{L}\sup_{z\in D}\left\Vert \frac{%
\partial f_{\mathrm{x}}}{\partial \mathrm{y}}(z)\right\Vert , \\
\overrightarrow{\mu }_{1} &=&\sup_{z\in D}l\left( \frac{\partial f_{\mathrm{y%
}}}{\partial \mathrm{y}}(z)\right) +\frac{1}{\mathcal{L}}\sup_{z\in
D}\left\Vert \frac{\partial f_{\mathrm{y}}}{\partial \mathrm{x}}%
(z)\right\Vert , \\
\overrightarrow{\mu }_{2} &=&\sup_{z\in D}l\left( \frac{\partial f_{\mathrm{y%
}}}{\partial \mathrm{y}}(z)\right) +\mathcal{L}\sup_{z\in D}\left\Vert \frac{%
\partial f_{\mathrm{x}}}{\partial \mathrm{y}}(z)\right\Vert .
\end{eqnarray*}%
Assume that 
\begin{equation}
\overrightarrow{\mu }_{1}<\overrightarrow{\xi },\qquad \overrightarrow{\mu }%
_{2}<2\overrightarrow{\xi }.  \label{eq:ode-mu-rates}
\end{equation}%
Let 
\begin{eqnarray*}
\overrightarrow{C}_{x} &=&\frac{1}{2}\max_{p\in D,\Vert v\Vert =1}\Vert
D^{2}f_{\mathrm{x}}(p)(v,v)\Vert , \\
\overrightarrow{C}_{y} &=&\frac{1}{2}\max_{p\in D,\Vert v\Vert =1}\Vert
D^{2}f_{\mathrm{y}}(p)(v,v)\Vert .
\end{eqnarray*}%
and 
\begin{eqnarray*}
\overrightarrow{C}_{y,1} &=&\sup_{p\in D}\frac{1}{2}\left\Vert \frac{%
\partial ^{2}f_{\mathrm{y}}}{\partial \mathrm{x}^{2}}(p)\right\Vert , \\
\overrightarrow{C}_{y,2} &=&\sup_{p\in D}\left\Vert \frac{\partial ^{2}f_{%
\mathrm{y}}}{\partial \mathrm{x}\partial \mathrm{y}}(p)\right\Vert , \\
\overrightarrow{C}_{y,3} &=&\sup_{p\in D}\frac{1}{2}\left\Vert \frac{%
\partial ^{2}f_{\mathrm{y}}}{\partial \mathrm{y}^{2}}(p)\right\Vert .
\end{eqnarray*}

Then for any $\mathrm{x}$ and $v$ holds (where it makes sense) where $w$ is
defined by (\ref{eq:G-ode-limit}) 
\begin{equation*}
w(\mathrm{x}+v)=w(\mathrm{x})+Dw(\mathrm{x})v+\Delta y(\mathrm{x},v),\quad
\Vert \Delta y(\mathrm{x},v)\Vert \leq M\Vert v\Vert ^{2}
\end{equation*}%
where 
\begin{equation}
M>\frac{(\mathcal{L}\overrightarrow{C}_{x}+\overrightarrow{C}_{y})(1+%
\mathcal{L}^{2})}{2\overrightarrow{\xi }-\overrightarrow{\mu }_{2}}.
\label{eq:ode-M-estm}
\end{equation}%
One can obtain an alternative (giving tighter estimates; see Remark \ref%
{rem:flat-manif}) expression for $M$ 
\begin{equation}
M>\frac{\mathcal{L}\overrightarrow{C}_{x}(1+\mathcal{L}^{2})+\overrightarrow{%
C}_{y,1}+\overrightarrow{C}_{y,2}\mathcal{L}+\overrightarrow{C}_{y,3}%
\mathcal{L}^{2}}{2\overrightarrow{\xi }-\overrightarrow{\mu }_{2}}
\label{eq:ode-M-estm-improved}
\end{equation}%
Hence for any $v\in \mathbb{R}^{u}$ holds 
\begin{equation*}
\left\Vert \frac{1}{2}D^{2}w\left( \mathrm{x}\right) (v,v)\right\Vert \leq
M\Vert v\Vert ^{2}.
\end{equation*}
\end{theorem}

\begin{proof}
We derive the result from Theorem~\ref{th:unstable-cones} for time shift by $%
h$ for sufficiently small $h$. From the proof of Theorem \ref{th:main-ode}
(in section \ref{sec:wcu-ode-proof}) we know that the limit (\ref%
{eq:G-ode-limit}) is independent of $h\in (0,h_{0}]$, provided that $h_{0}$
is small enough.

Let us fix $h\in (0,h_{0}]$. We define 
\begin{eqnarray*}
\xi (h) &=&\inf_{z\in D}m_{l}\left( \frac{\partial \Phi _{\mathrm{x}}}{%
\partial \mathrm{x}}(h,z)\right) -\mathcal{L}\sup_{z\in D}\left\Vert \frac{%
\partial \Phi _{\mathrm{x}}}{\partial \mathrm{y}}(h,z)\right\Vert , \\
\mu _{2}(h) &=&\sup_{z\in D}l\left( \frac{\partial \Phi _{\mathrm{y}}}{%
\partial \mathrm{y}}(h,z)\right) +\mathcal{L}\sup_{z\in D}\left\Vert \frac{%
\partial \Phi _{\mathrm{x}}}{\partial \mathrm{y}}(h,z)\right\Vert , \\
\mu _{1}(h) &=&\sup_{z\in D}l\left( \frac{\partial \Phi _{\mathrm{y}}}{%
\partial \mathrm{y}}(h,z)\right) +\frac{1}{\mathcal{L}}\sup_{z\in
D}\left\Vert \frac{\partial \Phi _{\mathrm{y}}}{\partial \mathrm{x}}%
(h,z)\right\Vert ,
\end{eqnarray*}

Let 
\begin{eqnarray*}
C_{x}(h) &=&\frac{1}{2}\max_{p\in D,\Vert v\Vert =1}\Vert D^{2}\Phi _{%
\mathrm{x}}(h,p)(v,v)\Vert , \\
C_{y}(h) &=&\frac{1}{2}\max_{p\in D,\Vert v\Vert =1}\Vert D^{2}\Phi _{%
\mathrm{y}}(h,p)(v,v)\Vert .
\end{eqnarray*}%
and 
\begin{eqnarray*}
C_{y,1}(h) &=&\sup_{p\in D}\frac{1}{2}\left\Vert \frac{\partial ^{2}\Phi _{%
\mathrm{y}}}{\partial \mathrm{x}^{2}}(p)\right\Vert , \\
C_{y,2}(h) &=&\sup_{p\in D}\left\Vert \frac{\partial ^{2}\Phi _{\mathrm{y}}}{%
\partial \mathrm{x}\partial \mathrm{y}}(h,p)\right\Vert , \\
C_{y,3}(h) &=&\sup_{p\in D}\frac{1}{2}\left\Vert \frac{\partial ^{2}\Phi _{%
\mathrm{y}}}{\partial \mathrm{y}^{2}}(p)\right\Vert .
\end{eqnarray*}

We have 
\begin{eqnarray*}
\Phi (h,z) &=&z+hf(z)+O(h^{2}) \\
\frac{\partial \Phi }{\partial z}(h,z) &=&I+hDf(z)+O(h^{2}), \\
D_{z}^{2}\Phi (h,z) &=&hD^{2}f(z)+O(h^{2}).
\end{eqnarray*}%
where the $O(h^{2})$ are uniform in $z$ for $z\in D$.

Using Lemma \ref{lem:ml-conv} we obtain 
\begin{eqnarray*}
m\left( \frac{\partial \Phi _{\mathrm{x}}}{\partial \mathrm{x}}(h,z)\right)
&=&m\left( I+h\frac{\partial f_{\mathrm{x}}}{\partial \mathrm{x}}%
+O(h^{2})\right) \\
&=&1+hm_{l}\left( \frac{\partial f_{\mathrm{x}}}{\partial \mathrm{x}}%
(z)\right) +O(h^{2}).
\end{eqnarray*}%
From Lemma \ref{lem:lognorm-conv} we obtain 
\begin{eqnarray*}
\left\Vert \frac{\partial \Phi _{\mathrm{y}}}{\partial \mathrm{y}}%
(h,z)\right\Vert &=&\left\Vert I+h\frac{\partial f_{\mathrm{y}}}{\partial 
\mathrm{y}}(z)+O(h^{2})\right\Vert \\
&=&1+hl\left( \frac{\partial f_{\mathrm{y}}}{\partial \mathrm{y}}(z)\right)
+O(h^{2}).
\end{eqnarray*}%
And finally 
\begin{eqnarray*}
\left\Vert \frac{\partial \Phi _{\mathrm{x}}}{\partial \mathrm{y}}%
(h,z)\right\Vert &=&\left\Vert h\frac{\partial f_{\mathrm{x}}}{\partial 
\mathrm{y}}(z)+O(h^{2})\right\Vert =h\left\Vert \frac{\partial f_{\mathrm{x}}%
}{\partial \mathrm{y}}(z)\right\Vert +O(h^{2}), \\
\left\Vert \frac{\partial \Phi _{\mathrm{y}}}{\partial \mathrm{x}}%
(h,z)\right\Vert &=&\left\Vert h\frac{\partial f_{\mathrm{y}}}{\partial 
\mathrm{x}}(z)+O(h^{2})\right\Vert =h\left\Vert \frac{\partial f_{\mathrm{y}}%
}{\partial \mathrm{x}}(z)\right\Vert +O(h^{2}).
\end{eqnarray*}%
By combining the above formulas we obtain 
\begin{eqnarray*}
\xi (h) &=&1+h\overrightarrow{\xi }+O(h^{2}), \\
\mu _{1}(h) &=&1+h\overrightarrow{\mu }_{1}+O(h^{2}), \\
\mu _{2}(h) &=&1+h\overrightarrow{\mu }_{2}+O(h^{2}), \\
C_{x}(h) &=&h\overrightarrow{C}_{x}+O(h^{2}), \\
C_{y}(h) &=&h\overrightarrow{C}_{y}+O(h^{2}), \\
C_{y,1}(h) &=&h\overrightarrow{C}_{y,1}+O(h^{2}), \\
C_{y,2}(h) &=&h\overrightarrow{C}_{y,2}+O(h^{2}), \\
C_{y,3}(h) &=&h\overrightarrow{C}_{y,3}+O(h^{2}).
\end{eqnarray*}

From (\ref{eq:ode-mu-rates}) and the above equalities we have for $h$
sufficiently small 
\begin{equation*}
\xi (h)>0,\qquad \frac{\mu _{1}(h)}{\xi (h)}<1,\qquad \frac{\mu _{2}(h)}{\xi
(h)^{2}}<1,
\end{equation*}%
hence we can apply Theorem~\ref{th:unstable-cones} to the map $F=\Phi
(h,\cdot )$ to obtain that%
\begin{equation*}
w(\mathrm{x}+v)=w(\mathrm{x})+Dw(\mathrm{x})v+\Delta y(\mathrm{x},v),\quad
\Vert \Delta y(\mathrm{x},v)\Vert \leq M(h)\Vert v\Vert ^{2},
\end{equation*}%
and%
\begin{equation*}
\left\Vert \frac{1}{2}D^{2}w\left( \mathrm{x}\right) (v,v)\right\Vert \leq
M(h)\Vert v\Vert ^{2},
\end{equation*}%
for any $M(h)$ satisfying (based on (\ref{eq:M})), 
\begin{equation}
M(h)>\frac{(\mathcal{L}C_{x}(h)+C_{y}(h))(1+\mathcal{L}^{2})}{\xi
(h)^{2}-\mu _{2}(h)},  \label{eq:ode-Mh}
\end{equation}%
or (based on (\ref{eq:M-improved})), 
\begin{equation*}
M(h)>\frac{\mathcal{L}C_{x}(h)(1+\mathcal{L}^{2})+C_{y,1}(h)+C_{y,2}(h)%
\mathcal{L}+C_{y,3}(h)\mathcal{L}^{2}}{\xi (h)^{2}-\mu _{2}(h)}.
\end{equation*}

Let us pass to the limit $h\rightarrow 0$ in (\ref{eq:ode-Mh}). Then we have 
\begin{multline*}
\frac{(\mathcal{L}C_{x}(h)+C_{y}(h))(1+\mathcal{L}^{2})}{\xi (h)^{2}-\mu
_{2}(h)}=\frac{(\mathcal{L}h\overrightarrow{C}_{x}+O(h^{2})+h\overrightarrow{%
C}_{y}+O(h^{2}))(1+\mathcal{L}^{2})}{(1+h\overrightarrow{\xi }%
+O(h^{2}))^{2}-(1+h\overrightarrow{\mu }_{2}+O(h^{2}))}= \\
\frac{h(\mathcal{L}\overrightarrow{C}_{x}+\overrightarrow{C}_{y}+O(h))(1+%
\mathcal{L}^{2})}{h(2\overrightarrow{\xi }-\overrightarrow{\mu }_{2}+O(h))}%
\rightarrow \frac{(\mathcal{L}\overrightarrow{C}_{x}+\overrightarrow{C}%
_{y})(1+\mathcal{L}^{2})}{2\overrightarrow{\xi }-\overrightarrow{\mu }_{2}}%
,\quad h\rightarrow 0.
\end{multline*}%
This establishes (\ref{eq:ode-M-estm}). The proof of (\ref%
{eq:ode-M-estm-improved}) is analogous.
\end{proof}


\section{Example of application\label{sec:num-example}}

We consider the following ODE%
\begin{equation}
(x,y)^{\prime }=f_{\varepsilon }\left( x,y,t\right) ,  \label{eq:example-ode}
\end{equation}%
\begin{equation*}
f_{\varepsilon }\left( x,y,t\right) =\left( y-\varepsilon \cos
(t)y^{2},x-x^{2}\right) ,
\end{equation*}%
which is a perturbation of the following Hamiltonian system%
\begin{align*}
q^{\prime }& =J\nabla H, \\
H\left( x,y\right) & =\frac{1}{2}\left( y^{2}-x^{2}\right) +\frac{1}{3}x^{3}.
\end{align*}%
The unperturbed system has a homoclinic orbit to the fixed point $\left(
0,0\right) $, which is depicted in Figure \ref{fig:fish}.

\begin{figure}[tbp]
\centering
\includegraphics[width=6cm]{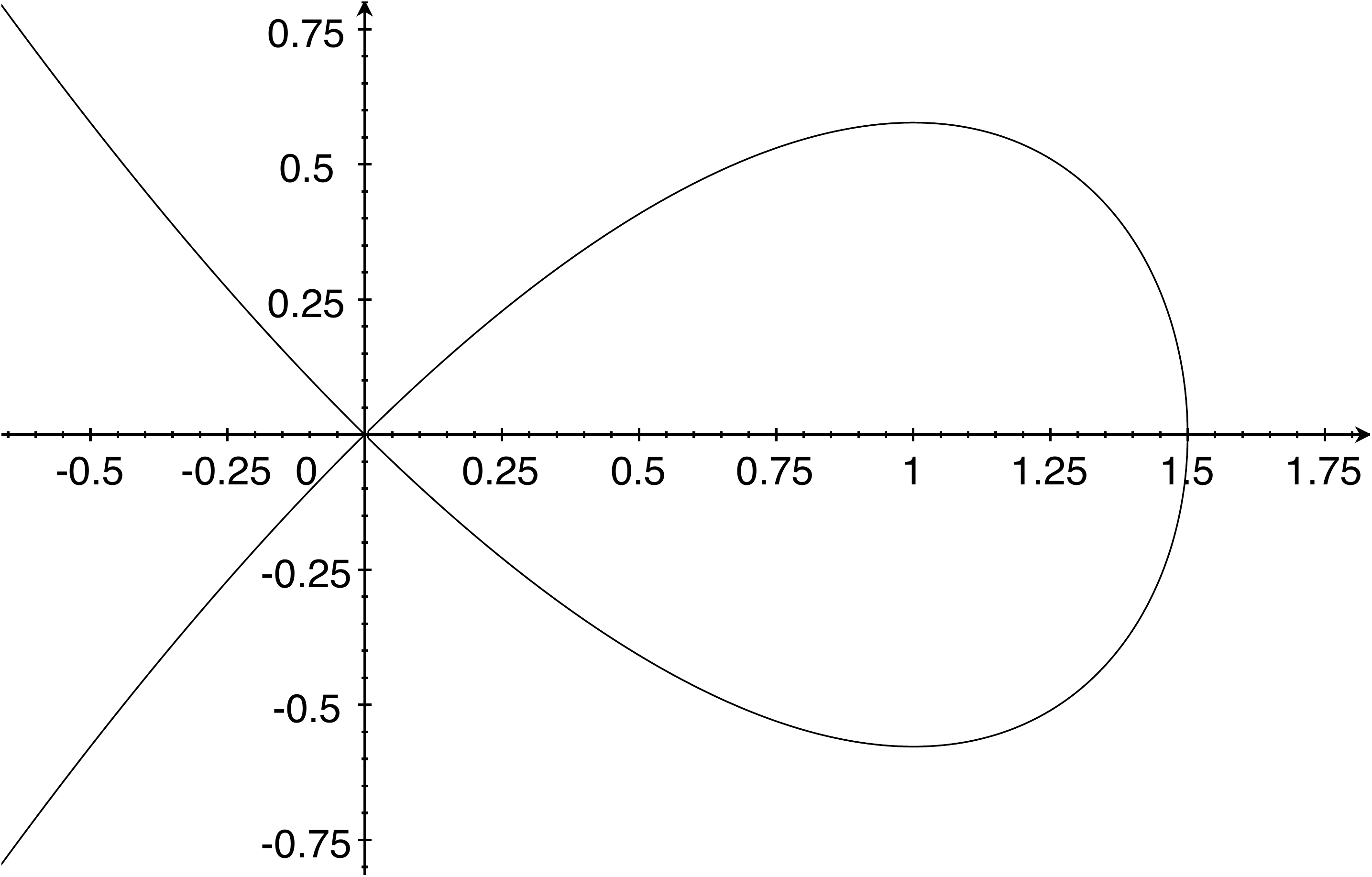}
\caption{The homoclinic orbit for $\protect\varepsilon =0$.}
\label{fig:fish}
\end{figure}

\subsection{Approximating the unstable manifold}

We first consider $\varepsilon =0.$ After a linear change of coordinates%
\begin{equation*}
\left( x,y\right) =C\left( u,v\right) ,
\end{equation*}%
\begin{equation}
C=\left( 
\begin{array}{ll}
1 & -1 \\ 
1 & 1%
\end{array}%
\right) ,\qquad \qquad C^{-1}=\left( 
\begin{array}{cc}
\frac{1}{2} & \frac{1}{2} \\ 
-\frac{1}{2} & \frac{1}{2}%
\end{array}%
\right) ,  \label{eq:linear-e-zero}
\end{equation}%
the ODE becomes%
\begin{equation*}
\left( u,v\right) ^{\prime }=F\left( u,v\right) =\left( 
\begin{array}{c}
u-\frac{1}{2}\left( u-v\right) ^{2} \\ 
-v-\frac{1}{2}\left( u-v\right) ^{2}%
\end{array}%
\right) .
\end{equation*}

Below we quickly describe how the unstable manifold can be approximated
using the parametrization method (for a detailed overview of the method see 
\cite{param-method}). We look for a function $K:\left( -r,r\right) \mathbb{%
\rightarrow R}^{2}$ and $R:\mathbb{R\rightarrow R}$ so that%
\begin{equation}
F\circ K\left( \xi \right) =DK\left( \xi \right) R\left( \xi \right) .
\label{eq:cohomology-eq}
\end{equation}%
The Taylor coefficients can be computed by power matching in the equation (%
\ref{eq:cohomology-eq}). There is a certain freedom regarding the choice of
the coefficients, and we have chosen them so that%
\begin{equation*}
K\left( \xi \right) =\left( \xi ,K_{2}\left( \xi \right) \right) .
\end{equation*}%
For our coordinate change we expand (\ref{eq:cohomology-eq}) only to powers
of three and choose%
\begin{align*}
K_{2}\left( \xi \right) & =-\frac{1}{6}\xi ^{2}-\frac{1}{12}\xi ^{3}, \\
R\left( \xi \right) & =\xi -\frac{1}{2}\xi ^{2}-\frac{1}{6}\xi ^{3}.
\end{align*}%
The set 
\begin{equation}
W^{u}\approx \left\{ \left( \xi ,K_{2}\left( \xi \right) \right) :\xi \in
\left( -r,r\right) \right\} ,  \label{eq:Wu-approx}
\end{equation}%
is an approximation of the unstable manifold.

\subsection{Approximating the stable manifold}

In this section we also consider $\varepsilon =0$. The parametrization of
the stable manifold follows from the reversing symmetry of the system: If we
let $\Phi _{t}$ stand for the flow, and $S(x,y)=\left( x,-y\right) $, then%
\begin{equation*}
\Phi _{t}\left( S(q)\right) =S\left( \Phi _{-t}\left( q\right) \right) .
\end{equation*}%
This means that the stable manifold is parameterized by%
\begin{equation*}
w^{s}\left( \xi \right) =S\left( w^{u}\left( \xi \right) \right) .
\end{equation*}%
For our later consideration, it will be convenient to consider coordinates
in which it the stable manifold is tangent to the $x$-axis. This can be
obtained by taking (see (\ref{eq:linear-e-zero}) for the definition of $C$)%
\begin{equation*}
T=\left( 
\begin{array}{ll}
0 & 1 \\ 
-1 & 0%
\end{array}%
\right) ,\qquad \qquad \bar{C}=CT=\left( 
\begin{array}{cc}
1 & 1 \\ 
-1 & 1%
\end{array}%
\right) 
\end{equation*}%
and computing%
\begin{align}
w^{s}\left( \xi \right) & =SC\left( \xi ,K_{2}\left( \xi \right) \right) 
\label{eq:ws-by-symmetry} \\
& =\bar{C}\left( \xi ,-K_{2}\left( \xi \right) \right) .  \notag
\end{align}%
We can therefore take $\bar{C}$ as the linear change of coordinates, and in
these local coordinates the stable manifold is parameterized by 
\begin{equation}
\xi \rightarrow \left( \xi ,-K_{2}\left( \xi \right) \right) .
\label{eq:ws-parameterization}
\end{equation}

\subsection{Suitable change of coordinates for the unstable manifold}

Now we consider $\varepsilon \geq 0.$ We treat the system (\ref%
{eq:example-ode}) in the coordinates $\left( x,\varepsilon ,t,y\right) $,%
\begin{equation}
\left( x,\varepsilon ,t,y\right) ^{\prime }=f\left( x,\varepsilon
,t,y\right) ,  \label{eq:Fish-ode}
\end{equation}%
with the vector field 
\begin{equation*}
f\left( x,\varepsilon ,t,y\right) =\left( y-\varepsilon \cos
(t)y^{2},0,1,x-x^{2}\right) .
\end{equation*}

Observe that for each $\varepsilon \geq 0$ we have the periodic orbit%
\begin{equation*}
\Lambda _{\varepsilon }=\left\{ \left( 0,\varepsilon ,t,0\right) :t\in 
\mathbb{S}^{1}\right\} .
\end{equation*}

We go through the following change of coordinates%
\begin{equation}
\left( x,\varepsilon ,t,y\right) =\tilde{C}_{u}\,\psi _{u}\left( \bar{x}%
,\varepsilon ,t,\bar{y}\right) ,  \label{eq:unstable-coord-change}
\end{equation}%
where $\tilde{C}$ is a linear change, motivated by (\ref{eq:linear-e-zero}),%
\begin{equation*}
\tilde{C}_{u}=\left( 
\begin{array}{llll}
1 & 0 & 0 & -1 \\ 
0 & 1 & 0 & 0 \\ 
0 & 0 & 1 & 0 \\ 
1 & 0 & 0 & 1%
\end{array}%
\right) ,\qquad \qquad \tilde{C}_{u}^{-1}=\left( 
\begin{array}{llll}
\frac{1}{2} & 0 & 0 & \frac{1}{2} \\ 
0 & 1 & 0 & 0 \\ 
0 & 0 & 1 & 0 \\ 
-\frac{1}{2} & 0 & 0 & \frac{1}{2}%
\end{array}%
\right) ,
\end{equation*}%
and $\psi _{u}$ is a nonlinear change motivated by (\ref{eq:Wu-approx}), 
\begin{equation*}
\psi _{u}\left( \bar{x},\varepsilon ,t,y\right) =\left( \bar{x},\varepsilon
,t,\bar{y}+K_{2}\left( \bar{x}\right) \right) .
\end{equation*}%
The $\psi _{u}$ is simple to invert%
\begin{equation*}
\psi _{u}^{-1}\left( a,\varepsilon ,t,b\right) =\left( a,\varepsilon
,t,b-K_{2}\left( a\right) \right) .
\end{equation*}

\begin{remark}
Our change of coordinates is independent of $\varepsilon $. It is motivated
by the approximation of the manifold for $\varepsilon =0$, which is also a
good approximation for small $\varepsilon $. For our method to work, the
coordinates do not need to be perfectly aligned with the dynamics. An
approximate alignment is sufficient.
\end{remark}

It is a simple task (though slightly laborious) to derive the formula for
the vector field in the local coordinates $\left( \bar{x},\varepsilon ,t,%
\bar{y}\right) $ 
\begin{equation}
\tilde{f}\left( \bar{x},\varepsilon ,t,\bar{y}\right) =\left( \tilde{f}%
_{1}\left( \bar{x},\varepsilon ,t,\bar{y}\right) ,0,1,-K_{2}^{\prime }\left( 
\bar{x}\right) \tilde{f}_{1}\left( \bar{x},\varepsilon ,t,\bar{y}\right) +%
\tilde{h}\left( \bar{x},\varepsilon ,t,\bar{y}\right) \right) ,
\label{eq:fish-local-forw}
\end{equation}%
where%
\begin{align*}
\tilde{f}_{1}\left( \bar{x},\varepsilon ,t,\bar{y}\right) & =\bar{x}-\frac{1%
}{2}\varepsilon \left( \cos t\right) \left( \bar{x}+\bar{y}+K_{2}\left( \bar{%
x}\right) \right) ^{2}-\frac{1}{2}\left( \bar{x}-\bar{y}-K_{2}\left( \bar{x}%
\right) \right) ^{2}, \\
\tilde{h}\left( \bar{x},\varepsilon ,t,\bar{y}\right) & =-\bar{y}%
-K_{2}\left( \bar{x}\right) +\frac{1}{2}\varepsilon \left( \cos t\right)
\left( \bar{x}+\bar{y}+K_{2}\left( \bar{x}\right) \right) ^{2}-\frac{1}{2}%
\left( \bar{x}-\bar{y}-K_{2}\left( \bar{x}\right) \right) ^{2}.
\end{align*}

\subsection{Suitable change of coordinates for the stable manifold}

The stable manifold of (\ref{eq:Fish-ode}) coincides with the unstable
manifold of an ODE with reversed sign:%
\begin{equation}
\left( x,\varepsilon ,t,y\right) ^{\prime }=-f\left( x,\varepsilon
,t,y\right) .  \label{eq:Fish-ode-reversed}
\end{equation}

\begin{remark}
We consider the formulation with reversed sign vector field, since in all
our discussions we have talked about unstable manifolds. This way we can use
our methods directly. The bounds on the unstable manifold of (\ref%
{eq:Fish-ode-reversed}) will be the bounds for the stable manifold for (\ref%
{eq:Fish-ode}).
\end{remark}

We consider the change of coordinates%
\begin{equation*}
\left( x,\varepsilon ,t,y\right) =\tilde{C}_{s}\,\psi _{s}\left( \bar{x}%
,\varepsilon ,t,\bar{y}\right) ,
\end{equation*}%
with $\tilde{C}_{s}$ motivated by (\ref{eq:ws-by-symmetry}), 
\begin{equation*}
\tilde{C}_{s}=\left( 
\begin{array}{llll}
1 & 0 & 0 & 1 \\ 
0 & 1 & 0 & 0 \\ 
0 & 0 & 1 & 0 \\ 
-1 & 0 & 0 & 1%
\end{array}%
\right) \qquad \tilde{C}_{s}=\left( 
\begin{array}{cccc}
\frac{1}{2} & 0 & 0 & -\frac{1}{2} \\ 
0 & 1 & 0 & 0 \\ 
0 & 0 & 1 & 0 \\ 
\frac{1}{2} & 0 & 0 & \frac{1}{2}%
\end{array}%
\right) ,
\end{equation*}%
and $\psi _{s}$ motivated by (\ref{eq:ws-parameterization}), 
\begin{align*}
\psi _{s}\left( \bar{x},\varepsilon ,t,y\right) & =\left( \bar{x}%
,\varepsilon ,t,\bar{y}-K_{2}\left( \bar{x}\right) \right) , \\
\psi _{s}^{-1}\left( \bar{x},\varepsilon ,t,y\right) & =\left( \bar{x}%
,\varepsilon ,t,\bar{y}+K_{2}\left( \bar{x}\right) \right) .
\end{align*}%
The vector field (\ref{eq:Fish-ode-reversed}) rewritten in these local
coordinates is%
\begin{equation}
\hat{f}\left( \bar{x},\varepsilon ,t,\bar{y}\right) =\left( \hat{f}%
_{1}\left( \bar{x},\varepsilon ,t,\bar{y}\right) ,0,-1,K_{2}^{\prime }(\bar{x%
})\hat{f}_{1}\left( \bar{x},\varepsilon ,t,\bar{y}\right) +\hat{h}\left( 
\bar{x},\varepsilon ,t,\bar{y}\right) \right) ,
\label{eq:fish-local-reversed}
\end{equation}%
with%
\begin{align*}
\hat{f}_{1}\left( \bar{x},\varepsilon ,t,\bar{y}\right) & =\bar{x}+\frac{1}{2%
}\varepsilon \left( \cos t\right) \left( -\bar{x}+\bar{y}-K_{2}\left( \bar{x}%
\right) \right) ^{2}-\frac{1}{2}\left( \bar{x}+\bar{y}-K_{2}\left( \bar{x}%
\right) \right) ^{2}, \\
\hat{h}\left( \bar{x},\varepsilon ,t,\bar{y}\right) & =-\bar{y}+K_{2}\left( 
\bar{x}\right) +\frac{1}{2}\varepsilon \left( \cos t\right) \left( -\bar{x}+%
\bar{y}-K_{2}\left( \bar{x}\right) \right) ^{2}+\frac{1}{2}\left( \bar{x}+%
\bar{y}-K_{2}\left( \bar{x}\right) \right) ^{2}.
\end{align*}

\subsection{Bounds on the unstable and stable manifolds}

In our computer assisted proof, we have used the vector field $\tilde{f}$
(see (\ref{eq:fish-local-forw})), to establish the existence and bound for $%
W^{u}\left( \Lambda _{\varepsilon }\right) $ inside of the set%
\begin{equation*}
D=\left[ -r,r\right] \times E\times \mathbb{S}^{1}\times \left[ -rL(E),rL(E)%
\right] ,
\end{equation*}%
for%
\begin{equation*}
r=2\cdot 10^{-4},
\end{equation*}%
and for various parameter intervals $E$. The size on the set depends through 
$L(E)$ on the range $E$ of the parameter $\varepsilon $ considered. For $E=%
\left[ 0,10^{-3}\right] $, which is the first parameter interval we
consider, we obtain that Theorem \ref{th:main-ode} can be applied with
constants $\overrightarrow{\mu _{s,1}}$, $\overrightarrow{\mu _{s,2}}$, $%
\overrightarrow{\xi _{u,1}}$, $\overrightarrow{\xi _{u,1,P}}$, $%
\overrightarrow{\mu _{cs,1}}$, $\overrightarrow{\mu _{cs,2}}$, $%
\overrightarrow{\xi _{cu,1}}$ and $\overrightarrow{\xi _{cu,1,P}}$ with the
following choice of the constant $L$: 
\begin{equation*}
L=L(E)=L\left( \left[ 0,10^{-3}\right] \right) =6.278276608\cdot 10^{-6}.
\end{equation*}%
In our code, the $L(E)$ is chosen automatically by the program to be as
small as possible to establish sharp bounds on the derivatives of $w^{cu}.$

From Theorem \ref{th:main-ode} we know that the function $w^{cu}$ is
Lipschitz with constant $L(E)$. Thus, 
\begin{equation*}
\frac{\partial w^{cu}}{\partial \bar{x}}\left( D\right) ,\frac{\partial
w^{cu}}{\partial \varepsilon }\left( D\right) ,\frac{\partial w^{cu}}{%
\partial t}\left( D\right) \in L\left( E\right) \cdot \lbrack -1,1].
\end{equation*}%
Bounds on the second derivatives also depend on the choice of $E.$ They can
be established using Theorem \ref{th:ode-unstable-cones}. For example, for $%
E=\left[ 0,10^{-3}\right] $, we obtained 
\begin{equation*}
M=M(E)=1.\,\allowbreak 127\,1\times 10^{-3}.
\end{equation*}%
Thus, for $\varepsilon \in \left[ 0,10^{-3}\right] ,$ 
\begin{equation*}
\frac{\partial ^{2}w^{cu}}{\partial v\partial w}\left( D\right) \in 2M\cdot
\lbrack -1,1]\qquad \text{for }v,w\in \{\bar{x},\varepsilon ,t\}.
\end{equation*}

The bounds can then be transported through the change of coordinates (\ref%
{eq:unstable-coord-change}). This is done automatically by the CAPD library,
which has an implementation of rigorous manipulation on jets.

Similar bounds can be obtained for the stable manifold by considering the
vector field $\hat{f}$ (with reversed time) given in (\ref%
{eq:fish-local-reversed}). The bounds on the slope of the stable manifold
and on the second derivatives are indistinguishable from those of the
unstable manifold, up to the accuracy which we have used above to display
results.

\subsection{The transversal intersections of manifolds}

We recall that by (\ref{eq:delta-def}) 
\begin{equation*}
\delta \left( \varepsilon ,\tau \right) :=\pi _{x}p^{u}\left( \varepsilon
,\tau \right) -\pi _{x}p^{s}\left( \varepsilon ,\tau \right) ,
\end{equation*}%
where $p^{u}$ and $p^{s}$ are defined in (\ref{eq:p-e-tau-def}).

We first consider $\varepsilon \in \left[ 0,10^{-3}\right] $. In the left
hand side of Figure \ref{fig:DeltaE} we give a plot of a computer assisted
bound for 
\begin{equation*}
\tau \rightarrow \pi _{x}\frac{\partial p^{u}}{\partial \varepsilon }\left(
\varepsilon ,\tau \right) \qquad \text{and}\qquad \tau \rightarrow \pi _{x}%
\frac{\partial p^{s}}{\partial \varepsilon }\left( \varepsilon ,\tau \right)
.
\end{equation*}%
For $\tau $ close to $4.6$, for all $\varepsilon \in \left[ 0,10^{-3}\right] 
$ we have 
\begin{equation*}
\pi _{x}\frac{\partial p^{u}}{\partial \varepsilon }\left( \varepsilon ,\tau
\right) >\pi _{x}\frac{\partial p^{s}}{\partial \varepsilon }\left(
\varepsilon ,\tau \right) ,
\end{equation*}%
hence for these $\tau $%
\begin{equation*}
\frac{d}{d\varepsilon }\delta \left( \varepsilon ,\tau \right) >0.
\end{equation*}%
Analogously, for the $\tau $ close to $4.8$ we have $\frac{d}{d\varepsilon }%
\delta \left( \varepsilon ,\tau \right) <0.$ The right hand side of Figure %
\ref{fig:DeltaE} contains the plots of 
\begin{equation*}
\tau \rightarrow \pi _{x}\frac{\partial ^{2}p^{u}}{\partial \tau \partial
\varepsilon }\left( \varepsilon ,\tau \right) \qquad \text{and}\qquad \tau
\rightarrow \pi _{x}\frac{\partial ^{2}p^{s}}{\partial \tau \partial
\varepsilon }\left( \varepsilon ,\tau \right) .
\end{equation*}%
For all $\varepsilon \in \left[ 0,10^{-3}\right] $ and the considered range
of $\tau $ we have 
\begin{equation*}
\pi _{x}\frac{\partial ^{2}p^{u}}{\partial \tau \partial \varepsilon }\left(
\varepsilon ,\tau \right) <0\qquad \text{and\qquad }\pi _{x}\frac{\partial
^{2}p^{s}}{\partial \tau \partial \varepsilon }\left( \varepsilon ,\tau
\right) >0,
\end{equation*}%
hence 
\begin{equation*}
\frac{d^{2}}{d\tau d\varepsilon }\delta \left( \varepsilon ,\tau \right) <0.
\end{equation*}

This way, by using Theorem \ref{th:main}, we obtain a proof of the
transversal intersections of $W^{u}\left( \Lambda _{\varepsilon }\right) $
with $W^{s}\left( \Lambda _{\varepsilon }\right) $ for $\varepsilon \in
(0,10^{-3}].$ The computations needed for this result took under 3 seconds
on a single 3GHz Intel i7 core processor.

\begin{figure}[tbp]
\centering
\includegraphics[width=6cm]{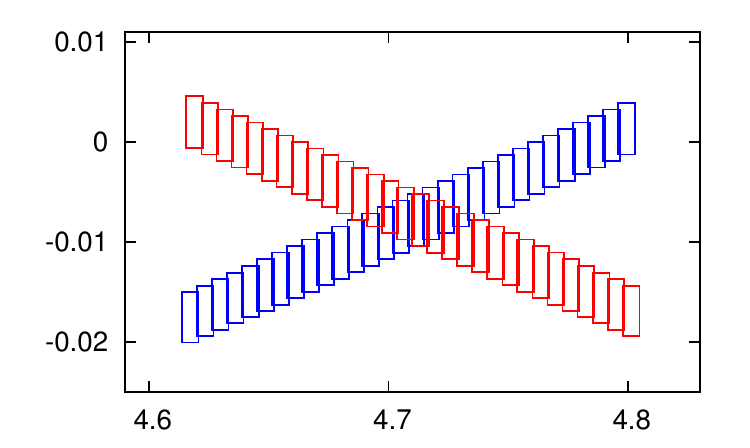} %
\includegraphics[width=6cm]{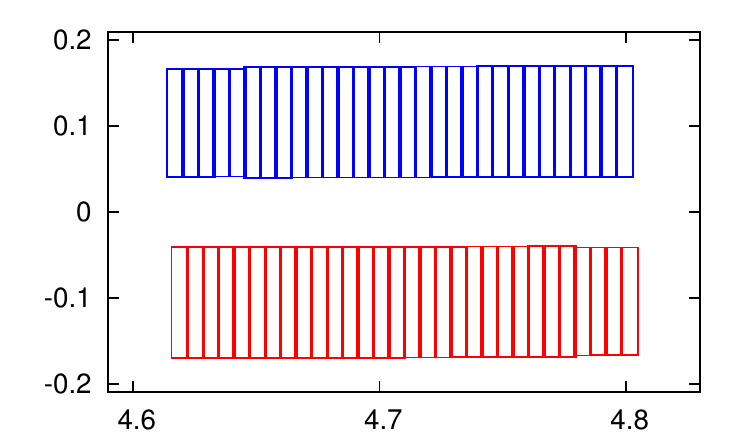}
\caption{Left: The plot of the bounds on $\protect\tau \rightarrow \protect%
\pi _{x}\frac{\partial p^{u}}{\partial \protect\varepsilon }\left( \protect%
\varepsilon ,\protect\tau \right) $ in red and $\protect\tau \rightarrow 
\protect\pi _{x}\frac{\partial p^{s}}{\partial \protect\varepsilon }\left( 
\protect\varepsilon ,\protect\tau \right) $ in blue. Right: The plot of $%
\protect\tau \rightarrow \protect\pi _{x}\frac{\partial ^{2}p^{u}}{\partial 
\protect\tau \partial \protect\varepsilon }\left( \protect\varepsilon ,%
\protect\tau \right) $ in red and $\protect\tau \rightarrow \protect\pi _{x}%
\frac{\partial ^{2}p^{s}}{\partial \protect\tau \partial \protect\varepsilon 
}\left( \protect\varepsilon ,\protect\tau \right) $ in blue. In both plots
we have $\protect\tau $ on the $x$-axis. The bounds are for $\protect%
\varepsilon \in \left[ 0,10^{-3}\right] $.}
\label{fig:DeltaE}
\end{figure}

It turns out that the perturbation $\varepsilon =10^{-3}$ is relatively
\textquotedblleft large". From such parameter we can directly observe,
through rigorous numerics, that $W^{u}\left( \Lambda _{\varepsilon }\right) $
and $W^{s}\left( \Lambda _{\varepsilon }\right) $ intersect transversally.
This can be seen by directly plotting bounds on $\tau \rightarrow \pi
_{x}p^{u}\left( \varepsilon ,\tau \right) $ and $\tau \rightarrow \pi
_{x}p^{s}\left( \varepsilon ,\tau \right) $. Such bounds, for $\varepsilon
\in \left[ 10^{-3},10^{-3}+10^{-4}\right] $, are given in the left hand side
plot of Figure \ref{fig:continuation1}. This way we establish that $%
W^{u}\left( \Lambda _{\varepsilon }\right) $ and $W^{s}\left( \Lambda
_{\varepsilon }\right) $ intersect (see the left hand side plot in Figure %
\ref{fig:continuation1}). To show that this intersection is transversal we
consider bounds on (right plot in Figure \ref{fig:continuation1}) 
\begin{equation*}
\tau \rightarrow \pi _{x}\frac{\partial p^{u}}{\partial \tau }\left(
\varepsilon ,\tau \right) \qquad \text{and}\qquad \tau \rightarrow \pi _{x}%
\frac{\partial p^{s}}{\partial \tau }\left( \varepsilon ,\tau \right) .
\end{equation*}
These bounds establish that for the investigated range of $\tau $, the
function $\tau \rightarrow \delta \left( \varepsilon ,\tau \right) $ is
strictly decreasing. Thus, the intersection between $W^{u}\left( \Lambda
_{\varepsilon }\right) $ and $W^{s}\left( \Lambda _{\varepsilon }\right) $
is transversal. 
\begin{figure}[tbp]
\centering\includegraphics[width=6cm]{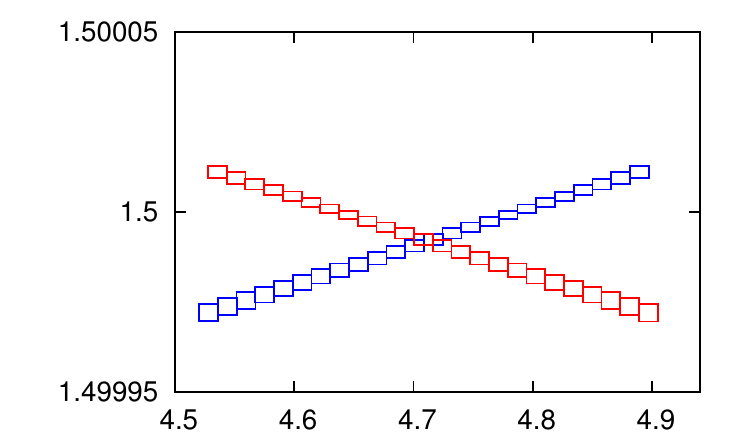} %
\includegraphics[width=6cm]{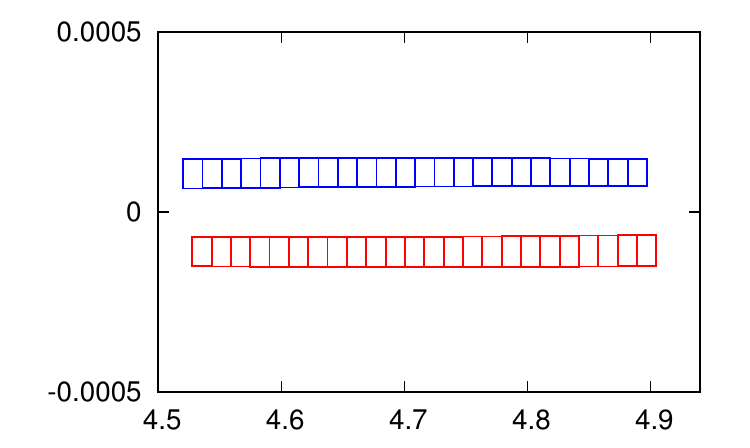}
\caption{Left: The plot of $\protect\tau \rightarrow \protect\pi %
_{x}p^{u}\left( \protect\varepsilon ,\protect\tau \right) $ in red and $%
\protect\tau \rightarrow \protect\pi _{x}p^{s}\left( \protect\varepsilon ,%
\protect\tau \right) $ in blue. Right: The plot of $\protect\tau \rightarrow 
\frac{\partial p^{u}}{\partial \protect\tau }\left( \protect\varepsilon ,%
\protect\tau \right) $ in red and $\protect\tau \rightarrow \frac{\partial
p^{s}}{\partial \protect\tau }\left( \protect\varepsilon ,\protect\tau %
\right) $ in blue. We have $\protect\tau $ on the $x$-axis. The bounds are
for $\protect\varepsilon \in \left[ 10^{-3},10^{-3}+10^{-4}\right] $.}
\label{fig:continuation1}
\end{figure}

This procedure can be continued by considering other interval parameters. We
have investigated the range $\left[ 10^{-3},10^{-2}\right] $, by dicing it
into $90$ intervals of length $10^{-4}$. The results are given in Figure \ref%
{fig:continuation2}, where we have highlighted the bounds for $\varepsilon
\in \left[ 10^{-3},10^{-3}+10^{-4}\right] $ in black. Thus, the black part
of Figure \ref{fig:continuation2} corresponds to Figure \ref%
{fig:continuation1} (only in different scale on the vertical coordinate). In
gray we have highlighted the bounds for $\varepsilon \in \left[
10^{-2}-10^{-4},10^{-2}\right] $, which was the last of the $90$ considered
parameter intervals. Each of the $90$ intervals took around half a second on
a single 3GHz Intel i7 core processor. (The computation for the $90$
intervals in total took $54$ seconds on the single core.)

In sum, in this example we have established a computer assisted proof of
transversal intersections of $W^{u}\left( \Lambda _{\varepsilon }\right) $
and $W^{s}\left( \Lambda _{\varepsilon }\right) $ for all $\varepsilon \in
(0,10^{-2}]$. The whole computation time required under one minute on a
single processor. There is no obstacle of course to continue such proof for
larger $\varepsilon $. The subtle part was how to separate from $\varepsilon
=0$. Once relatively far away, one can continue with ease.

\begin{figure}[tbp]
\centering\includegraphics[width=6cm]{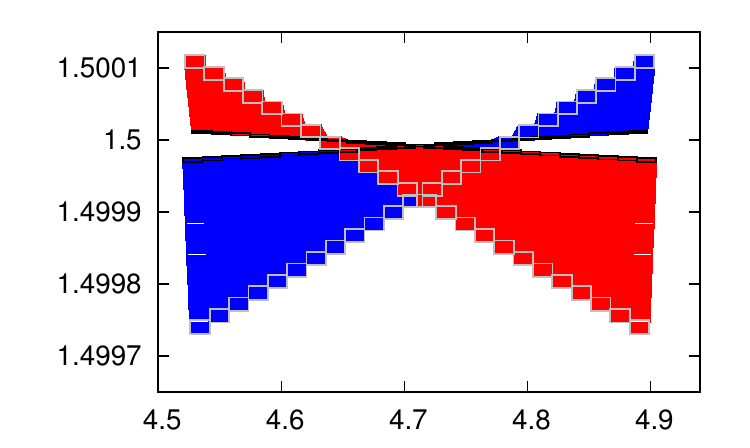} %
\includegraphics[width=6cm]{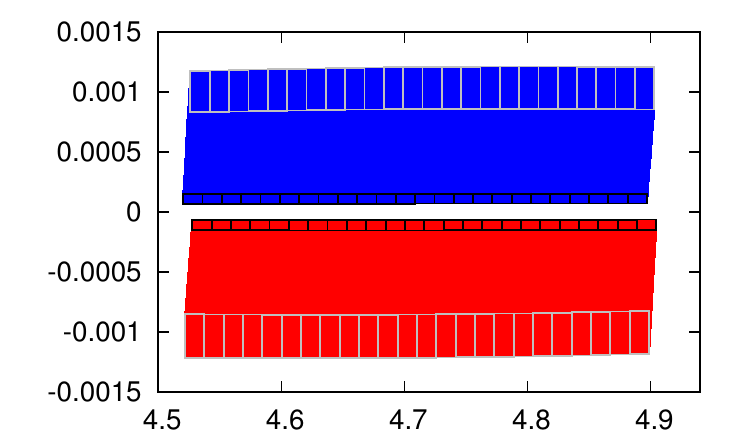}
\caption{ Left: The plot of $\protect\tau \rightarrow \protect\pi %
_{x}p^{u}\left( \protect\varepsilon ,\protect\tau \right) $ in red and $%
\protect\tau \rightarrow \protect\pi _{x}p^{s}\left( \protect\varepsilon ,%
\protect\tau \right) $ in blue. Right: The plot of $\protect\tau \rightarrow 
\protect\pi _{x}\frac{\partial p^{u}}{\partial \protect\tau }\left( \protect%
\varepsilon ,\protect\tau \right) $ in red and $\protect\tau \rightarrow 
\protect\pi _{x}\frac{\partial p^{s}}{\partial \protect\tau }\left( \protect%
\varepsilon ,\protect\tau \right) $ in blue. We have $\protect\tau $ on the $%
x$-axis. The bounds are for $\protect\varepsilon \in \left[ 10^{-3},10^{-2}%
\right] $.}
\label{fig:continuation2}
\end{figure}


\section{Acknowledgements}

We would like to thank Daniel Wilczak for his advice and discussions
concerning higher order derivatives and jet manipulation in the CAPD library.


\section{Appendix}

\label{sec:app}

\subsection{Proof of Lemma \protect\ref{lem:def-ml-ok}}

\label{app:ml}

\begin{proof}
It is known (see for example \cite[Sec. 3]{KZ}) that the limit in the
definition of logarithmic norms exists and the convergence is locally
uniform with respect to $A$. We will reduce our question to this.

We have for $h\in(0,1]$ on compact sets of $A$'s 
\begin{align*}
\frac{m(I+hA)-1}{h} & =\frac{\frac{1}{\Vert(I+hA)^{-1}\Vert}-1}{h} \\
& =\frac{\frac{1}{\Vert I-hA+O(h^{2})\Vert}-1}{h} \\
& \leq\frac{\frac{1}{\Vert I-hA\Vert-O(h^{2})}-1}{h} \\
& =\frac{\Vert I-hA\Vert^{-1}+O(h^{2})-1}{h} \\
& =-\frac{\Vert I-hA\Vert-1}{h}\frac{1}{\Vert I-hA\Vert}+O(h) \\
& \rightarrow-l(-A),\quad h\rightarrow0
\end{align*}

It is known that $l(A)$ is a convex function. Since $l(-A)$ is convex, $%
m_{l}(A)$ is concave.
\end{proof}

\subsection{Proof of Theorem \protect\ref{th:log-norm-ode-lower-bound}}

\label{app:log-norm-ode-lower-bound}

\begin{proof}
Observe that from Lemma~\ref{lem:def-ml-ok} it follows that 
\begin{equation*}
m_{l}(Df,W)=-\sup_{x\in W}l(-Df(x)).
\end{equation*}%
From Theorem~\ref{th:log-norm-ode} applied to $x^{\prime }=-f(x)$ with
initial conditions $x(t)$ and $y(t)$ we obtain 
\begin{equation*}
\Vert x(0)-y(0)\Vert \leq \exp \left( t\sup_{z\in W}l(-Df(z))\right) \Vert
x(t)-y(t)\Vert .
\end{equation*}%
Hence 
\begin{equation*}
\Vert x(0)-y(0)\Vert \exp \left( -t\sup_{z\in W}l(-Df(z))\right) \leq \Vert
x(t)-y(t)\Vert .
\end{equation*}%
Since $m(A)=-l(-A)$, our claim follows from the above.
\end{proof}

\subsection{Proof of Lemma \protect\ref{lem:lognorm-conv}}

\label{sec:app-lognorm}

\begin{proof}
We have 
\begin{eqnarray*}
((I+hA)x|(I+hA)x) &=&((I+hA)^{\top }(I+hA)x|x) \\
&=&\left( \left( I+h(A+A^{\top })\right) x|x\right) +h^{2}(A^{\top }Ax|x) \\
&=&\left( \left( I+h(A+A^{\top })\right) x|x\right) +O(h^{2}\Vert x\Vert
^{2}\Vert A^{\top }A\Vert ).
\end{eqnarray*}%
Therefore, (below we use the fact that $\sqrt{1+x}=1+\frac{1}{2}x+O\left(
x^{2}\right) $) 
\begin{eqnarray*}
\Vert I+hA\Vert &=&\max_{\Vert x\Vert =1}\sqrt{((I+hA)x|(I+hA)x))} \\
&=&\max_{\Vert x\Vert =1}\sqrt{\left( \left( I+h(A+A^{\top })\right)
x|x\right) +O(h^{2}\Vert A^{\top }A\Vert )} \\
&=&\max_{\Vert x\Vert =1}\sqrt{1+h((A+A^{\top })x|x)+O(h^{2}\Vert A^{\top
}A\Vert )} \\
&=&1+\frac{h}{2}\max_{\Vert x\Vert =1}((A+A^{\top })x|x)+O(h^{2}),
\end{eqnarray*}%
where $O(h^{2})$ is uniform with respect to $A\in W$. Hence by (\ref%
{eq:eucl-log-norm}) 
\begin{eqnarray*}
\Vert I+hA\Vert &=&1+h\max \{\lambda \in \text{Spectrum}((A+A^{\top
})/2)\}+O(h^{2}) \\
&=&1+hl(A)+O(h^{2}),
\end{eqnarray*}%
which concludes the proof.
\end{proof}

\subsection{Proof of Lemma \protect\ref{lem:ml-conv}}

\label{sec:app-ml-norm}

\begin{proof}
The proof follows from Lemma~\ref{lem:lognorm-conv}. All below estimates are
clearly uniform over a compact set $W$ and $h\in \lbrack 0,h_{0}]$, for $%
h_{0}$ which is sufficiently small $h_{0}=h_{0}(W)$.

We have (applying Lemma \ref{lem:lognorm-conv} in the 5th and (\ref%
{eq:ml-min-l}) in the last line) 
\begin{eqnarray*}
m(I+hA) &=&\frac{1}{\Vert (I+hA)^{-1}\Vert } \\
&=&\frac{1}{\Vert I-hA+O(h^{2})\Vert } \\
&=&\frac{1}{\Vert I-hA\Vert +O(h^{2})} \\
&=&\frac{1}{\Vert I-hA\Vert }+O(h^{2}) \\
&=&\frac{1}{1+hl(-A)+O(h^{2})}+O(h^{2}) \\
&=&1-hl(-A)+O(h^{2}) \\
&=&1+hm_{l}(A)+O(h^{2}),
\end{eqnarray*}%
as required.
\end{proof}

\subsection{Solving an implicit function problem in interval arithmetic\label%
{app:implicit-f-sol}}

Consider $f:\mathbb{R}^{k}\times \mathbb{R}^{l}\rightarrow \mathbb{R}^{l}$.
We wish to solve for $y$ satisfying 
\begin{equation*}
f(x,y\left( x\right) )=0.
\end{equation*}%
Consider $x\in \mathbb{R}^{k},y_{0}\in \mathbb{R}^{l},$ and a cube $%
Y=\prod_{i=1}^{l}\left[ a_{i},b_{i}\right] \subset \mathbb{R}^{l}$ and define%
\begin{equation*}
N\left( x,y_{0},Y\right) :=y_{0}-\left[ \frac{\partial f}{\partial y}\left(
x,Y\right) \right] ^{-1}f\left( x,y_{0}\right) .
\end{equation*}%
If $N\left( x,y_{0},Y\right) \subset Y$, then by the interval Newton method $%
y(x)\in Y.$ In practice, we can consider a cube $X\subset \mathbb{R}^{k},$
verify that $N\left( X,y_{0},Y\right) \subset Y,$ obtaining that $y(x)\in Y$
for all $x\in Y$. The method can be further refined by appropriate choices
of coordinates to improve the estimates (see for instance \cite[section 4.1]%
{Jay}).


\begin{thebibliography}{00}

\bibitem{Arnold-diff64} V.I. Arnold, \emph{Instability of dynamical systems with several degrees of freedom}, Sov. Math. Doklady, 5, 342–355, 1964.

\bibitem{param-method} X. Cabr\'e, E. Fontich, R. de la Llave,  \emph{The parameterization method for invariant manifolds III: overview and applications}, J. Diff. Eq., 218 (2005) 444--515

\bibitem{CF} R. Calleja, J-L. Figueras, \emph{Collision of invariant bundles of quasi-periodic attractors in the dissipative standard map}. Chaos 22 (2012), no. 3, 033114, 10 pp. 37--99



\bibitem{Jay} M.J. Capinski, J.D. Mireles James, \emph{Validated computation of heteroclinic sets}, http://arxiv.org/abs/1602.02973

\bibitem{conecond} M.J. Capinski, P. Zgliczynski, \emph{Cone conditions and covering relations for topologically normally hyperbolic manifolds}, Discrete Contin. Dyn. Syst. 30 (2011) 641--670.

\bibitem{Geom} M.J. Capinski, P. Zgliczynski, \emph{Geometric proof for normally hyperbolic invariant manifolds},  J. Diff. Eq., 259(2015) 6215--6286

\bibitem{CLJ} R. Castelli, J-P. Lessard, J.D. Mireles James, \emph{Parameterization of invariant manifolds for periodic orbits I: Efficient numerics via the Floquet normal form}. SIAM J. Appl. Dyn. Syst. 14 (2015), no. 1, 132--167.

\bibitem{Chow} S.N. Chow, J.K. Hale, J. Mallet-Paret   \emph{An example of bifurcation to homoclinic orbits}. J. Diff. Eqns. 37 (1980), 351--373.

\bibitem{D} G. Dahlquist, \emph{Stability and Error Bounds in the Numerical Intgration of Ordinary Differential Equations}, Almqvist \& Wiksells, Uppsala, 1958; Transactions of the Royal Institute of Technology, Stockholm, 1959.

\bibitem{DRR} A. Delshams, R. Ramirez-Ros, \emph{Melnikov Potential for Exact Symplectic Maps}, Commun. Math. Phys. 190, 213 – 245 (1997)



\bibitem {DLS1}A. Delshams, R. de~la Llave, T.~M. Seara. \emph{Geometric properties of the scattering map of a normally hyperbolic invariant manifold.} Adv. Math., 217(3):1096--1153, 2008.

\bibitem{F1} N. Fenichel, \emph{Asymptotic stability with rate conditions for dynamical systems}, Bull. Amer. Math. Soc. 80 (1974) 346.

\bibitem{F2} N. Fenichel,\emph{ Asymptotic stability with rate conditions.} II, Indiana Univ. Math. J. 26 (1) (1977) 81--93.

\bibitem{FH} J-L. Figueras, A. Haro, \emph{Reliable computation of robust response tori on the verge of breakdown}. SIAM J. Appl. Dyn. Syst. 11 (2012), no. 2, 597--628.

\bibitem{GH} J. Guckenheimer and P. Holmes. Nonlinear Oscillations, Dynamical Systems and Bifurcations of Vector Fields. Springer, New York, 1983.

\bibitem{HNW} E.\ Hairer, S.P.\ N\o rsett and G.\ Wanner, {\em Solving Ordinary Differential Equations I, Nonstiff Problems}, Springer-Verlag, Berlin Heidelberg 1987.

\bibitem {HM}P. J. Holmes, J. E. Marsden. \textit{Melnikov's method and Arnold diffusion for perturbations of integrable Hamiltonian systems}. J. Math. Phys. 23 (1982), no. 4, 669--675.

\bibitem{KZ}T. Kapela and P. Zgliczy\'nski, \textit{A Lohner-type algorithm for control systems and ordinary differential inclusions}, Discrete Cont. Dyn. Sys. B, vol. 11(2009), 365-385.

\bibitem{LJR} J-P. Lessard, J.D. Mireles James, C. Reinhardt, \emph{Computer assisted proof of transverse saddle-to-saddle connecting orbits for first order vector fields}. J. Dynam. Differential Equations 26 (2014), no. 2, 267--313.

\bibitem{L}  S. M. Lozinskii, Error esitimates for the numerical integration of ordinary differential equations, part I,  {\em Izv. Vyss. Uceb. Zaved. Matematica},6 (1958), 52--90 (Russian)

\bibitem{M} Melnikov V.K   {\em On the stability of the center for time periodic perturbations.} Trans. Moscow Math. Soc. 12 (1963), 1--57.

\bibitem{JM} J. D. Mireles James, K. Mischaikow, \emph{Rigorous a posteriori computation of (un)stable manifolds and connecting orbits for analytic maps. }SIAM J. Appl. Dyn. Syst. 12 (2013), no. 2, 957--1006.

\bibitem{MM} K. Mischaikow, M. Mrozek, \emph{Conley index}. Handbook of dynamical systems, Vol. 2, 393--460, North-Holland, Amsterdam, 2002.

\bibitem{Poincare1890} H.J. Poincaré. Sur le probleme des trois corps et les équations de la dynamique. Acta Mathematica, 13, 1–270, 1890.

\bibitem{Wi} S. Wiggins. \newblock {\em Normally hyperbolic invariant manifolds in dynamical systems}, volume 105 of \emph{Applied Mathematical Sciences}. \newblock Springer-Verlag, New York, 1994.

\bibitem{WZ} D. Wilczak and P. Zgliczy\'nski, \emph{Cr-Lohner algorithm}, Schedae Informaticae, 20 (2011), pp. 9--46.

 \bibitem{Z} P. Zgliczy\'nski, \emph{C1 Lohner algorithm}. Found. Comput. Math. 2 (2002), no. 4, 429--465.

\end{thebibliography}
\end{document}